\documentclass[english,a4paper,10pt,final,usenames, dvipsnames]{amsart}%final,draft,openbib
\usepackage[pagewise]{lineno}
\usepackage{ulem,tikz}
\usetikzlibrary{patterns}
\usepackage[english]{babel}% \usepackage[dvips]{graphicx}
\usepackage{amsmath,amssymb,amscd,amsxtra,amsgen,enumitem, amsthm}
\usepackage[?,initials]{amsrefs} 

\usepackage[all]{xy}
\usepackage{microtype} \CompileMatrices
\usepackage{hyperref}
\newcommand{\nc}{\newcommand}
\renewcommand{\frak}{\mathfrak}
\providecommand{\cal}{\mathcal}

\renewcommand{\bold}{\mathbf}
\numberwithin{equation}{section}

\newcommand{\pfname}{Proof.}

                      {$\square$\vskip\medskipamount\par}

\newenvironment{pfof}[1]{\vskip-\lastskip\vskip\medskipamount{\it
    Proof of #1.}}%
                      {$\square$\vskip\medskipamount\par}

\newtheorem{thm}{Theorem}[section]
\newtheorem{theorem}[thm]{Theorem}%[subsection]

\newtheorem{thmnn}{Theorem}

\newtheorem{prop}[thm]{Proposition}%[subsection]
\newtheorem{proposition}[thm]{Proposition}%[subsection]

\newtheorem{lemma}[thm]{Lemma} %[subsection]
\theoremstyle{definition}
\newtheorem{definition}[thm]{Definition}%[subsection]

\theoremstyle{definition}

\newtheorem{remark}[thm]{Remark}%[subsection]
\newtheorem{example}[thm]{Example} %[subsection]
\newtheorem{question}[thm]{Question}%[subsection]

\nc{\Theorem}[1]{Theorem~{#1}}
\nc{\Th}[1]{({\sl Th.}~#1)}
\nc{\Thd}[2]{({\sl Th.}~{#1} {#2})}
\nc{\Theorems}[2]{Theorems~{#1} and ~{#2}}
\nc{\Thms}[2]{({\it Thms. ~{#1} and ~{#2}})}
\nc{\Lemmas}[2]{Lemma~{#1} and ~{#2}}

\nc{\manga}[6]{({\it Thms. ~ #1, ~ #2, ~ #3,\\ ~ #4, ~ #5, ~ #6})}

\nc{\Prop}[1]{({\sl Prop.}~{#1})}
\nc{\Proposition}[1]{Proposition~{#1}}
\nc{\Propositions}[2]{Propositions~{#1} and ~{#2}}
\nc{\Props}[2]{({\sl Props.}~{#1} and ~{#2})}
\nc{\Cor}[1]{({\sl Cor.}~{#1})}
\nc{\Corollary}[1]{Corollary~{#1}}
\nc{\Corollaries}[2]{Corollaries~{#1} and ~{#2}}
\nc{\Definition}[1]{Definition~{#1}}
\nc{\Defn}[1]{({\sl Def.}~{#1})}
\nc{\Lemma}[1]{Lemma~{#1}} 
\nc{\Lem}[1]{({\sl Lem.} ~{#1})} 
\nc{\Eq}[1]{equation~({#1})}
\nc{\Equation}[1]{Equation~({#1})}
\nc{\Section}[1]{Section~{#1}}
\nc{\Sections}[1]{Sections~{#1}}
\nc{\Sec}[1]{({\sl Sec.} ~{#1})} 
\nc{\Chapter}[1]{Chapter~{#1}}
\nc{\Chapt}[1]{({\sl Ch.}~{#1})}

\nc{\Ex}[1]{{\sl Ex.}~{#1}}
\nc{\Exa}[1]{{\sl Example}~{#1}}
\nc{\Example}[1]{{\sl Example}~{#1}}
\nc{\Examples}[1]{{\sl Examples}~{#1}}
\nc{\Exercise}[1]{{\sl Exercise}~{#1}}

\nc{\Rem}[1]{({\sl Rem.}~{#1})}
\nc{\Remark}[1]{{\sl Remark}~{#1}}
\nc{\Remarks}[1]{{\sl Remarks}~{#1}}
\nc{\Note}[1]{{\sl Note}~{#1}}

\nc{\Conjecture}[1]{Conjecture~{#1}}

\nc \Proof{{  \it Proof. }}

\nc{\xmu}{\mu}
\nc{\w}{\omega}
\nc{\xv}{\mbox{\boldmath$x$}}
\nc{\uv}{\mbox{\boldmath$u$}}
\nc{\xiv}{\mbox{\boldmath$\xi$}}
\nc{\bbeta}{\mbox{\boldmath$\beta$}}
\nc{\balpha}{\mbox{\boldmath$\alpha$}}
\nc{\bgamma}{\mbox{\boldmath$\gamma$}}
\nc{\bdelta}{\mbox{\boldmath$\delta$}}
\nc{\bepsilon}{\mbox{\boldmath$\epsilon$}}
\nc \Ab{{\ensuremath{\bold A}}}
\nc \ab{{\ensuremath{\bold a}}}
\nc \bb{{\ensuremath{\bold b}}}
\nc \cb{{\ensuremath{\bold c}}}
\nc \Bb{{\ensuremath{\bold B}}}
\nc \Gb{{\ensuremath{\bold G}}}
\nc \Qb{{\ensuremath{\bold Q}}}
\nc \Rb{{\ensuremath{\bold R}}} \nc \Cb{{\ensuremath{\bold C}}} 
\nc \Eb{{\ensuremath{\bold E}}}
\nc \eb{{\ensuremath{\bold e}}}
\nc \Db{{\ensuremath{\bold D}}}
\nc \Fb{{\ensuremath{\bold F}}}
\nc \ib{{\ensuremath{\bold i}}}
\nc \jb{{\ensuremath{\bold j}}}
\nc \kb{{\ensuremath{\bold k}}}
\nc \Kb{{\ensuremath{\bold K}}}
\nc \nb{{\ensuremath{\bold n}}}
\nc \rb{{\ensuremath{\bold r}}}
\nc \Ob{{\ensuremath{\bold O}}}
\nc \Pb{{\ensuremath{\bold P}}}
\nc \pb{{\ensuremath{\bold p}}}
\nc \qb{{\ensuremath{\bold q}}}
\nc \SPb{{\ensuremath{\bold {SP}}}}
\nc \Zb{{\ensuremath{\bold Z}}} 
\nc \zb{{\ensuremath{\bold z}}} 
\nc \gb{{\ensuremath{\bold g}}} 
\nc \fb{{\ensuremath{\bold f}}} 
\nc \ub{{\ensuremath{\bold u}}} 
\nc \vb{{\ensuremath{\bold v}}} 
\nc \yb{{\ensuremath{\bold y}}} 
\nc \xb{{\ensuremath{\bold x}}} 
\nc \Xb{{\ensuremath{\bold X}}} 
\nc \xib{{\ensuremath{\bold \xi}}} 
\nc \Nb{{\ensuremath{\bold N}}} 
\nc \Hb{{\ensuremath{\bold H}}} 
\nc \wb{{\ensuremath{\bold w}}} 
\nc \Wb{{\ensuremath{\bold W}}} 
\nc \syz{{\mathbf {syz}}}
\nc \bnoll{{\ensuremath{\bold 0}}} 
\nc \mf{\frak m} \nc \mh{\hat{\m}} 
\nc \nf{\frak n}
\nc \Of{\frak O}
\nc \of{\frak o}
\nc \rf{\frak r}
\nc \tf{\frak t}
\nc \mufr{{\mathbf \mu}}
\nc \hf{\frak h} 
\nc \qf{\frak q} 
\nc \bfr{\frak b} 
\nc \kfr{\frak k} 
\nc \pfr{\frak p} 
\nc \af{\frak a }
\nc \cf{\frak c }
\nc \sfr{\frak s} 
\nc \ufr{\frak u} 
\nc \g{\frak g} 
\nc \gA{\g_{\Ao}} 
\nc \lfr{\frak l}
\nc \afr{\frak a}
\nc \gfh{\hat {\frak g}}
\nc \gl{\frak { gl }}
\nc \Sl{\frak {sl}}
\nc \SU{\frak {SU}}
\nc{\Homf}{\frak{Hom}}

\newcommand{\on}{\operatorname}
\nc\hankel{\on {Hankel}}
\nc\row{\on {row\ }}
\nc\nullity{\on {nullity }}
\nc\col{\on {col\ }}
\nc\rowm{\on {Row \ }}
\nc\loc{\on {lc \ }}
\nc\nullo{\on {null\ }}
\nc\Nul{\on {Nul\ }}
\nc \Ann {\on {Ann }}
\nc \Ass {\on {Ass \ }}
\nc \Coker {\on {Coker}}
\nc \Co{\on C}
\nc \Diag{\on {Diag}}
\nc \Homo{\on {Hom}}
\nc \Ker {\on {Ker}}
\nc \omod{\on{mod}}
\nc \No {\on N}
\nc \NN {\on {NN}}
\nc \NGo {\on {NG}}
\nc \Oo {\on O}
\nc \ch {\on {ch}}
\nc \rko {\on {rk}}
\nc \Sing {\on {Sing\ }}
\nc \Reg {\on {Reg}}
\nc \CoI {\on {CI}}
\nc \CoM {\on {CM}}
\nc \Gor {\on {Gor}}
\nc \Type {\on {Type}}
\nc \can {\on {can}}
\nc \Top {\on {T}}
\nc \Tr {\on {Tr}}
\nc \rel {\on {rel}}
\nc \tr {\on {tr}}
\nc \sgn {\on {sgn }}
\nc \trdeg {\on {tr.deg}}
\nc \codim {\on {codim }}
\nc \coht {\on {coht}}
\nc \divo {\on {div \ }}
\nc \coh {\on {coh}}
\nc \Clo {\on {Cl}}
\nc \embdim{\on {embdim}}
\nc \ed{\on {ed}}
\nc \embcodim{\on {embcodim  }}
\nc \qcoh {\on {qcoh}}
\nc \grad {\on {grad}\ }
\nc \grade {\on {grade}}
\nc \hto {\on {ht}}
\nc \depth {\on {depth}}
\nc \prof {\on {prof}}
\nc \reso{\on {res}}
\nc \ind{\on {ind}}
\nc \prodo{\on {prod}}
\nc \coind{\on {coind}}
\nc \Con{\on {Con}}
\nc \Crit{\on {Crit}}
\nc \Der{\on {Der}}
\nc \Char{\on {Char}}
\nc \Ch{\on {Ch}}

\nc \Ext{\on {Ext}}
\nc \Eo{\on {E}}
\nc \End{\on {End}}
\nc \ad{\on {ad}}
\nc \Ad{\on {Ad}}
\nc \gr{\on {gr}}
\nc \Fo{\on {F}}
\nc \Gr{\on {Gr}}
\nc \Go{\on {G}}
\nc \GFo{\on {GF}}
\nc \Glo{\on {Gl}}
\nc \PGlo{\on {PGl}}
\nc \Ho{\on {H}}
\nc \CMo{\on {\CM}}
\nc \SCM{\on {SCM}}
\nc \hol{\on {hol}}
\nc{\sgd}{\on{sgd}}
\nc \supp{\on {supp}}
\nc \ssupp{\on {s-supp}}
\nc \singsupp{\on {singsupp}}
\nc \msupp{\on {msupp}}
\nc \spec{\on {spec}}
\nc \spano{\on {span }}
\nc \Span{\on {Span }}
\nc \Max{\on {Max}}
\nc \Mat{\on {Mat}}
\nc \Min{\on {Min}}
\nc \nil{\on {nil}}
\nc \Loc{\on {Loc}}
\nc \Mod{\on {Mod}}
\nc \Rad {\on {Rad}}
\nc \rad {\on {rad}}
\nc \rank {\on {rank}}
\nc \range {\on {range}}
\nc \Slo{\on {SL}}
\nc \soc {\on {soc}}
\nc \Irr {\on {Irr}}
\nc \Reo {\on {Re}}
\nc \Imo {\on {Im}}
\nc \SSo{\on {SS}}
\nc \lub{\on {lub}}
\nc \gldim{\on {gl.d.}}
\nc \length{\on {length}}
\nc \pdo{\on {p.d.}} 
\nc \fdo{\on {f.d.}} 
\nc \ido{\on {i.d.}} 
\nc \dSSo{\dot {\SSo}}
\nc \So{\on S}
\nc \Io{\on I}
\nc \Jo{\on J}
\nc \jo{\on j}
\nc \Ko{\on K}
\nc \PBW{\Ac_{PBW}}
\nc \Ro{\on R}
\nc \To{\on T}
\nc \Ao{\on A}

\nc \Do{{\on D}}
\nc \Bo{\on B}
\nc \Po{\on P}
\nc \Qo{\on Q}
\nc \Zo{\on Z}
\nc \wt{\on {wt}}
\nc \Uh{\hat {\U}}
\nc \T{\on T}
\nc \Lo{\on L}
\nc{\dop}{\on d}
\nc{\eo}{\on e}
\nc{\ado}{\on{ad}}
\nc{\Tot}{\on{Tot}}
\nc{\Aut}{\on{Aut}}
\nc{\sinc}{\on {sinc}}

\nc{\overrightleftarrows}[2]{\overset{#1}{\underset{#2}{\rightleftarrows}}}
\nc{\CCF}{\cal{CF}}
\nc{\CDF}{\cal{DF}}
\nc{\CHC}{\check{\cal C}}

\nc{\Cone}{\on{Cone}}
\nc{\dec}{\on{dec}}
\nc{\Diff}{\on{Diff}}
\nc{\dirlim}{\underset{\to}{\on{lim}}}
\nc{\dpar}{\partial}
\nc{\GL}{\on{GL}}
\nc{\glo}{\on{gl}}
\nc{\CGr}{\cal{G}r}
\nc{\pr}{\on{pr}}
\nc{\semid}{|\!\!\!\times}
\nc{\Hom}{\on{Hom}}
\nc \RHom{\on {RHom}}

\nc \Proj{\mathrm {Proj\ }}
\nc \proj{\mathrm {proj}}
\nc{\Id}{\on{Id}}
\nc{\id}{\on{id}}
\nc{\Ima}{\on{Im}}
\nc{\invtimes}{\underset{\gets}{\otimes}}
\nc{\invlim}{\underset{\gets}{\on{lim}}}
\nc{\Lie}{\on{Lie}}
\nc{\re}{\on{Re }}
\nc{\Pic}{\on{Pic }}
\nc{\LPic}{\on{LPic }}
\nc{\Sch}{\on{Sch}}
\nc{\Sh}{\on{Sh}}
\nc{\Set}{\on{Set}}
\nc{\spo}{\on{sp\  }}
\nc{\Spec}{\on{Spec}}
\nc{\mSpec}{\on{mSpec}}
\nc{\Specb}{\bold {Spec}}
\nc{\Projb}{\bold {Proj}}
\nc{\Specan}{\on{Specan}}
\nc{\Spo}{\on{Sp}}
\nc{\Spf}{\on{Spf}}
\nc{\sym}{\on{sym}}
\nc{\symm}{\on{symm}}
\nc{\rop}{\on{r}}
\nc{\Td}{\on{Td}}
\nc{\Tor}{\on{Tor}}

\nc{\Alg}{\on {Alg}}
\nc{\Artin}{\cal{A}rtin}
\nc{\Dgcoalg}{\cal{D}gcoalg} \nc{\Dglie}{\cal{D}glie}
\nc{\Ens}{\cal{E}ns} \nc{\Fsch}{\cal{F}sch}
\nc{\Groupoids}{\cal{G}roupoids}
\nc{\Holie}{\cal{H}olie}
\nc{\Mor}{\cal{M}or}
\nc{\CF}{\ensuremath{\cal{F}}}
\nc \Kc{{\ensuremath{\cal K}}}
\nc \Lc{{\ensuremath{\cal L}}}
\nc \lcc{{\mathcal l}} 
\nc \CC{{\ensuremath{\cal C}}} 
\nc \Cc{{\ensuremath {\cal C}}}
\nc \Pc{{\ensuremath{\cal P}}}
\nc \Dc{\ensuremath{\mathcal D}}
\nc \DC{\ensuremath{\mathcal C}}
\nc \Ac{{\ensuremath{\cal A}}} 
\nc \Bc{{\ensuremath{\cal B}}}
\nc \Ec{{\ensuremath{\cal E}}}
\nc \Fc{{\ensuremath{\cal F}}}
\nc \Mcc{{\ensuremath{\cal M}}} 
\nc \hM{\hat{\Mcc}} 
\nc \bM{\bar {\Mcc}} 
\nc\hbM{\hat{\bar \Mcc}}  
\nc \Nc{{\ensuremath{\cal N}}}
\nc \Hc{{\ensuremath{\cal H}}} 
\nc \Ic{{\ensuremath{\cal I}}} 
\nc \Oc{\ensuremath{{\cal O}}}
\nc \Och{\hat{\cal O}} 
\nc \Sc{{\ensuremath{{\cal S}}}}
\nc \Tc{\ensuremath{{\cal T}}} 
\nc \Vc{{\ensuremath{{\cal V}}}} 
\nc{\CA}{{\ensuremath{{\cal A}}}}
\nc{\CB}{{\ensuremath{{\cal B}}}}
\nc{\Gc}{{\ensuremath{{\cal G}}}}
\nc{\CH}{\ensuremath{\mathcal H}}
\nc{\CI}{{\ensuremath{{\cal I}}}}
\nc{\CM}{{\ensuremath{{\cal M}}}}
\nc{\CN}{{\ensuremath{{\cal N}}}}
\nc{\CO}{{\ensuremath{{\cal O}}}}
\nc{\Rc}{{\ensuremath{{\cal R}}}}
\nc{\CT}{{\ensuremath{\mathcal T}}}
\nc{\CU}{\ensuremath{{\cal U}}}
\nc{\CV}{\ensuremath{{\cal V}}}
\nc{\CZ}{\ensuremath{{\cal Z}}}
\nc{\Homc}{\ensuremath{{\cal {Hom}}}}

\nc{\fa}{\frak{a}}
\nc{\fA}{\frak{A}}
\nc{\fg}{\frak{g}}
\nc{\fh}{\frak{h}}
\nc{\fI}{\frak{I}}
\nc{\fK}{\frak{K}}
\nc{\fm}{\frak{m}}
\nc{\fP}{\frak{P}}
\nc{\fS}{\frak{S}}
\nc{\ft}{\frak{t}}
\nc{\fX}{\frak{X}}
\nc{\fY}{\frak{Y}}

\nc{\bF}{\bar{F}}
\nc{\bCP}{\bar{\cal{P}}}
\nc{\bm}{\mbox{\bf{m}}}
\nc{\bT}{\mbox{\bf{T}}}
\nc{\hB}{\hat{B}}
\nc{\hC}{\hat{C}}
\nc{\hP}{\hat{P}}
\nc{\htest}{\hat P}

\nc{\nen}{\newenvironment}
\nc{\ol}{\overline}
\nc{\ul}{\underline}
\nc{\ra}{\to}
\nc{\lla}{\longleftarrow}
\nc{\lra}{\longrightarrow}
\nc{\Lra}{\Longrightarrow}
\nc{\Lla}{\Longleftarrow}
\nc{\Llra}{\Longleftrightarrow}
\nc{\hra}{\hookrightarrow}
\nc{\iso}{\overset{\sim}{\lra}}

\nc{\dsize}{\displaystyle}
\nc{\sst}{\scriptstyle}
\nc{\tsize}{\textstyle}
\nen{exa}[1]{\label{#1}{\bf Example.\ } }{}

\nen{rem}[1]{\label{#1}{\em Remark.\ } }{}
\nen{exer}[1]{\label{#1}{\em Exercise.\ } }{}

\theoremstyle{definition}
\theoremstyle{remark}

\nc{\Sats}[1]{Sats~\ref{#1}}
\nc{\Sa}[1]{({\sl Sa.}~\ref{#1})}
\nc{\Kor}[1]{({\sl Kor.}~\ref{#1})}
\nc{\Korollarium}[1]{Korollarium~\ref{#1}}
\nc{\Exe}[1]{{\sl Exempel}~\ref{#1}}
\nc{\Anm}[1]{{\sl Anmärkning}~\ref{#1}}
\nc{\Not}[1]{{\sl Not}~\ref{#1}}

\nc{\Formodan}[1]{Förmodan~\ref{#1}}
\nc{\Pastaende}[1]{Påstående~\ref{#1}}

                      {$\square$\vskip\medskipamount\par}

\begin{document} 
\title{Hilbert series of modules over Lie algebroids}
\author{Rolf K{\"a}llstr{\"om}}
\address{Department of Mathematics, University of G{\"a}vle,   S-801 76, Sweden}
\email{rkm@hig.se}
\author{Yohannes Tadesse}
\address{Department of Mathematics,
  University of Stockholm,
S-106 91, Sweden}
\email{tadesse@math.su.se}
\begin{abstract} 
  We consider modules $M$ over Lie algebroids $\g_A$ which are of finite type over a local
  noetherian ring $A$.  Using ideals $J\subset A$ such that $\g_A \cdot J\subset J $ and the length
  $\ell_{\g_A}(M/JM)< \infty$ we can define in a natural way the Hilbert series of $M$ with respect
  to the defining ideal $J$.  This notion is in particular studied for modules over the Lie
  algebroid of $k$-linear derivations $\g_A=T_A(I)$ that preserve an ideal $I\subset A$, for example
  when $A=\Oc_n$, the ring of convergent power series.  Hilbert series over
  Stanley-Reisner rings are also considered.
 \end{abstract}
%  \begin{keyword}
% Hilbert series \sep tangential vector field \sep Lie algebroid \sep
% complete intersection \sep invariant rings \sep isolated singularity
%  \end{keyword}
\maketitle
\tableofcontents

\section{Introduction}
Let $(A, \mf_A, k)$ be an {\it allowed} local commutative noetherian $k$-algebra of characteristic
$0$, which entails in particular that the (generic) rank of the $A$-module of $k$-linear derivations
$T_{A/k}$ coincides with the Krull dimension of $A$.  Let $I$ be an ideal of $A$ and $
T_{A}(I)\subset T_{A}$ be the $A$-submodule of $k$-linear derivations $\delta$ of $A$ such that
$\delta \cdot I \subset I$, which we call the tangential Lie algebroid along $I$. More generally, a
Lie algebroid is an $A$-module $\g_A$ equipped with a structure of Lie algebra over $k$ and a
homomorphism of $A$-modules $\alpha :\g_A\to T_{A}$ satisfying natural compatibility relations for
the Lie algebra and module structures; the notion of module over a Lie algebroid $\g_A$ is more or
less what can be expected \Defn{\ref{def-module}}.  If $M$ is a $\g_A$-module of finite type as
$A$-module, we say that a proper ideal $J$ is a {\it defining ideal} for $M$ if $\alpha(\g_A) \cdot
J \subset J$ and the length $l_{\g_A}(M/JM)< \infty$.  We prove that if $J$ is a defining ideal,
then $l_{\g_A}(J^nM/J^{n+1}M) < \infty$ for every positive integer $n$, so one can define the {\it
  Hilbert series of a $\g_A$-module $M$ with respect to the defining ideal $J$}
\begin{displaymath}
  H_{M}^J (t) =  \sum_{n\geq 0} l_{\g_A}(\frac{J^nM}{J^{n+1}M})t^n \in \Zb[[t]].
\end{displaymath}
This series extracts information about the complicated structure of $\g_A$-modules, which in general
do not have a finite length. Some basic examples of $T_A(I)$-modules are $A$, $I$, the integral
closure of $I$, and the Jacobian ideal of $I$ \citelist{\cite{kallstrom:preserve}*{Sec
  3.3}\cite{kallstrom:liftingder}*{Th. 3.2.2}}. In the study of regular
singular $\Dc_A$-modules $N$, where $\Dc_A$ is the ring of differential operators on $A$, there is a
need to understand $T_{A}(I)$-submodules $N_0 \subset N$ that are of finite type over $A$.

When $\alpha (\g_A)=0$, so $\g_A$ is a Lie algebra over $A$, a defining ideal for $A$ is the same as
$\mf_A$-primary ideal. If moreover our $\g_A$-modules $\frac{J^nM}{J^{n+1}M}$ have complete flags,
so $l_{\g_A}(\frac{J^nM}{J^{n+1}M})= \dim_k(\frac{J^nM}{J^{n+1}M})$, it follows from Hilbert's
theorem that $H_M^J(t)$ is a rational function.  When $(R,\g_R)$ is a Lie algebroid over an allowed
local ring $R$ such that $R$ is simple over $\g_R$ (then $R$ will also be regular) we define its
{\it fibre Lie algebra} by $ \g_k= k\otimes_R \Ker (\g_R \to T_{R/k})$.  For example, if $J$ is a
radical ideal in a regular allowed local ring $A$ of dimension $n$, defining a smooth variety of
codimension $r$, and $\g_A= T_A(J)$, then $J$ is a maximal defining ideal of the $\g_A$-module $A$
and the fibre Lie algebra $\g_k $ of $\g_R=\g_A/J \g_A$ ($R= A/J$) is isomorphic to $\gl_r $, the
general linear algebra.  In general, the length of a $\g_R$-module $M$ of finite type over $R$ is
less than the length of the fibre $k\otimes_R M$ as module over the fibre Lie algebra $ \g_k$, but
sometimes equality holds
\begin{displaymath}\tag{L}
  \ell_{\g_R}(M)=\ell_{\g_k}(k\otimes_RM),
\end{displaymath}
and we then say $M$ is a {\it local system}.  Letting $J_m$ be a
maximal defining ideal of the $\g_A$-module $A$ we prove that $J_m$ is
a maximal defining ideal of any $\g_A$-module $M$ of finite type over
$A$, and that $R=A/J_m$ is a regular local ring.  We can therefore
more generally say $M$ is a local system along the maximal defining
ideal $J_m$ if each homogeneous component of the graded $\g_R$-module $
G^\bullet_{J_m}(M)=\oplus_{i\geq 0} J_m^iM/J_m^{i+1}M$ is a local system. 

A defining ideal $J$ is contained in a unique maximal defining ideal $J_m= \sqrt{J}$, the radical of
$J$ \Prop{\ref{reg-def-ideal}}.

\newtheorem{thmlabel}{\bf Theorem}
\renewcommand{\thethmlabel}{\ref{main}}

\begin{thmlabel}\label{maintheorem} Let $M$ be a
  $\g_A$-module of finite type, $J$  a defining ideal, and $J_m$ be its maximal defining ideal.
  If $M$ is a local system along $J_m$, then $H_M^J(t)$ is a rational
  function and the function $n\to \ell_{\g_A}(M/J^{n+1}M)$ is a quasi-polynomial for high $n$.
\end{thmlabel}

The proof of \Theorem{\ref{maintheorem}} is based on a reduction to modules over
the fibre Lie algebra $\g_k$ and by taking invariants over a maximal nilpotent
subalgebra of a Levi factor of $\g_k$, applying Hilbert's finiteness theorem on
the finite generation of invariant rings as extended by Hadziev \cite{hadziev},
who used an idea that arguably can be traced back to the classical invariant
theorists \cite{roberts:invariants}.  

Now the condition in \Theorem{\ref{maintheorem}} that $M$ be a local system along the maximal
defining ideal $J_m$, in the indicated graded sense, perhaps at a first glance seems to be a rather
special and technical one, but this is actually not so. On the contrary, the condition is naturally
satisfied in many important situations and local systems are in fact abundant.  Perhaps the most
important example of local system appears when the defining ideal is a maximal ideal in the ordinary
sense, meaning simply that $R=k$ so that \thetag{L} is self-evident, which is what occurs in the
study of complex analytic singularities (see e.g. \Theorem{\ref{yau-cor}}); another case is when $A$
is regular and is either $\mf_A$-adically complete, or an analytic algebra over the complex numbers
\Th{\ref{complete-prop}}.  If a Lie algebra $\afr$ acts transitively on a regular allowed ring $A$
then we get local systems by localizing finite-dimensional $\afr$-modules
\Th{\ref{localisation}}. Yet another favourable case is worked out in \Section{\ref{toral-section}}
, where we show how to compute Hilbert series for monomial ideals in an allowed regular ring $A$
(see \Theorem{\ref{mon-the}} and \Proposition{\ref{mon-hilb}}); this Hilbert series reflects
symmetries of the monomial ideal unlike the ordinary Hilbert series.  Finally, if a $\g_A$-module
$M$ is cyclic over $A$, then the $\g_R$-module $G^\bullet_{J_m}(M)$ is a direct sum of modules
satisfying \thetag{L}; this is proven for principal ideals in \Proposition{\ref{loc-ideal}}, but the
general case is similar.

In \Section{\ref{hypersurfaces}} we study Hilbert series of complex analytic
singularities. The general set-up is an ideal $I \subset \mf$, where $\mf$ is
the maximal ideal of the ring $\Oc_n$ of convergent power series in $n$
variables, and $\g= T_{\Oc_n}(I)$, where we can assume also without loss of
generality that $\g \subset T_{\Oc_n}(\mf)$, so that the fibre Lie algebra is
$\g_{\Cb}= \g/\mf \g$.  The Hilbert series
\begin{displaymath}
  H_M(t)=\sum_{i\geq 0} \ell_{\g}(\frac{\mf^iM}{\mf^{i+1}M})t^i = \sum_{i \geq 0} \ell_{\g_{\Cb}}(\frac{\mf^iM}{\mf^{i+1}M}) t^i 
\end{displaymath} 
is a useful summary of a $\g$-module $M$ of finite type over $\Oc_n$. The first question to settle
is when the fibre Lie algebra $\g_{\Cb} $ is solvable, as the Hilbert series will then coincide with
the ordinary Hilbert series and thus give us a rather good control of the $\g$-module structure. We
prove that if $\Oc_n/I$ is a complete intersection ring with an isolated singularity and $I \subset
\mf^2$ ($I\subset \mf^3$ when $I$ is principal), then $\g_{\Cb}$ is solvable \Th{\ref{isol-solv}},
which was proven for hypersurfaces by Granger and Schulze \cite{granger-schulze:initial-lieb}.  Now
put $J= (f+ T_{\Oc_n}\cdot f)$, where $f\in \mf$, and assume that $B=\Oc_n/(f)$ has an isolated
singularity. We prove that the fibre Lie algebra of the Lie algebroid $T_{\Oc_n}(J)$ is solvable
\Th{\ref{yau-solv}}, where the proof is based on Schulze and Yau's result \cite{schulze:solvable,
  yau:solvability} that the derivation algebra $T_{A}$ of the modular algebra $A= \Oc_n/J$ is
solvable. Allowing $f$ also to have a non-isolated singularity we consider Hilbert series of $J$
both as $\g= T_{\Oc_n}(I)$ - and $\g= T_{\Oc_n}(J)$ -modules, $H_{A}(t)=\sum_{i \geq 0} \ell_{\g}
(J^i/J^{i+1}) t^i$, proving it is a rational function in either case; moreover by a theorem of Yau
and Mather \cite{mather-yau} (isolated singularities) and Greuel, Lossen and Shustin
\cite{greuel-lossen-shusten} (general case) it follows that $H_{A}(t)$ is completely determined by
the algebraic structure of its degree zero part $A$ (see \Theorem{\ref{yau-cor}}). So there arises a
natural converse problem: If two hypersurfaces $B$ and $B'= \Oc_n/(f')$ of equal dimension, have
equal Hilbert series $H_A(t)= H_{A'}(t)$, how are the rings $B$ and $B'$ then related?

In \Section{\ref{section2}} we work out some basic results for modules $M$ over Lie algebroids
$\g_A$, which are of finite type over $A$.  First we give some salient relations between the length
of $A$ and $M$, where we want to emphasise the importance of \Proposition{\ref{length-prop}}, (5),
for the very definition of Hilbert series \Prop{\ref{defining-module}}.  Then the notion of local
system is explained, and we give some examples; in particular we show how representations of Lie
algebras give rise to local systems after localisation \Th{\ref{localisation}}.  To determine the
structure of fibre Lie algebras of Lie algebroids that contain a (weakly) toral subalgebra we need a
recognition theorem \Th{\ref{recognition}}, which determines the structure of the semi-simple part
of a Lie algebra $\g \subset \gl_k(V)$ that contains a Cartan subalgebra of $\Sl_k(V)$. This is used
in \Section{\ref{toral-section}} to work out the fibre Lie algebras of monomial rings. We give
special attention to Stanley-Reisner rings of simplicial complexes, showing how the structure of the
fibre Lie algebra is closely related to a certain stratification of the vertex set and that the Lie
algebroid Hilbert series turns out to be the ordinary Hilbert series of a certain smaller simplicial
complex given by the same stratification \Thms{\ref{hilb-eq}}{ \ref{lie-alg-hilb}}.

We want to express our deep gratitude to Michel Granger and also for the important input from an
anonymous referee for pointing out several unclear or incorrect statements in earlier versions
of this work, and also for help at improving the disposition. We thank Jan Stevens for the reference
\cite{greuel-lossen-shusten}.

\section{Lie algebroids, modules and defining ideals}\label{section2}
\subsection{Allowed  rings}\label{allowed}
Let $k$ be a field that contains the rational numbers $\Qb$. We will deal with noetherian
commutative $k$-algebras $A/k$ such that the $A$-module of $k$-linear derivations $T_{A/k}$ is ``big
enough'', so that in particular the Jacobian criterion of regularity applies. 

First we recall that any local ring $A$ containing the rational numbers also contains a quasi
coefficient field, which is a subfield $l\subset A$ such that $l \to k_A=A/\mf_A$ is $0$-etale, so
in particular $T_{k_A/l}=0$. 

\begin{theorem}[\cite{matsumura}*{Thms. 30.6, 30.8}] \label{goodring}
  Let $(R,\mf_R)$ be a regular local ring of dimension $n$ containing
  the rational numbers $\Qb$.  Let $l$ be a quasi-coeffient field of
  $R$ and $K$ be a coefficient field of its completion $R^*$ such that
  $l\subset K$. The following conditions are equivalent:
  \begin{enumerate}
  \item There exist $\partial_1, \dots, \partial_n \in T_{R/l}$ ($l$-linear derivations) and
    $f_1, \dots , f_n \in \mf_R$ such that $\det \partial_i(f_j)\not
    \in \mf_R$.
  \item If $\{x_1, \dots , x_n\}$ is a regular system of parameters
    and $\partial_{x_i}$ are the partial derivatives of $R^* = K[[x_1,
    \dots , x_n]]$, $\partial_{x_i}(x_j)= \delta_{ij}$, then
    $\partial_{x_i} \in T_{R/l}$.
  \item $T_{R/l}$ is free of rank $n$.  
  \end{enumerate}
  Furthermore, if these conditions hold, then for any $P\in \Spec R$,
  putting $A= R/P$, we have $T_{A/l}= T_{R/l}(P)/PT_{R/l}$, and $\rank
  T_{A/l} = \dim A$.
\end{theorem}
If the equivalent conditions in \Theorem{\ref{goodring}} hold, then we
say that $(R, \mf_R)$ satisfies the weak Jacobian condition $(WJ)_l$.
If we now should want to work over an arbitrary base field $k\subset A$ technical problems would
appear, having to do with difficulties in describing $T_{A/k}$, and also for using weight structures
for modules over Lie algebras. To simplify the exposition we therefore make the following definition.

\begin{definition}
  An {\it allowed} $k$-algebra is a local ring of the form $A=R/I$, where $R$ satisfies $(WJ)_k$,
  $I$ is an ideal of $R$, and $k$ is an algebraically closed coefficient field of characteristic
  $0$.
\end{definition}
In particular, $A= R/I= k+\mf_{A}$, and if $P$ is a minimal prime divisor of $I$, then  $\rank
A_{P}\otimes_{A}T_{A/k} = \dim R/P $. 

The {\it main examples} of allowed $k$-algebras $A=R/I$ appear when $R$ is either: (1) a localisation
of a polynomial ring, (2) a formal power series ring and (3) a ring of convergent power series,
where $l$ is either the field of real or complex numbers and $k$ is the field of complex numbers
$\Cb$.

For an ideal $I$ of height $r$ we let $J $ be the ideal that is generated by all the determinants
$\det (\partial_i (f_j))$, where $\partial_i \in T_{R/l}$ and $f_j \in I$, $1\leq i,j \leq r$.  It
is straightforward to see that $T_{R/l}(I) \cdot J \subset J$ (see e.g.  \cite{kallstrom:preserve}).
Therefore the {\it Jacobian ideal} $\bar J = A J $ is a $T_{A/l}$-submodule of $A$.

Recall that the ring of differential operators $
\Dc_{A/l}\subset \End_l (A)$ is defined inductively as $\Dc_{A/l}= \cup
_{m\geq 0} \Dc^m_{A/l} $, $\Dc^0_{A/l}= \End_A(A)= A$, $\Dc^{m+1}_{A/l}= \{P
\in \End_l (A) \ \vert \ [P, A]\subset \Dc^m_{A/l}\}$, where $[P, A]= PA -
A P \subset \End_l (A)$.  It is easy to see that $T_{A/l}\subset \Dc^1_{A/l}
\subset \Dc_{A/l}$, and conversely, if $P\in \Dc^1_{A/l}$, then $P - P(1) \in
T_{A/l}$; hence
\begin{displaymath}
  \Dc^1_{A/l} = A+ T_{A/l} .
\end{displaymath}
In general the algebra $\Dc_{A/l}$ need not be generated by the $A$-submodule $\Dc^1_{A/l}$, and in
fact need not even be noetherian, as was first exemplified in \cite{bernstein-gege:cubic}. On the
other hand we have the following companion to \Theorem{\ref{goodring}}:
\begin{proposition}\label{diffop-gen-deg1}
  Assume that $R/l$ satisfies $(WJ)_l$. Then $\Dc^1_{R/l}$ generates the algebra $\Dc_{R/l}$.
\end{proposition}
Select $x_i$ and $\partial_{x_i}$ as in \Theorem{\ref{goodring}}. Given a multi-index $\alpha =
(\alpha_1, \dots , \alpha_n)$, $n= \dim R$, we put $X^\alpha = x_1^{\alpha_1}x_2^{\alpha_2} \cdots
x_n^{\alpha_n} \in R$, $\partial^\alpha
= \partial_{x_1}^{\alpha_1} \partial_{x_2}^{\alpha_2}\cdots \partial_{x_n}^{\alpha_n}\in \Dc_{R/l}$,
$|\alpha| = \sum\alpha_i$, and $\alpha ! = \alpha_1!  \cdots \alpha_n!$.
\begin{lemma}\label{zero-lemma}
If $P \in \Dc^m_{R/l}$ and $P(X^{\alpha})=0$ when $|\alpha|\leq m$, then
  $P=0$.
\end{lemma}
\begin{proof}
  We use induction over $m$. If $P\in \Dc_{R/l}^0 = R$, then
  $P=P(1)=0$. Assume the assertion is true when $P\in \Dc^m_{R/l}$, and
  let now $P\in \Dc^{m+1}_{R/l}$, such that $P(X^\alpha)=0$ when $|\alpha|
  \leq m+1$. Then $P^{(i)}=[P, x_i] \in \Dc^m_{R/l}$, and
  $P^{(i)}(X^\alpha) = P(x_i X^\alpha) - x_i P(X^\alpha) =0$ when
  $|\alpha|\leq m$, so by induction, $P^{(i)}=0$. Therefore $[P,
  X^\alpha] =0$, and since $P(1)=0$, we get $P(X^\alpha)=0$ for any
  monomial $X^\alpha$. Expanding an element $f\in R$ in the form $f=
 \sum_\alpha c_\alpha X^\alpha + f_{i+1}$, where $c_\alpha \in k$ and
  $f_{i+1}\in \mf_R^{i+1}$, it follows that $P(f) \in \cap_{i >0}
  \mf_R^i =\{0\}$, by Krull's intersection theorem. Therefore $P=0$.
\end{proof}
\begin{pfof}{\Proposition{\ref{diffop-gen-deg1}}} Let $\Dc(T_{R/l})\subset \Dc_{R/l}$ be the subalgebra that is
  generated by $\Dc^1_{R/l}$.  If $P \in \Dc^m_{R/l}$, define inductively for
  $\alpha$ such that $|\alpha |\leq m$, $a_0 = P(1)$, and $a_\alpha
  =\frac 1{\alpha !}(P(X^\alpha)- \sum_{|\beta| < |\alpha|}
  a_\beta \partial^\beta (X^\alpha) )$, so $\sum_\alpha
  a_\alpha\partial^\alpha \in\Dc(T_{R/l})$.  One checks that $P - \sum_\alpha
  a_\alpha \partial^\alpha $ kills all monomials $X^\alpha$ such that
  $|\alpha |\leq m$, hence by \Lemma{\ref{zero-lemma}} $P =
  \sum_\alpha a_\alpha \partial^\alpha \in \Dc(T_{R/l})$.
\end{pfof}

\Proposition{\ref{diffop-gen-deg1}} was proven in
\citelist{\cite{sweedler:simplealgebras}*{Th. 18.2}\cite{hart:diffoperators}*{Th. 2}\cite{mcconnel-robson}*{15.5.6}\cite{
    bjork:analD}*{Th. 1.1.8}}, assuming either that $R$ is essentially of finite type or is the ring
of convergent power series. After the introduction of parameters, the methods in [loc cit]
presumably can be applied, but the argument above is maybe slightly more direct, and we also want to
stress that the result applies to any regular local $l$-algebra of characteristic $0$ satisfying the
weak Jacobian criterion.

For brevity we will write $T_A=T_{{A/k}}$ and $\Dc_A=\Dc_{{A/k}}$.

\subsection{Modules over Lie algebroids}\label{sec-modules}The following definition is basic to this paper.
\begin{definition} A Lie algebroid over a $k$-algebra $A/k$ is an $A$-module $\g_A$ of finite type,
  which is equipped with a structure of Lie algebra over $k$ and a homomorphism of Lie algebras and
  $A$-modules $\alpha: \g_A\to T_{A}$, such that the compatibility condition $[\delta, r\eta] =
  \alpha(\delta)(r)\eta + r [\delta, \eta]$, $\delta, \eta\in \g_A, r\in A$, is satisfied.
\end{definition}

For
example, if $I$ is an ideal of $A$, then $T_{A}(I)= \{\delta \in T_A \ \vert \ \delta (I)\subset
I\}$ is a Lie algebroid, where $\alpha $ is the inclusion map, and of course Lie algebras over $A$
are Lie algebroids, with $\alpha =0$.
\begin{remark} Strictly speaking, given a sheaf of Lie algebroids $\g_X$ on a scheme $X$, it is the
  spectrum $\Specb \So_{{\Oc_{X}}}(\g_X^*)$ of the symmetric algebra of its dual that is the
  infinitesimal version of a groupoid, and hence its associated ``geometric Lie algebroid'', while
  $\g_X$ is the sheaf of sections of $\Specb \So_{{\Oc_{X}}}(\g_X^*)$ over $X$.  For this reason,
  and in differential geometry in particular \cite{mackenzie-kirill:generalth}, one wants to make a
  clear distinction between the geometric Lie algebroid and its sheaf of sections, where the latter
  often instead are called Lie-Rinehart algebras \cite{huebschmann-duality}, and sometimes Atiyah
  algebras, d-Lie algebras, or pseudo Lie algebras.  All in all, we find our concise ``Lie
  algebroid'' evocative of a sheaf of Lie algebras (over a base field) being represented in a module
  as infinitesimal symmetries both along the fibres and horizontally for nearby fibres of the
  module.
\end{remark}

We want to define modules over a Lie algebroid in the same way as for Lie algebras, where in the
latter case it is a homomorphism $\g_k \to \gl_k(V)$ from a Lie algebra to the general Lie algebra
of a vector space $V$. When dealing with $A$-modules $M$ there is in general no exact Lie algebroid
counterpart to $\gl_k(V)$, due to the possibility that $A$ does not act faithfully on $M$, but one
gets fairly close to having a ``linear Lie algebroid'' of $M$. Define the map $i: A \to \End_k (M)
$, $i(a)(m)= am$.  We have the $A$-submodule
\begin{displaymath}
  \cf_{A}(M)= \{ \delta \in \End_k(M) \ \vert \ [\delta , i(A)]\subset
  i(A)\},
\end{displaymath}
and let $\Dc^1_A(M)= \{p\in \End_k (M) \ \vert \
[p,\End_A(M)]\subset
\End_A(M)\} $ be the module of first order differential operators on $M$, so $i(A)\subset
\Dc^1_A(M)$. By the Jacobi identity for the Lie bracket $[\cdot, \cdot]$ in $\End_k(M)$ it follows
that $\cf_{A}(M) \subset \Dc^1_A(M)$, defining moreover an $A$-submodule and Lie subalgebra over
$k$. If $\Ann M =0$, so the map $i$ is injective, there exists a natural map $\beta: \cf_A(M) \to
T_{A}$, giving $\cf_A(M)$ a structure of Lie algebroid (see \cite{kallstrom:preserve}); if $\delta
\in \cf_A(M)$, then $\beta(\delta)$ is determined by the relation $[\delta, i(a)]=
i(\beta(\delta)(a))$.  More generally, we have the Lie algebroid $\beta : \cf_{A}(M) = \cf_{i(A)}(M
)\to T_{i(A)}$, where $M$ is considered as $A$- or $i(A)$-module.
\begin{definition}\label{def-module}
  \begin{enumerate}
  \item We say that $\g_A$ {\it acts} on an $A$-module $M$ if we are given a map (not necessarily
    $A$-linear) $\rho : \g_A\to \cf_A(M)$ such that $\rho(\delta)(rm)= \alpha (\delta)(r)m +
    r\rho(\delta)(m)$, $\delta \in \g_A, r\in A, m\in M$.
  \item A module $M$ over $\g_A$ is given by a homomorphism $\rho : \g_A\to \cf_A(M)$ both as
    $A$-modules and Lie algebras, which moreover satisfies (1).
  \end{enumerate}
\end{definition}
\begin{remark}
  It is in a sense unnatural to require (1) for a module, i.e. the identity $i(\alpha(\delta)(r)) =
  [\rho(\delta),i(r)]$. For example, if $M$ is torsion free as $A$-module, $M\neq 0$, and
  $\rho(\g_A)\cdot M=0$ (trivial module), then (1) implies $\alpha =0$.  However, (1) follows from
  the first part of (2) if $\Ann_{i(A)} (\rho(\g_A)) =0$. Proof: Let $\delta, \eta \in \g_A$, and
  $r\in A$.  Then $[\delta, r\eta] = \alpha(\delta)(r)\eta + r [\delta, \eta]$, implying that
  $i(\alpha(\delta)(r))\rho(\eta) = [\rho(\delta), i(r)]\rho(\eta)$, which implies
  $i(\alpha(\delta)(r)) = [\rho(\delta), i(r)]$.  Thus (1) is automatic when $M$ is a non-trivial
  $\g_A$-module and torsion free as $A$-module.
\end{remark}

A Lie algebroid $\g_{A}$ acts on itself by $\g_A \to \cf_A(\g_A)$, $\delta \mapsto [\delta, \cdot]$,
but this does not define a $\g_{A}$-module when $\alpha \neq 0$.

Naturally, we usually  write $\delta m= \delta \cdot m = \rho(\delta)(m)$, for $\delta \in \g_A$ and
$m\in M$, when the map $\rho$ is clear from the context.  By \Proposition{\ref{diffop-gen-deg1}},
$T_{R}$-modules are the same as modules over the ring of differential operators $\Dc_{R}$ when $R$
is a regular allowed local $k$-algebra over a field of characteristic $0$.  Since $A$ is noetherian,
if $M$ is a $\g_A$-module of finite type over $A$, then $M$ is noetherian as $\g_A$-module and
contains in particular a maximal proper $\g_A$-submodule $N\subset M$, so the quotient $M/N$ is a
simple $\g_A$-module.

The $A$-module $A$ is always a $\g_A$-module, using the map $\alpha$, and the nilradical $\nil A$ is
a $\g_A$-submodule, since $\Char k =0$ (see \cite{scheja:fortzetsungderivationen, kallstrom:preserve}).
\begin{proposition}\label{regularprop}  Let $A$ be  an allowed 
  ring, and consider the conditions:
  \begin{enumerate}
  \item $A$ is regular.
  \item $\alpha : \g_A \to T_{A}$ is surjective (we say that $\g_A$ is transitive).
\item $A$ is a  simple $\g_A$-module.
\item $A$ is integral,  the depth $\depth \alpha(\g_A)\geq 2$, and $\mf_A^d T_A
  \subset \alpha(\g_A)$ for some integer $d\geq 0$.
  \end{enumerate}
  We have $(3)\Rightarrow (1)$, $(1)  \& (2)$ $   \Rightarrow (3)$, and
  $(4)\Rightarrow (2)$.

  Assume $l_{\g_A}(A)< \infty$. If $I\subset \mf_A$ is an invariant ideal, i.e. $\g_A \cdot I \subset I$,
  then $I \subset \nil A $, and the reduced ring $A/\nil A$ is regular, and simple over $\g_{A}$.
\end{proposition}
\begin{remark}\label{simpleremark}
  \begin{enumerate}[label=(\roman*)]
  \item $(3)$ in \Proposition{\ref{regularprop}} does not imply $(2)$.  Derivations of a regular
    noetherian ring which have no proper invariant ideals has been much studied, for example in
    \citelist{\cite{seidenberg:diffideals}*{Th. 3 p. 26}
      \cite{hart:der_fin_type}\cite{jordan:diffsimple} \cite{coutinho-levcovitz:diffsimple}}.
  \item If $A$ is simple over $\g_A$, hence regular, and $\mf_A^d T_A \subset \alpha(\g_A)$, one can
    prove directly in local coordinates that $(2)$ follows.
  \end{enumerate}
\end{remark}
\begin{proof}

  $(3)\Rightarrow (1)$: The Jacobian ideal $\bar J$ is a non-zero
  $\g_A$-submodule of the simple $\g_A$-module $A$, hence $\bar
  J=A$. Since $A$ is allowed, it follows that $A$ is regular  \Th{\ref{goodring}}.

  $(1 )\& (2)\Rightarrow (3)$: \Theorem{\ref{goodring}} implies, since $A$ is regular allowed, that
  $T_{A}$ is free and also that $A$ has a regular system of parameters $x_1, \dots , x_r$, $r= \dim
  A$, together with derivations $\partial_1, \dots , \partial_r \in T_{A}$ satisfying
  $\partial_i(x_j)= \delta_{ij}$.  By the same theorem the $T_A$-invariant ring $ l \subset
  A^{T_{A}} \subset (\cap_i (A^*)^{\partial_{x_i}}) \cap A = K \cap A= l$ ($K$ is determined by $l$
  \cite{matsumura}*{Th. 28.3, (iv)}), hence if $f\in A \setminus l $, then there exists
  $\partial_{x_i}$ such that $\partial_{x_i} (f)\neq 0$.  Moreover, $\partial_{x_i} (\mf_A \setminus
  \mf_A^2)\subset (A \setminus \mf_A) \cup \{0\}$, and by induction over $r\geq 1$,
  $\partial_{x_i}\cdot (\mf_A^r\setminus \mf_A^{r+1}) \subset (\mf_A^{r-1}\setminus \mf_A^r) $.  Let
  $I$ be a $T_A$-invariant non-zero ideal of $A$ and select $f\in I\cap \mf_A^r $ with minimal $r$,
  so that $f \in \mf_A^r \setminus \mf_A^{r+1}$. If $r \geq 1$, selecting $\partial_{x_i}$ such that
  $\delta (f) \neq 0$, then implies that $\partial_{x_i}(f)\in \mf^{r-1}\setminus \mf^r$, so that
  $r$ cannot be minimal.  It follows that $r=0$, and therefore $A = A f \subset I$. Since $\alpha$
  is surjective this implies that $A$ is simple over $\g_A$.

  $(4)\Rightarrow (2) $: Put $X= \Spec A$ and let $j: X_0= X\setminus \{\mf_A\}\to X$ be the open
  inclusion.  Letting $\bar {\g}_X$ and $T_X$ be the coherent sheaves of $\Oc_X$-modules that are
  associated to the $A$-modules $\alpha (\g_A)\subset T_A$, we get an injective map of
  $\Oc_X$-modules $\alpha : \bar {\g}_X \to T_X$. Since $\mf_A^d T_A\subset \alpha(\g_A)$, it
  follows that $j^*(\bar g_X) = j^*(T_X)$, and since $\depth \alpha (\g_A)\geq 2$, so the local
  cohomology group $H^1_{\mf_A}(\alpha (\g_A))=0$, we get
\begin{displaymath}
  \bar \g_X  = j_*j^*(\bar \g_X) = j_*j^*(T_X). 
\end{displaymath}
Since $\bar \g_X \subset T_X$ and as $A$ is integral, hence $T_A$ is torsion free, so $T_X \subset
j_*j^*(T_X)$, it follows that $ \bar \g_X = T_X = j_*j^*(T_X)$. Therefore the $A$-module of global
sections $\alpha(\g_A)$ of $\bar \g_X$ coincides with the $A$-module $T_A$ of global sections of
$T_X$.

  We prove the last assertion. As mentioned above $\nil A$ is preserved by $\g_A$, so $A'= A/\nil A$
  is a $\g_{A}$-module of finite length.  If, on the contrary, $\bar I = I \mod \nil A $ is a
  non-zero ideal in $\mf_{A'}$, by Nakayama's lemma and since $A'$ is reduced, $\bar I^n\subset
  \mf_{ A'} $, $n=1,2,\dots $ is a strictly descending sequence of $\g_{A}$-invariant ideals, so
  $A'$ will not have a finite length. Therefore $I \subset \nil A$.  It follows that $\nil A$ is a
  maximal $\g_A$-submodule of $A$, hence $A/\nil A$ is simple, and therefore it is regular, since
  $(3)$ implies $(1)$.
\end{proof}

Condition (3) in  \Proposition{\ref{regularprop}} has the  following implication:
\begin{proposition}\label{freelemma}
  Let $\g_A$ be a Lie algebroid over an allowed ring $A$ such that $A$ is simple over $\g_A$.  If
  $\g_A$ acts on an $A$-module of finite type $M$, then $M$ is free. In particular, $\g_A$ is free
  over $A$.
\end{proposition}
\begin{proof} 
  The Fitting ideals of $M$ are $\g_A$-invariant ideals of $A$ (see
  e.g. \cite{kallstrom:preserve}). Since $A$ is simple it follows that
  all Fitting ideals equals either $0$ or $A$, implying $M$ is free. A
  Lie algebroid $\g_A$ acts on itself by the adjoint action
  $\rho(\delta)(\eta)= [\delta, \eta]$, therefore $\g_A$ is free.
\end{proof}

\begin{proposition}\label{length-prop} Let $\g_A$ be a Lie algebroid
  over a local $k$-algebra $A$, and $M$ be $\g_A$-module of finite
  type as $A$-module.
  \begin{enumerate}
  \item If $M$ is simple and $\Ann M =0 $, then $A$ is simple and, in particular, $M$ is free over
    $A$.
  \item If $\ell_{\g_A}(M)< \infty$ and $\Ann_A M\subset \nil A$, then $ A/\nil A$ is simple.
  \item If $A$ is simple over $\g_A$, then $\ell_{\g_A}(M)\leq \rank
    (M)$.
  \item If $\ell_{\g_A}(A)< \infty$, then $\ell_{\g_A}(M) < \infty$.
  \item  $\ell_{\g_A}(A/\Ann M) \leq  \ell_{\g_A}(M).$
\end{enumerate}
   \end{proposition}

    \begin{proof}
      (1): If $I \subset A$ is a non-zero $\g_{A}$-invariant ideal, then since $\Ann_{A}(M)=0$, $I
      M$ is a non-zero $\g_{A}$-submodule, hence $IM = M$ since $M$ is simple, hence $I= A$ by
      Nakayama's lemma; hence $A$ is simple.  That $M$ is free follows from
      \Proposition{\ref{freelemma}}.

      (2): Recall that $\nil A$ is a $\g_A$-invariant ideal in $A$, so that $A_1
      = A/\nil A$ and $M_1= M/\nil A \cdot M $ are $\g_A$-modules in a natural
      way.  We have $ \ell_{\g_A}(M_1) \leq \ell_{\g_A}(M) <\infty$, and if on
      the contrary $I\subset A_1$ is a non-zero $\g_A$-invariant ideal, $I\neq
      A_1$, since $A_1$ is reduced, Nakayama's lemma implies that $I^n \cdot
      M_1$ is a strictly decreasing sequence of submodules, which results in a
      contradiction. Therefore $A_1$ is a simple $\g_A$-module.

 (3): If $0 \to M_1 \to M \to M_2 \to
      0 $ is a short exact sequence of $\g_A$-modules, by \Proposition{\ref{freelemma}} all these
      modules are free, and $\rank M = \rank M_1 + \rank M_2$.  From this follows that in any
      descending chain of $\g_A$-modules each subquotient is free over $A$, and hence it is
      stationary, and that the number of simple subquotients cannot exceed $\rank M$.

  (4): Since $\ell_{\g_A}(A) < \infty$ it follows from
  \Proposition{\ref{regularprop}} that the maximal proper
  $\g_A$-submodule $I$ of $A$ equals the nilradical. Put $B= A/I$ and
  $\g_B = B \otimes_A \g_A$, which is a Lie algebroid over the simple
  $\g_B$-module $B$.  Since $I^n =0$ for high $n$, there exists a
  finite set $\Omega$ of indices, so that $i\in \Omega$ if and only if
  $I^iM/I^{i+1}M \neq 0 $. Therefore by (2)
  \begin{displaymath}
    \ell_{\g_A}(M) = \ell_{\g_A} (\oplus_{i\in \Omega} \frac {I^iM}{I^{i+1}M}) = \sum_{i\in \Omega} \ell_{\g_B}(\frac {I^iM}{I^{i+1}M})< \infty.
  \end{displaymath}
%
%
% Let $0 \subset I_{n+1} \subset I_{n} \subset \cdots \subset
%      A $ be a composition series of $A$. Then $0 \subset I_{n+1}M
%      \subset I_{n}M \subset \cdots \subset M$ is a chain of
%      $\g_A$-modules, and it suffices to see that each subquotient
%      $I_iM/I_{i+1}M$ is of finite length. But $I_iM/I_{i+1}M$ is a
%      module over the Lie algebroid $(\g_{A_i}= A_i \otimes_A \g_A, A_i = I_i/I_{i+1})$,
%      where $A_i$ is simple over $\g_{A_i}$. Since $I_iM/I_{i+1}M$ is
%      of finite type over $A_i$, by (1) $l_{\g_A}(I_iM/I_{i+1}M) =
%      l_{\g_{A_i}}(I_iM/I_{i+1}M)\leq \rank (I_iM/I_{i+1}M)$.

  (5): We can assume that $\ell_{\g_A}(M)< \infty$, and we use induction over $n=\ell_{\g_A}(M)$.
  Since $\alpha (\g_A) \cdot \Ann M \subset \Ann M$, putting $A_1 = A/\Ann M$, then $\g_{A_1}=
  \g_A/(\Ann (M) \g_A)$ is a Lie algebroid over $A_1$ and $M$ is a $\g_{A_1}$-module such that
  $\Ann_{A_1} M =0$.  Since $\ell_{\g_{A_1}}(A_1)= \ell_{\g_A}(A_1)$ and $\ell_{\g_A}(M)=
  \ell_{\g_{A_1}}(M) $, it suffices thus to prove
      \begin{displaymath}
        \ell_{\g_{A_1}}(A_1)\leq \ell_{\g_{A_1}}(M).
\end{displaymath} 
If $M$ is a simple, (1) implies that $A_1$ is simple, proving the assertion when $n=1$.  Put $J
=\nil A_1$ , $B=A_1/J$, and $G(M)= \oplus_{i\geq 0} J^iM/J^{i+1}M$. Put also $\g_B=
\g_{A_1}/J\g_{A_1}$, which is a Lie algebroid over $B$ since $\alpha(\g_{A_1}) \cdot J \subset J$
(\cites{scheja:fortzetsungderivationen, kallstrom:preserve}). Then $G(M)$ and $G(A_1)$ are
$\g_{B}$-modules, and $\Ann_B (G(M))=0$. Moreover, $\ell_{\g_{A_1}}(A_1)=\ell_{\g_B}(G(A_1))$ and
$\ell_{\g_{A_1}}(M) = \ell_{\g_B}(G(M))$, so it is equivalent to prove
\begin{displaymath}
  \ell_{\g_B} (G(A_1)) \leq \ell_{\g_B}(G(M)),
\end{displaymath}
when $\Ann_B G(M)=\{0\}$, which we thus know is true when $M$ is simple.  Since $\ell_{\g_B}(G(M))=
\ell_{\g_{A_1}}(M)= \ell_{\g_A}(M)< \infty$, (2) implies that $B$ is a simple $\g_B$-module, and so
any $\g_B$-module of finite type is free over $B$ \Prop{\ref{length-prop}}.  Assume $n >1$ and that
the assertion is true for all $\g_A$ and $M$ such that $\ell_{\g_A}(M)\leq n-1$; hence the above
inequality holds when $\ell_{\g_B}(G(M))= \ell_{\g_A}(M)\leq n-1$.  Since $M$ is not simple there
exists a proper $\g_A$-submodule $M_1 \subset M$, so that $G(M_1)$ is a non-zero free $B$-module
\Prop{\ref{length-prop}}, in particular $\Ann_B G(M_1)= \{0\}$, hence by induction,
$\ell_{\g_B}(G(A_1)) \leq \ell_{\g_B}(G(M_1)) < \ell_{\g_B}(G(M))$.
% If $s\geq 2$, put $I=\sum P_i$ and $M_1 = IM$. For each prime we have $(M_1)_{P_i} = M_{P_i}$,
% implying $\supp M_1=\supp M = \Spec A_1$. Moreover, $T_{A_1} \cdot P_i \subset P_i$, and therefore
% $\alpha(\g_{A_1})\cdot I \subset I$, implying that $M_1$ is a proper $\g_{A_1}$-submodule of $M$
% such that $\supp M_1 = \supp M$. Therefore $\supp G(M_1) = \Spec B$, hence since $B$ is reduced
% $\Ann G(M_1)= 0$.  By induction $\ell_{\g_{A_1}}(M) > \ell_{\g_{A_1}}(M_1)\geq
% \ell_{\g_{A_1}}(A_1)$.
\end{proof}

% For a $\g_B$-module $N$, let $[N]$ be its image of in the Grothendieck group of
% $\g_B$-modules that are of finite type over $B$.  We can then write
% \begin{displaymath}
%   [G(M)] = \sum_{i=1}^r [N_i]
% \end{displaymath}
% where $N_i$ are $\g_B$ modules such that $\ell_{\g_B}(N_i)< n$ and $\cup \supp
% N_i = \Spec B$, since $\Ann_B G(M) =0$.  Since $B$ is reduced the canonical map
% \begin{displaymath}
%   B \to \bigoplus_{i=1}^r \frac B{\Ann_B G(N_i)}
% \end{displaymath}
% is injective. Therefore 
% \begin{displaymath}
%   \ell_{\g_B}(B) \leq \sum_{i=1}^r  \ell_{\g_B}(\frac B{\Ann_B G(N_i)}) \leq
%   \sum_{i=1}^r \ell_{\g_B} (N_i) = \ell_{\g_B} (G(M))  
% \end{displaymath}

  \subsection{Fibre Lie algebras and Local systems}\label{fibreslocal} Let  $(R, \mf_R)$ be a local allowed
  $k$-algebra which is {\it simple} over a Lie algebroid $\g_R$, so we have the exact sequence
    \begin{equation}\label{ex-seq}
      0 \to \g \to       \g_{R}\to  \bar \g_R\to 0, 
    \end{equation}
    where $\g$ is a Lie algebra over $R$ and $\bar \g_R $ is a Lie subalgebroid of $T_{R}$, such
    that $R$ contains no $\bar \g_R$-invariant ideals.  In practice, most likely $\bar \g_R= T_R$
    when R is $\bar \g_R$-simple, but see \Remark{\ref{simpleremark} (i)}, and if $\bar \g_R$ is a
    proper submodule of $T_R$ we have:
    \begin{lemma}\label{splitlemma-1} The sequence of $R$-modules
   \begin{displaymath}
     0 \to \bar \g_R \to T_R \to     \frac {T_R}{\bar \g_R}\to 0
 \end{displaymath}
 is split exact,  and each module is free.
 \end{lemma}
 \begin{proof} All terms in the sequence are of finite type over $R$ due to
   \Theorem{\ref{goodring}}, and since  $\bar \g_R$ acts on each of them we can conclude from
   \Proposition{\ref{freelemma}}.
 \end{proof}
 \begin{remark}
   It follows from \Theorem{\ref{goodring}} that $T_R$ is a {\it simple} Lie algebroid, so that if
   $\bfr$ is a Lie subalgebroid of $T_R$ such that $[T_R, \bfr] \subset \bfr$, then either $\bfr =0$
   or $\bfr = T_R$.  Therefore the exact sequence in \Lemma{\ref{splitlemma-1}} is not split as Lie
   algebroids (see before \Proposition{\ref{complete-prop}}) .
 \end{remark}

 The quotient
    \begin{displaymath}
      \g_k = k\otimes_R \g
    \end{displaymath}
    is a finite-dimensionaI Lie algebra over $k$, which we call the { \it fibre Lie algebra} of
    $\g_R$. 
    \begin{example}\label{fibre-example}
      Let $A$ be an allowed regular ring (e.g. $A= \Oc_n$), $I$ an ideal, and put $\g_A=
      T_A(I)$. Then $\g_A$ preserves the (radical) ideal $I_1$ of the singular locus of $V(I)$,
      $\g_A \subset T_A(I_1)$. Similarly, $T_A(I_1)$ preserves the ideal $I_2$ of the singular locus
      of $V(I_1)$, so $\g_A \subset T_A(I_1)\subset T_A(I_2) $. Iterating one eventually arrives at
      an ideal $I_n$ such that $A_n= A/I_n$ is regular, and $\g_A \subset T_A(I_n)$. However,
      $A/I_n$ need not be a simple module over $\g_A/I_n\g_A$ (it can happen that $\g_A\subset
      I_nT_A$), so let $J_m$ be a maximal ideal such that $\g_A \subset T_A(J_m)$, where now $I
      \subset I_1 \subset \cdots \subset I_n \subset J_m$; the ideal $J_m$ will be a so-called
      defining ideal of the $\g_A$-module $A$, discussed below \Sec{\ref{def-ideal-section}}.
      Putting $R= A/J_m$ and $\g_R = \g_A/J_m\g_A $ we can consider the exact sequence
      (\ref{ex-seq}), where $\bar \g_R \subset T_R = T_A(J_m)/J_mT_A$.  Here $\g = ( T_A(I)\cap J_mT_A)/ (J_mT_A(I))$ and the fibre Lie algebra
      \begin{displaymath}
        \g_k = \frac{\g}{\mf_R \g} = \frac{ T_A(I)\cap
        J_mT_A}{ \mf_A(T_A(I)\cap J_mT_A) + J_m T_A(I) }.
    \end{displaymath}
    The fibre Lie algebra $\g_k$ encodes symmetries of the singularites of $V(I)$.  For example, if
    $I= (x_1, x_2)\subset \Oc_3$, then we get the commutative Lie algebra $\g_k = k
    \overline{x_1\partial_{x_1}} + k\overline {x_2
        \partial_{x_2}}$, where $\overline{x_i\partial_{x_i}}$ is represented by
      $x_i \partial_{x_i}\in T_A(I)$. 
    \end{example}
    In \Theorem{\ref{mon-the}} the fibre Lie algebra is computed for monomial ideals
    and \Section{\ref{hypersurfaces}} contains some complex analytic examples.

    We have a natural specialisation homomorphism of Lie algebras $f: \g \to \g_k$, $\delta \mapsto
    1\otimes \delta$, and given a $\g_R$-module $M$, which is of finite type over $R$, there is a
    $\g$-linear specialisation map $M \to k\otimes_R M $, so that its fibre
    \begin{displaymath}
      k\otimes_R M
    \end{displaymath}
    is a $\g$-module in a natural way, and since clearly $\mf_R \g \cdot k\otimes_R M =0$, the fibre
    is even a $\g_k$-module, which we call the associated {\it fibre module} of the $\g_R$-module
    $M$.  Letting $\Mod_f(\g_R)$ be the category of $\g_R$-modules that are of finite type over $R$
    and $\Mod_f(\g_k)$ the category of $\g_k$-modules that are of finite dimension over $k$, we thus
    have defined the fibre functor
    \begin{eqnarray*}
      F:  \Mod_f(\g_R)&\to&  \Mod_f(\g_k),  \\ M &\mapsto& k\otimes_R M.
    \end{eqnarray*}

    Since $R$ is simple, so $\g$ and $M$ are free over $R$ \Prop{\ref{freelemma}}, any choice of
    basis of $M$ and $\g$ induces an isomorphism of $R$-modules
\begin{eqnarray*}
  M &\to&    R\otimes_k k\otimes_R M,\\
  \g &\to& R \otimes_k \g_k.
\end{eqnarray*}
However, the second map is in general not a homomorphism of Lie algebras over $R$. Also the first
map is not a homomorphism of $\g$-modules, where $\g$ acts trivially on $R$, and even less a
homomorphism of $\g_R$-modules.  Still, we have  some good behaviour:
\begin{prop}\label{lem-ineq} Let $M$ be a $\g_R$-module of finite type.
  \begin{enumerate}
  \item The functor $F$ is exact and faithful.
\item 

$      l_{\g_R}(M)\leq l_{\g_k}(k\otimes_R M).$

  \end{enumerate}
\end{prop}
Simple examples show that $F$ is not a full functor.
\begin{lemma}\label{0-lemma}
  If $N$ is a $\g_R$-submodule of $\mf_RM$, then $N=0$.
\end{lemma}
\begin{proof}
  Since $R$ is a simple $\g_R$-module the $\g_R$-module $M/N$, being of finite type over $R$, is fee
  over $R$ \Proposition{\ref{freelemma}}. Hence the exact sequence $0 \to N \to M \to M/N \to 0$ is
  a split and therefore the map $k\otimes_R N \to k\otimes_R M$ is injective.  Since $N \subset
  \mf_R M$ this map is $0$, hence $N= \mf_R N$, so that by Nakayama's lemma, $N=0$.
\end{proof}

\begin{pfof}{\Proposition{\ref{lem-ineq}}}
  (1): That $F$ is exact follows since the $\g_R$-modules are free \Prop{\ref{freelemma}}. Having
  proved exactness we now prove faithfulness. A
  morphism $\phi \in Hom_{\g_R}(M_1 , M_2) $ maps to a morphism $\bar \phi \in Hom_{\g_k}(
  k\otimes_R M_1, k\otimes_R M_2)$, and $\bar \phi =0 $ means $\phi (M_1 )\subset \mf_R M_2$.   Now
  the assertion follows from \Lemma{\ref{0-lemma}}.

  (2): Let $M_1 $ be a proper submodule of $M$. Since $F$ is exact, Nakayama's lemma implies that
  $k\otimes_R M_1 $ is a proper submodule of $k\otimes_R M$. From this observation the assertion
  immediately follows.
\end{pfof}

We are interested in modules whose structure is well reflected in its corresponding fibre module.
    \begin{definition}\label{local-system}
      Let $\g_R$ be a Lie algebroid such that $R$ is simple over $\g_R$. A {\it local system} over
      $\g_R$ is a $\g_R$-module $M$ which is of finite type as $R$-module, such that
      \begin{displaymath}\label{local-sys-def}
        l_{\g_R}(M)=  l_{\g_k}(F(M)).
      \end{displaymath}
      Let $\Loc(\g_R)$ denote the category of local systems over $\g_R$.
   \end{definition} 
   This notion is important for computing Hilbert series of $\g_R$-modules
   in \Section{\ref{loc-syst-sec}}. The seemingly trivial case when $R=k$, so $\g_R = \g_k$ and
   therefore all $\g_R$-modules of finite type belong to $\Loc(\g_R)$, is important in the study of
   analytic algebras \Sec{\ref{hypersurfaces}}.  Other concrete examples will be
   worked out later: criteria to have local systems are in \Propositions{\ref{local-cond}}{
     \ref{loc-ideal}}; the complex analytic case is described in \Proposition{\ref{complete-prop}};
   localisations of representations of Lie algebras, \Theorem{\ref{localisation}}; monomial ideals,
   \Section{\ref{toral-section}}.

   \begin{remark}
     On a complex manifold $(X, \Oc_X)$ with tangent sheaf $T_{X/\Cb}$ a $T_{X/\Cb}$-module $M$ is
     the same as an integrable connection, and its sheaf of horizontal sections
     $\Lc= M^{T_{X/\Cb}} $ forms a local system in the usual sense.  Locally near a point $x$ the
     local system it is determined by its fibre, isomorphic to $M/\mf_x M$, which can be regarded as
     a vector space over the trivial fibre Lie algebra $\g_{\Cb}= 0$ of $\g_R = T_{X/\Cb,x}$, with
     $R= \Oc_{X,x}$. Thus ``locally'' on $X$ local systems in the usual sense correspond to a local
     systems in the sense of \Definition{\ref{local-sys-def}}, due to Cauchy's theorem for
     integrable connections.
   \end{remark}

   % \begin{remark}
   %   If $M$ is a finitely generated $\g_R$-module, but not necessarily finite over $R$, we can say
   %   it is a local system if $\ell_{\g_k}(H^\bullet (k\otimes^L_R M)) = \ell_{\g_R}(M)$, where
   %   $H^\bullet (k\otimes^L_R M)$ is the homology of the derived tensor product.  To make the
   %   exposition more digestable we decided to stay away from derived categories.
   % \end{remark}
   Since all $\g_{R}$-subquotients of $M$ are free over $R$
   \Prop{\ref{freelemma}} it is straightforward to see that $M\in
   \Loc(\g_R)$ if and only if each simple subquotient of $M $ induces a
   simple $\g_k$-subquotient of $k\otimes_R M$.
  % \begin{remark}
   %   Although the fibre functor $F: \Loc (\g_R)\to \Mod(\g_k)$, $F(M)=
   %   k\otimes_R M$ preserves length it is in general not fully
   %   faithful, and it does not have an adjoint functor.
   % \end{remark}
    \begin{proposition}\label{short-ex}
Assume that $R$ is simple over the Lie algebroid $\g_R$ and let
\begin{displaymath}
  0\to M_1 \to M \to M_2 \to 0
\end{displaymath}
be an exact sequence of $\g_R$-modules of finite type over $R$. Then
$M\in \Loc (\g_R)$  if and only if  $M_1, M_2\in \Loc (\g_R)$.
    \end{proposition}
    \begin{proof}
      Since $M,M_1,M_2$ are free over $R$, hence flat, \Proposition{\ref{lem-ineq}} implies
      \begin{displaymath}
        \ell_{\g_R}(M)\leq \ell_{\g_k}(k\otimes_RM) =
        \ell_{\g_k}(k\otimes_R M_1) + \ell_{\g_k}(k\otimes_R M_2) \geq
        \ell_{\g_R}(M_1) + \ell_{\g_R}(M_2).
    \end{displaymath} 
    Since $ \ell_{\g_R}(M)= \ell_{\g_R}(M_1) + \ell_{\g_R}(M_2)$, if one of the inequalities is an
    equality, then the other is too, so that $\ell_{\g_k}(k\otimes_R M_1) + \ell_{\g_k}(k\otimes_R
    M_2) = \ell_{\g_R}(M_1) + \ell_{\g_R}(M_2)$; as moreover $ \ell_{\g_k}(k\otimes_R M_i) \geq
    \ell_{\g_R}(M_i)$, $i=1,2$, the assertion follows.
    \end{proof}

    A simple $T_{R}$-module, finitely generated over $R$, is a local system if and only if it is of
    rank $ 1$ as $R$-module, but still there exist a great many non-isomorphic modules in
    $\Loc(T_R)$.
    \begin{example}\label{exp-example}
      Let $R= \Cb[x]_{(x)}$ be the localisation of the polynomial ring at the maximal ideal $(x)$ so
      $T_R = R\partial_x$ and $\Dc_R$ is the localisation of the first Weyl algebra at $(x)$.  The
      fibre Lie algebra of $\g_R= T_R$ is $0$.
      \begin{enumerate}
      \item The $T_R$-module $M= \Dc_R/\Dc_R (\partial^2 + x)$ is simple of rank $2$ over $R$
        \cite{macconnell-robson:weylsimple}*{Prop. 5.18}, while the fibre module is just a
        2-dimensional vector space, so it is not simple.
      \item  The $T_R$-module $M_p= R e^p $, $p\in R$, is a
        a simple $T_R$ -module \Prop{\ref{length-prop}, (3)}; the exponential generator is
        interpreted symbolically by the relation $\partial_x \cdot e^p = \partial_x (p) e^p$.  We
        have $M_p \cong M_q$ if and only of there exists $\phi \in R$ such that $(\partial_x
        -\partial_x(p))\phi e^q =0$, i.e. $\partial_x (\phi)/\phi = \partial_x(p-q)$. Such a $\phi$
        exists if and only of $p-q \in \Cb$. Replacing $R$ by the ring of convergent power series
        $\Cb[[x]]$, then all modules $M_p$ are isomorphic.
      \end{enumerate}

    \end{example}
    For general Lie algebroids it is a non-trivial task to decide if a given $\g_R$-module belongs
    to $\Loc(\g_R)$, but we do have some general conditions ensuring this. Before discussing these
    we need more notation.

    Let $\Dc(\g_k)$ and $\Dc(\g)$ be the universal enveloping algebra of the fibre Lie algebra and
    the kernel $\g$ in (\ref{ex-seq}). Let $\Dc(\g_R)$ be the enveloping ring of differential
    operators of a Lie algebroid $\g_R$ (see \cite{bei-ber:jantzen}*{1.2.5}); one can, when
    considering a $\g_R$-module $\rho : \g_R \to \cf_R(M)$ \Sec{\ref{sec-modules}}, instead of
    having to construct $\Dc(\g_R)$, consider its image in $\End_k (M)$, which coincides with the
    $k$-algebra that is generated $\rho(\g_R)$ and $i(R)$ in $ \End_k(M)$.  

    We can extend the observation before \Example{\ref{exp-example}}. Say that a $\g$-module $M$ is
    cyclic if there exists an element $m\in M$ such that $M = \Dc(\g) m$.
    \begin{lemma} \label{cyclic} Assume that $M\in \Loc(\g_R)$. Each simple subquotient of $M$ is
      cyclic over its restriction to $\g\subset \g_R$.  If $M$ is simple and $m\in M \setminus \mf_R M$, then
      $\Dc(\g) m = M$.
    \end{lemma}
    For any $\g_R$-module $M$ of finite type over $R$ such that $k\otimes_k M$ is simple over
    $\g_k$, it follows that $M$ is simple over $\g_R$ \Prop{\ref{lem-ineq}, (2)} and that moreover
    $M = \Dc(\g)m$ when $m\in M \setminus \mf_R M$.  I can see no reason, however, why simple
    $\g_R$-modules that are cyclic over $\g$ need have simple fibre modules.

    \begin{proof}
      If $L$ is a simple subquotient of $M$, then $L\in \Loc(\g_R)$ \Prop{\ref{short-ex}}, so that
      $k\otimes_R L$ is a simple $\g_k$-module.  Therefore, if $l\in L\setminus  \mf L$, then
      \begin{displaymath}
        k\otimes_R L = \Dc(\g_k)    1\otimes l = 1\otimes \Dc(\g) l,       
      \end{displaymath}
      so that $ L \subset \Dc(\g)l + \mf_R L$, which by Nakayama's lemma implies that $\Dc(\g)
      l = L$.  
    \end{proof}

    \begin{prop}\label{local-cond} Let $\g_R$ be a Lie algebroid where $R$ is
      simple over $\bar \g_R$, and $M$ be a $\g_R$-module which is of finite type over $R$. Assume
      that $M$ contains a finite-dimensional $k$-subspace $V\subset M$ that generates $M$ as
      $R$-module, and that $\g_R$ contains elements $\{\delta_1, \dots , \delta_s\}\subset \g_R$
      whose image $\{\alpha(\delta_1), \dots , \alpha(\delta_s)\}$ generates the Lie algebroid $\bar
      \g_R\subset T_R$, and satisfying the following condition: 
      \begin{displaymath}\tag{L}
        \delta_i v \in   \Dc(\g)v + \mf_R M, \quad v \in V,\quad i=1, \dots , s. 
      \end{displaymath}

 Then $M\in \Loc (\g_R)$. 
\end{prop}

\begin{proof} Let $0 \to M_1 \to M \to M_2 \to 0$ a short exact sequence of $\g_R$-modules, where
  $M$ contains a finite-dimensional generating subspace $V$ satisfying \thetag{L}.  If $V'$ is
  another $k$-subspace of $M$ that generates $M$ over $R$, by Nakayama's lemma we have
  \begin{displaymath}
    V' \equiv    V \mod \mf_R M,
  \end{displaymath}
  so that $V'$ will also satisfy \thetag{L}.   Therefore, if $V_1 $ is a $k$-subspace of $M_1$ such
  that $M_1= R V_1$, $V_1\subset V_1' \subset M$ (identifying $M_1$ with a subset of $M$), and $M= R
  V_1'$, it follows that
\begin{displaymath}
  \delta_i v \in ( \Dc(\g)v + \mf_R M)\cap M_1 = \Dc(\g)v + \mf_R M_1  , \quad v\in V_1, \quad
  i=1, \dots , s,
\end{displaymath}
where the equal sign follows from the fact that $\Dc(\g)v\in M_1$ and $M_1$ and $M$ are
free \Prop{\ref{freelemma}}.  Therefore $V_1$ satisfies \thetag{L} for the
module $M_1$.  Also, the image $V_2$ of $V$ in $M_2$ satisfies \thetag{L} for
the module $M_2$. Since $M$ is of finite length \Prop{\ref{length-prop},(3)}, it
follows therefore from \Proposition{\ref{short-ex}} that it suffices to prove
that $k\otimes_R M$ is simple over $\g_k$ when $M$ is simple. Let $\g'_R$ be the
Lie subalgebroid of $\g_R$ that is generated by the $\delta_i$, so $\g_R = \g +
\g'_R$ (see (\ref{ex-seq})).  Let $1\otimes v\in k\otimes_R M$ be a non-zero
element, where   $v\in V$.  We then get (as detailed below)
      \begin{displaymath}
        \Dc(\g_k)(1\otimes v) = 1\otimes \Dc(\g)v = 1\otimes
        \Dc(\g)\Dc(\g'_R)v= 1\otimes \Dc(\g_R)v = 1\otimes M.
      \end{displaymath}
      The first equality follows since the $\g_k$-action on the fibre module comes by the
      specialisation map $f: \g \to \g_k$.  Since $\delta_i \cdot v = Q_iv + m'$ where $m' \in \mf_R
      M$, $Q_i \in \Dc(\g)$, we get for any $P\in \Dc(\g'_R)\subset \Dc(\g_R)$ that $Pv \equiv Q v
      \mod \mf_R M$, for some $Q\in \Dc(\g)$. As moreover the $\g$-action is $R$-linear, this
      explains the second equality.  The third equality follows from the fact that $\g$ and $\g_R'$
      together generates $\g_R$, and the last equality follows since $M$ is simple.
    \end{proof}

    Say that a Lie algebroid is {\it split}, or has an integrable connection, if the homomorphism $\alpha
    : \g_A \to \bar \g_A= \Imo (\alpha)$ has a split $\phi: \bar \g_A \to \g_A $ as $A$-modules and
    Lie algebras.  If a split $\phi$ is chosen, we write $\g_A = \g \rtimes \bar \g_A$, where $\g=
    \Ker (\alpha)$, and call it a semi-direct product of $\bar \g_A$ by $\g$.  Notice that a
    transitive Lie algebroid over an allowed regular ring $R$ always is of the form $\g_R=\g \oplus
    T_R$ as $R$-module \Prop{\ref{freelemma}}, where $T_R$ is of finite type by
    \Theorem{\ref{goodring}}, but in general $\g_R$ need not be isomorphic to $\g \rtimes T_R$ ,
    i.e.  $\g_R$ has no integrable connection.
    % However, if we are studying $\g_R$-modules $M$ of finite type over $R$, so $M$ is free over $R$
    % \Prop{\ref{freelemma}}, then its associated Lie algebroid $\cf_{R}(M)\cong \cf_R(R^n)$, some
    % integer $n$, has an integrable connection (for instance the map $T_R \to \cf_R(R^n) $
    % corresponding the direct sum of the $T_R$-module $R$ by itself $n$ times), although such a
    % connection need not arise from a connection of $\g_R$. More precisely, given the integrable
    % connection $\nabla: T_R \to \cf_{R/k}(M)$ there need not exist a map $\hat \nabla : T_R \to
    % \g_R$ such that $\nabla = \rho \circ \hat \nabla$, where $\rho : \g_R \to \cf_{R}(M)$ is the
    % action of $\g_R$ on $M$.

    We have the following basic examples of split Lie algebroids, and some other ones are in
    \Lemma{\ref{splitlemma}}.
\begin{proposition}\label{complete-prop} Assume that $R$ is either a
   formal power series ring over $k$ or a ring of convergent power
   series over $k= \Cb$ or $k= \Rb$. Let $\g_R$ be a transitive Lie
   algebroid, i.e. $\alpha(\g_R)= T_{R}$, and $M$ be a $\g_R$-module
   of finite type over $R$. Then
    \begin{enumerate}
  \item  $\g_R \cong \g \rtimes T_{R}$.
  \item $M= RM^{T_{R}}= R\otimes_k M/\mf_R M $, where after a choice of split $T_R \to \g_R$, giving
    an isomorphism in (1), the subalgebroid $T_{R}$ acts trivially on $M/\mf_R M$. Hence $M\in \Loc
    {\g_R}$.
  \end{enumerate}
\end{proposition}
\begin{proof} (1): In the sequence (\ref{ex-seq}), where $\bar \g_R = T_{R}$, all $R$-modules are
  free since $R$ is simple over $T_R$ \Prop{\ref{freelemma}} , so there exists an $R$-linear split
  $\phi_0: T_{R}\to \g_R$. It gives rise to a map $\omega: \wedge^2 T_{R}\to \g$ defining a
  2-cocycle in the de~Rham complex $\Omega^\bullet_{R}(\g)$ (the curvature), by $\omega (\delta
  \wedge \eta )= [\phi_0(\delta), \phi_0(\eta)] - \phi_0([\delta, \eta])$, $\delta, \eta \in
  T_R$. Since Poincar\'e's lemma holds when $R$ is either complete or a ring of convergent power
  series, there exists an $R$-linear map $\eta :T_{R}\to \g$ such that $\omega = d \eta$. The map
  $\phi= \phi_0 -\eta$ then defines an integrable connection on $\g_R$.

  (2): This is standard (see \cite[Th. 1.1.25]{bjork:analD }), but we still include a proof. Let $(x_1,
  \dots ,x_d)$ be a regular system of parameters of $\mf_R$, and let
  $\partial_{i}$ be an $R$-basis of $T_{R}$ such that
  $\partial_i(x_j)= \delta_{ij}$.  We contend that there exist free
  generators $m_i$ of $M$ such that $\partial_{i}\cdot m_j =0$. We
  know that $M$ is free over $R$, so let $m' =(m_1', \dots , m'_n)$ be
  a column matrix of free generators, so $\partial_d m = \Phi m$ for
  some $n\times n$-matrix with coefficients in $R$. Write
  \begin{displaymath}
    \Phi = \sum_{i\geq 0} x_d^i\Phi_i
  \end{displaymath}
  where the $\Phi_i$ are independent of $x_d$, and put
  \begin{displaymath}
    \psi = \exp(-\sum_{i\geq 0} \frac{x^{i+1}_d}{i+1}\Phi_i).
  \end{displaymath}
  This series converges to a nonsingular matrix when $R$ is the ring
  of convergent power series $k\{x_1, \dots, x_d \}$ ($k=\Rb$ or $k=
  \Cb$) or the ring of formal power series $ k[[x_1, \dots , x_d]]$,
  since the exponent is divisible by $x_d$. Putting $m= \Psi m'$ we
  get our basis of $M$, satisfying $\partial_d m =0$, i.e.  we have
  free generators of $M$ that belong to $M^{\partial_d}$. The space
  $M^{\partial_d}$ is a module over the Lie algebroid $T_{R^0}$ where
  $R^0 = R^{\partial_d} = k\{x_1, \dots , x_{d-1}\} $ (resp. $k[[x_1,
  \dots , x_{d-1}]]$). Putting $N = \sum R^0 m_i \subset
  M^{\partial_d}$ we have a $T_{R^0}$-module of finite type over the
  $T_{R^0}$-simple module $R^0$. An induction over $d= \dim R$
  completes the proof that $M$ has free $T_{R}$-invariant generators.
\end{proof}
% \begin{example} When $R$ is incomplete, then the conclusion of
%   \Proposition{\ref{complete-prop}} is false. Put $R= k[x,y]$,
%   $I=(y)$, $A= R/I$, $\g_R = R \delta $, where $\delta = \partial_x +
%   ay\partial_y$, $a\neq 0$, $a\in k[x]\subset R$, so $\g_R \subset
%   T_R(I)$. The $A$-Lie algebroid $\g_A= \g_R/I\g_R = A \overline
%   \delta$-module is transitive ($\alpha : \g_A \to T_A $, $\overline
%   \delta \mapsto \partial_x$, is an isomorphism), and $I/I^2 $ is a
%   $\g_A$-module with action $\overline \delta p(x) = p'(x) + ap(x)$,
%   when $p\in k[x]\cong I/I^2$.  Hence $I/I^2 \cong A_1/(A_1(\partial_x
%   -a))$, where $A_1$ is the first Weyl algebra in the variable
%   $x$. But $ A_1/A_1(\partial_x -a)$ is isomorphic to $k[x]$ only as
%   $k[x]$-module, and not as $A_1$-module.
% \end{example}
\begin{remark}
  Let $M$ be a $\g_A$-module of finite length and of finite type over
  the allowed local ring $A$. The completed $\hat A$-module $\hat M =
  \hat A \otimes_A M$ is then a module over the completed Lie
  algebroid $\g_{\hat A} = \hat A \otimes_A \g_A$. In general, the
  length increases upon completion: $\ell_{\g_A}(M) \leq
  \ell_{\g_{\hat A}}(\hat M)$.
\end{remark}
% We put the following well-known lemma here for reference, for the lack
% of a better place.
% \begin{lemma} Let $\nf$ be a maximal nilpotent Lie subalgebra of a
%   Levi subalgebra in a Lie algebra $\g_k$ over an algebrically closed
%   field $k$ of characteristic $0$.  If $V$ is a finite-dimensional
%   $\g_k$-module, then
%   \begin{displaymath}
%     l_{\g_k}(V) = \dim_k V^{\nf},
% \end{displaymath}
% where $V^{\nf}= \{v \in V \ \vert \ \nf_k \cdot v =0\} $.
%   \end{lemma}

\subsection{Localisations of representations} \label{Localisations of representations} Given a Lie
algebra $\afr$ over $k$ and a homomorphism of $k$-Lie algebras $\alpha : \afr \to T_R $, we get a
Lie algebroid $\g_R = R\otimes_k \afr$, where $[r_1\otimes \delta, r_2\otimes \eta ] = r_1
\alpha(\delta)(r_2)\otimes \eta - r_2 \alpha(\eta) (r_1)\otimes \delta + r_1r_2\otimes [\delta,\eta
]$, $r_1, r_2 \in R$, $\delta, \eta \in \afr$. Define also $\alpha_R : \g_R \to T_R$, $r\otimes
\delta \mapsto r \alpha(\delta)$. We first give conditions that ensure that $\g_R$ be split.
\begin{lemma}\label{splitlemma} If there exists a Lie subalgebra $\nf \subset \afr$ such that the induced map $\nf\to
  k\otimes_R T_R$ is an isomorphism, then $\g_R \cong \g \rtimes T_R$.
\end{lemma}
\begin{proof}
  The assumption together with Nakayama's lemma implies that $R\otimes_k \nf \cong T_R$, as
  $R$-modules.  Define $\phi : T_R \to \g_R$ by $r\alpha( \delta) \to r\otimes \delta $, $r\in R$,
  $\delta \in \nf$, which is a well-defined $R$-linear homomorphism such that $\alpha_R \circ \phi =
  \id$. Since $[\nf, \nf]\subset \nf$ it follows that $\phi$ satisfies
  \begin{displaymath}
    \phi([ \alpha(\delta), \alpha(\delta')]) = \phi (\alpha ([\delta, \delta '])),
  \end{displaymath}
  which implies that $\phi$ is a split.
\end{proof}
If $V$ is an $\afr$-module, then $\Delta(V)= R\otimes_k V$ is a $\g_R$-module, such that $r\otimes
\delta (r'\otimes v) = r\alpha(\delta)(r')\otimes v + rr'\otimes \delta (v)$, $r,r'\in R$, $v\in
V$. This defines an exact functor $\Delta : \Mod(\afr)\to \Mod(\g_R)$, which we call the {\it
  localization functor}.  Let $\Mod_{fd}(\afr)$ be the category of finite-dimensional
$\afr$-modules.
\begin{remark}
  Notice that although $F(\Delta(V))= k\otimes_R \Delta(V) $ is isomorphic to $V$ as $k$-space, the
  left hand side is a $\g_k$-module, while $V$ is defined to be an $\afr$-module. For example, if
  $\afr$ is a complex simple Lie algebra, $X=G/B$ its associated Borel variety, and $R= \Oc_{X,x}$
  for a point $x$ in $X$, then $\g_\Cb$ is the kernel of the map $\afr \to \Cb \otimes_R T_{X/\Cb,x}$;
  i.e. $\g_\Cb$ is isomorphic to the Borel subalgebra of $\afr$ that stabilizes $x$.
\end{remark}
In general $\Delta(V)$ need not be of finite length even if $V\in \Mod_{fd}(\afr)$ when $R$ is
non-simple over $\g_R$, so it is satisfying to know, on the other hand, that if $R$ contains no
non-trivial $\afr$-invariant ideals, it is in fact a local system.

    \begin{theorem}\label{localisation} Let $\afr$ be a
      finite-dimensional Lie algebra over $k$, $R$ be a regular allowed $k$-algebra, and $\alpha:
      \afr\to T_R$ be a homomorphism of Lie algebras such that $R$ has no $\afr$-invariant proper
      ideals.  If $V $ is a finite-dimensional $\afr$-module, then
   \begin{displaymath}
     \ell_{\g_R}(\Delta(V)) = \ell_{\g_k}(V),
   \end{displaymath}
   so that we have a functor $\Delta : \Mod_{fd}(\afr)\to \Loc (\g_R)$, which is exact and faithful.
 \end{theorem}
 \Theorem{\ref{localisation}} can be regarded as a weak version of the well-known equivalence
 between the category of $G$-equivariant sheaves on a homogeneous space $G/H$ and finite-dimensional
 $H$ -modules, where $H$ is a closed subgroup of an algebraic group $G$ (defined over $k$).  In most
 important examples of \Theorem{\ref{localisation}} the Lie algebroid $\g_R$ is transitive, so that
 $ \bar g_R = \Imo (\alpha : \g_R \to T_R ) = T_R$, but we do not need this assumption in the proof,
 essentially because of \Lemma{\ref{splitlemma-1}}.

 To get a rough idea of the proof assume for simplicity that $\g_R$ is split, so that we can regard
 $T_R$ as a Lie subalgebroid of $\g_R$, as in \Lemma{\ref{splitlemma}}, although this is not a
 requirement.  Then the idea is to find a basis $B$ of the $R$-module $\Delta(V)$ that is adapted to
 a composition series of $F(\Delta(V))$, and which moreover satisfies $T_R \cdot B \subset \mf_R \Delta(V)$; compare
 also to \Proposition{\ref{local-cond}}.

 \begin{pfof} {\Theorem{\ref{localisation}}}
   It suffices to prove that $M= \Delta(V)$ is a local system when $V$ is simple over $\afr$
   \Prop{\ref{short-ex}} and define $\bar \g_R$ as in (\ref{ex-seq}).  Let $\{\delta _1,\dots,
   \delta_r\}\subset \afr$ be a minimal subset such that $\bar \g_R = \sum_{i=1}^r
   R\alpha(\delta_i)$.  Since $\bar \g_R$ is a free $R$-module \Prop{\ref{freelemma}} and $T_R$ is
   free of rank $\dim R $ \Th{\ref{goodring}}, it follows that $r \leq \dim R$ ($r= \dim R$ if $\bar
   \g_R= T_R$).  Since $R$ is simple over $\bar \g_R$ it follows from \Lemma{\ref{splitlemma-1}} that
   there exists a subset $\{x_1, \dots , x_r \}\subset \mf_R$ such that $\alpha(\delta_i)(x_j) =
   \delta_{ij} + \phi_{ij}$, where $\phi_{ij}\in \mf_R$ (the Kronecker symbol $\delta_{ij}$ clearly
   is not an element in $\g_R$). Therefore,
   \begin{displaymath}
     [\delta_i ,1- x_j \delta_j] = -(\delta_{ij} + \phi_{ij})\delta_i- x_j [\delta_i, \delta_j] =
     -\delta_{ij}\delta_i + \eta_{ij}, \quad \eta_{ij} \in \mf_R \g_R.
   \end{displaymath}
   Putting $\Lambda= \prod_{i=1}^r (1-x_i\delta_i)\in \Dc(\g_R)$, we get  
\begin{align*}\tag{C}
  \delta_i\cdot \Lambda &= \Lambda \cdot \delta_i + \sum_{s=1}^r( \prod_{1\leq j \leq s-1}(1-
  x_{j}\delta_{j}) ) [\delta_i, 1- x_s\delta_s ] ( \prod_{s+1\leq j \leq r}(1-
  x_{j}\delta_{j}) )\\
  & = \Lambda \cdot \delta_i + \sum_{s=1}^r( \prod_{1\leq j \leq s-1}(1- x_{j}\delta_{j}) ) (
  -\delta_{is}\delta_i + \eta_{is}) ( \prod_{s+1\leq j \leq r}(1-
  x_{j}\delta_{j}) )\\
  & = \Lambda \cdot \delta_i - \delta_i + \eta_i, \quad \eta_i \in \mf_R \Dc(\g_R),
\end{align*}
where for the last step we have made use of the observation
\begin{align*} [(1-x_j\delta_j), -\delta_{is}\delta_i + \eta_{is} ] &= \delta_{is}[x_j\delta_j,
  \delta_i] - [x_j \delta_j , \eta_{is}] \\ &=\delta_{is} (x_j [\delta_j, \delta_i] - (\delta_{ij}
  + \phi_{ij}) \delta_j) - [x_j \delta_j , \eta_{is}]
\\ &= \delta_{sj}\delta_j + \epsilon_{is},\quad    \epsilon_{is}\in \mf_R \g_R.
\end{align*}

Let $ W_s \subset \cdots \subset W_{i+1} \subset W_i \subset \cdots \subset W_1\subset W_0 = F(M)$
be a composition series of the $\g_k$-module $F(M)$ and select a basis of $F(M)$ that is compatible
with the composition series. This basis can be lifted to a basis $B'= (v'_1, \dots ,
v'_n)$ of $M$, and we let $\tilde W_i \subset M$ be the $R$-submodule that is generated by the part
of the basis $B'$ that corresponds to a basis of $W_i$. We then have
\begin{displaymath}\tag{I}
  \Dc(\g) \tilde W_i \subset \tilde W_i + \mf_R M.
\end{displaymath}
Since $\Lambda \cdot v'_i \equiv v'_i \mod \mf_R M$ it follows that $B= (\Lambda v_i')=
(v_1, \dots , v_n)$ is also a basis of $M$, and by \thetag{C} 
\begin{displaymath}\tag{K}
  \delta_i \cdot v_j \in \mf_R M, \quad i=1, \dots , r, j=1, \dots , n.
\end{displaymath}
Set now $ L_i = \Lambda \tilde W_i$, yielding a filtration of $R$-modules $L_s \subset\cdots \subset
L_{i+1}\subset L_i\subset \cdots \subset L_1 \subset L_0 = M$. Let $\g'_R$ be the Lie subalgebroid
of $\g_R$ that is generated by the $\delta_i$.  We have now (as detailed below)
\begin{displaymath}\tag{*}
  \Dc(\g_R)  L_i = \Dc(\g)\Dc(\g'_R) L_i \subset  \Dc(\g) (L_i + \mf_RM)  \subset \Dc(\g) \tilde W_i + \mf_R M \subset
  \tilde W_i + \mf_R M.
\end{displaymath}
The equality follows since $\g_R= \g + \g'_R$ and the first inclusion follows from \thetag{K}. The
second inclusion follows since the $\g$-action on $M$ is $R$-linear, so that $\Dc(\g)\mf_R M \subset
\mf_R M$, and since $L_i \subset \tilde W_i + \mf_RM$. The last  inclusion follows from
\thetag{I}.

The inclusion \thetag{*} implies that $\Dc(\g_R)L_{i+1}\neq \Dc(\g_R) L_i$ when $W_{i+1} \neq
W_{i}$, which implies $\ell_{\g_R}(M)\geq \ell_{\g_k}(V)$. Together with
\Proposition{\ref{lem-ineq}}, (2), this completes the proof that $M\in \Loc(\g_R)$.
 \end{pfof}

 Here is a concrete computational illustration for the Lie algebra $\Sl_2$, where we do not lift the
 composition series of the fibre module in the same way as in the proof of
 \Theorem{\ref{localisation}}.
   \begin{example}\label{loc-rep}
     Let $R= k[x]$ be the polynomial ring of one variable over $k$. The Lie algebra $\afr= \Sl_2(k)
     = kX_- \oplus kH \oplus kX_+$, where $[X_+, X_- ]= H$, $[H, X_\pm] = \pm 2 X_\pm$, acts on $R$ by
     the map $\alpha : \g \to T_R $, $\alpha(H)= 2 x\partial_x$, $\alpha(X_+)=x^2\partial_x$ and
     $\alpha(X_-) =-\partial_x$. Then $\g_R = R\otimes_k \afr$ is a transitive Lie algebroid over
     $R$, $\g := \Ker (\g_R\to T_R) = R(x\otimes H - 2 \otimes X_+) + R(1\otimes H +2x\otimes X_-)$,
     and we have the integrable connection $T_R \to \g_R$, $a\partial_x \mapsto a\otimes X_-$, so
     $\g_R = \g \rtimes T_R$. The fibre Lie algebra $\g_k = k\otimes_R \g $ is 2-dimensional, hence
     it is a solvable Lie algebra.  If $V$ is a finite-dimensional $\afr$-module, we have the
     $\g_R$-module $M= \Delta(V)$.  By \Theorem{\ref{localisation}} 
     \begin{displaymath}
       \ell_{\g_R}(M) = \ell_{\g_k}(V)  = \dim_k V,
   \end{displaymath}
   so that each simple subquotient of $M$ is of rank $1$ over $R$, and $M$ is a local system over
   $\g_R$. Let us compute the composition series of $M$ when $V$ is simple. The Cartan algebra $\hf
   = k H$ gives a weight decomposition $M= \oplus_{\lambda \in \Zb} M_\lambda$, where $H m = \lambda
   m$ when $m\in M_\lambda$.  Let $\lambda_0$ be the lowest integer such that $M_{\lambda_0}\neq 0$,
   so $\dim_k V_{\lambda_0+2i} =1$ when $i=0,1,\dots , -\lambda_0$, and $\dim_k V =- \lambda_0+1$.
   We have $M_\lambda = \oplus_{\mu +i = \lambda} x^i\otimes v_\mu$, so that $\dim M_\lambda - \dim
   M_{\lambda -2} =1$ when $\lambda = \lambda_0 +2i $, $i=1 , 2, \dots ,- \lambda_0$, and $\dim
   M_\lambda = \dim M_{-\lambda_0} $ when $\lambda = -\lambda_0 + 2i$, $i= 0, \dots $.  It follows
   that there exists a nonzero vector $m_{\lambda_0 +2i} \in M_{\lambda_0 + 2i}$ such that
   $(1\otimes X_-) m_{\lambda_0 +2i} =0$, when $i=0, \dots , -\lambda_0$. In fact, putting
   $m_{\lambda_0} = 1\otimes _kv_{\lambda_0}$, where $v_{\lambda_0}$ is a basis of $V_{\lambda_0}$,
   and
   \begin{displaymath}\tag{*}
 m_{\lambda_0 +2i} = (1\otimes X_+ - (\lambda_0 +2(i-1)) x) m_{\lambda_0 + 2(i-1)} , \quad i=1,
   \dots , -\lambda_0,
 \end{displaymath}
we get
 \begin{displaymath}
     m_{\lambda_0 +2i} = [\prod_{j=1}^i (1\otimes X_+ -(\lambda_0 + 2(j-1))  x\otimes 1)]m_{\lambda_0}.
   \end{displaymath}
   % Define the $\g_R$-module $R \mu_i$, by putting $(1\otimes X_-) \mu_i =0$, $(1\otimes H)\mu_i =
   % (\lambda_0 + 2(i-1)) \mu_i$ and $(1\otimes X_+) \mu_i =\frac 12(\lambda_0 +2(i-1)) x \mu_i$.
   % Since $R$ is simple, it follows that $R\mu_i $ is a simple $\g_R$-module.    
% We have the isomorphism
% \begin{displaymath}
%   \bigoplus_{i=1}^{\dim V}  R \mu_i \to  R\otimes_k V, \quad \sum P_i \mu_i \mapsto  \sum P_i m_{\lambda_0 +2i},
% \end{displaymath}
% so that $R\otimes V$ is  a semi-simple $\g_R$-module. The map is clearly surjective, and it is
% injective since the ranks agree. 
   This results in a filtration by $\g_R$-modules
   \begin{displaymath}
     \Dc(\g_R) m_{-\lambda_0}\subset \cdots \subset      \Dc(\g_R) m_{\lambda_0 +2i}  \subset        \Dc(\g_R)
     m_{\lambda_0 +2(i-1)}\subset \cdots     \subset  \Dc(\g_R) m_{\lambda_0 }
     = M.
  \end{displaymath}
  The module $\Dc(\g_R) m_{-\lambda_0}$ is simple and the remaining successive quotients are
 \begin{displaymath}
   \frac{\Dc(\g_R) m_{\lambda_0 +2(i-1)}}  {\Dc(\g_R) m_{\lambda_0
       +2i}} \cong R\mu_i , \quad i= 1, \dots , -\lambda_0,
 \end{displaymath}
 where $\mu_i = m_{\lambda_0 +2(i-1)} \mod \Dc(\g_R) m_{\lambda_0 +2i} $ generates a module wich is
 cyclic over $R$ and satisfies the equations $(1\otimes X_-) \mu_i =0$, $(1\otimes H)\mu_i =
 (\lambda_0 + 2(i-1)) \mu_i$ and $(1\otimes X_+) \mu_i =(\lambda_0 +2(i-1)) x \mu_i$ (see \thetag{*}).  Since $R$ is
 simple, it follows that $R\mu_i $ is also simple as $\g_R$-module. This indeed results in a composition
 series of $M$ of length $-\lambda_0 +1$.
\end{example}

  \subsection{Ideals of  definition}\label{def-ideal-section}
  When $M$ is a $\g_A$-module and the length $\ell_{\g_A}(M)= \infty$ it is interesting to study the
  length function $n \mapsto \ell_{\g_A}(M/I^{n+1}M)$ for certain ideals $I$ of $A$.  For this
  purpose we make the following fundamental definition.
\begin{definition}
  Let $(A/k, \g_A)$ be a Lie algebroid and $M$ a $\g_A$-module. An ideal of definition relative to
  $M$ is an ideal $J\subset A$, $J \neq A$, such that $\alpha(\g_A)\subset T_{A}(J)$ and
  $\ell_{\g_A}(M/JM)< \infty$.  A {\it defining ideal} is an ideal of definition relative to the
  $\g_A$-module $A$.
\end{definition} 
If $\alpha (\g_A)=0$, then an ideal $J$ of definition (relative to $A$) is the same as the
``classical'' one, that $\mf_A^l \subset J \subset \mf_A$ for some integer $l \geq 1$ for some
maximal ideal $\mf_A$.  That defining ideals relative to $A$ exist follows from Zorn's lemma, and in
fact there exist defining ideals for any $\g_A$-module of finite type over $A$.
\begin{proposition}\label{defining-module} Let $M$ be a $\g_A$-module
  of finite type over $A$.
  \begin{enumerate}
  \item If $J$ is an ideal of definition relative to $A$, then $J$ is
    an ideal of definition relative to $M$.
  \item If $J$ is an ideal of definition relative to $M$, then
    $\ell_{\g_A}(M/J^{n+1}M)< \infty$ for any positive integer $n$.
 \end{enumerate}
 \end{proposition}

 \begin{proof} (1): Put $B= A/J$ and $\g_B= B\otimes_A \g_A$, so $\g_B$ is a Lie algebroid over $B$
   and since $J$ is a defining ideal, $B$ is of finite length over $\g_B$.  Then apply
   \Proposition{\ref{length-prop}}, (4), to the $\g_B$-module $M/JM$.

   (2): Put $J_1 = \Ann (M/JM) $ and $A_1 = A/J_1$. By \Proposition{\ref{length-prop}}, (5),
   $\ell_{\g_{A}}(A_1)< \infty$. Since $J^iM/J^{i+1}M$ and $A_1$ are modules over
   $\g_{A_1}=\g_A/J\g_A$, it follows by \Proposition{\ref{length-prop}}, (4), that
   $\ell_{\g_A}(J^iM/J^{i+1}M) = \ell_{\g_{A_1}}(J^iM/J^{i+1}M) < \infty$, $i=0,1,2,\dots $,
   implying $\ell_{\g_A}(M/J^{n+1}M) < \infty$, $n= 0,1,2, \dots $.
\end{proof}

\begin{lemma} Let $I$ and $ J$ be $\g_A$-stable ideals in $A$.
  \begin{enumerate}
  \item If $I \subset J$ and $I$ is defining, then $J$ is also defining.
  \item If $I$ is a defining ideal, then $I^n$ is a defining ideal for any positive integer $n$.
  \item If $J$ is defining and $J^n \subset I$ for some positive integer, then $I$ is also defining.
    If $\sqrt{I}= \sqrt{J}$, then $I$ is defining if and only of $J$ is defining.
  \item If $A$ is a local ring, there exists a unique maximal defining ideal.
  \end{enumerate}
  \end{lemma}
  (If $\Char A >0$, then $\sqrt{I}$ need not be preserved by $\g_A$ even if $I$ is preserved.)
  \begin{proof}
    (1): Evident. (2): This follows from \Proposition{\ref{defining-module},(2)}. (3): Combine (1)
    and (2). (4): Just notice that the sum of any maximal defining ideals belongs to $\mf_A$ so is
    not equal to $A$.
  \end{proof}
  When a $\g_A$-module of finite type is not finitely generated over $A$ there need not exist a
  defining ideal.  For example, the Weyl algebra $M= k\left< x, \partial_x\right>$ is a module over the Lie
  algebroid $T_{k[x]/k}$ (by left multiplication) that lacks a defining ideal. On the other hand, if
  $\ell_{\g_A}(M)< \infty$ then any preserved ideal is defining, and if $\ell_{\g_A}(A)< \infty$ and
  $M$ is of finite type over $A$, then $\ell_{\g_A}(M)< \infty$ (see
  \Proposition{\ref{length-prop}}, (4)).
\begin{proposition}\label{reg-def-ideal}
  Let $A/k$ be an allowed ring and $\g_A$ a Lie algebroid over
  $A$.
  \begin{enumerate}
  \item If $J$ is a maximal defining ideal, then $A/J$ is a regular ring, hence if $A$ is also a
    regular ring, then $J= (x_1, \dots , x_r)$, where $x_1, \dots , x_r$ are elements in $\mf_A$
    that form a subset of a regular system of parameters of $A$.
  \item If $J$ is a $\g_A$-defining ideal in $A$, then its radical $J_m= \sqrt{I}$ is the unique
    maximal defining ideal containing $J$.
\end{enumerate}
\end{proposition}

\begin{proof}
  (1): By \Proposition{\ref{regularprop}}, $A/J$ is a regular ring. If
  $A$ is also regular it follows that $J$ is generated by subset of a
  regular system of parameters.
 
  (2): This follows since $A/J/\nil (A/J)$ is a simple module over $\g_A$ \Prop{\ref{regularprop}}.
\end{proof}

\section{Hilbert series}\label{hilb-section}
In this section we will first discuss commutative graded algebras over Lie algebras, and their
graded modules, leading up to rational generating functions for the lengths of the homogeneous
components of a module. This is then applied to modules over Lie algebroids.

\subsection{Representation algebras}\label{rep-alg} We refer to \citelist{\cite{bourbaki-lie-Ch1}*{\S 6}
  \cite{bourbaki-lie-ch7-9}*{\S 6-7}} for unexplained notions  in this section pertaining to Lie algebras.

Let $\g_k$ be a finite-dimensional Lie algebra over the algebraically closed field $k$. By a graded
$\g_k$-algebra we mean a graded Noetherian commutative $k$-algebra $S^\bullet= \oplus_{i\geq 0} S^i
$ which at the same time is a $\g_k$-module by a homomorphism of $k$-Lie algebras $\g_k \to
T_{S^\bullet/k}$, such that $\g_k \cdot S^i \subset S^i $.  A graded $(S^\bullet, \g_k)$-module is a
graded $S^\bullet$-module and $\g_k$-module $M^\bullet = \oplus_{i\in \Zb} M^i$ such that $\delta
\cdot M^i \subset M^i$ and $\delta (sm)= \delta (s)m + s \delta (m)$, $\delta \in \g_k$, $m\in
M^\bullet$. We let $\Mod(S^\bullet, \g_k)$ be the category of graded $(S^\bullet, \g_k) $-modules
that are of finite type over $S^\bullet$.

Let $\rf$ be the radical of $\g_k$, $\sfr= [\g_k, \rf] = [\g_k, \g_k] \cap \rf \subset \rf$ the
nilpotent radical (see \cite{bourbaki-lie-Ch1}*{\S 5, Th 1, \S 6, Prop. 6}), and put
\begin{displaymath}
  Q= (\frac{\rf}{\sfr})^* = (\frac {\g_k}{[\g_k, \g_k]})^*
\end{displaymath}
where the isomorphism is induced from the inclusion $\rf \subset \g_k$.  Here the character group
$\Ch(\g_k)= Q$ is a subgroup of the character group $\Ch (\rf)= (\rf/[\rf, \rf])^*$, where the group
structure is by addition of dual vectors.  The group algebra $k[Q]$ is the set of functions $Q \to
k$ which has the value $0$ for almost all points in $Q$; its elements are commonly described by the
expressions $\sum a_qq$, designating the function that maps $q$ to $a_q$.  The algebra $k[Q]$ is
also an $\rf$-module by defining $r \cdot \sum a_qq = \sum a_qq( r)q $ and since every element $q\in
Q$ satisfies $q(\sfr)=0$, so $q$ extends to a homomorphism of Lie algebras $\g_k \to k$, it follows
that $k[Q]$ is also a $\g_k$-module. Since $Q$ parametrizes the isomorphism classes of simple
1-dimensional $\g_k$-modules we get that $k[Q]$ is a $\g_k$-algebra which is a semi-simple
$\g_k$-module such that every simple 1-dimensional $\g_k$-module occurs with multiplicity $1$.

Let $\Lc_k= \g_k/\rf$ be the semi-simple quotient and $\hf \subset \bfr_k $ a Cartan algebra and
Borel subalgebra of $\Lc_k$.  Let $P_{++}\subset \hf^*$ be the set of integral dominant weights,
which is a commutative sub-semigroup of $\hf^*$, and $\{\omega_1, \dots , \omega_l\}$ be the set of
fundamental weights.  For a weight $\phi \in P_{++}$ we let $L_\phi$ be the simple
finite-dimensional $\Lc_k$-module of heighest weight $\phi$.  Put $L = \oplus_{i=1}^l L_{\omega_i}$
and let $S^\bullet(L)$ be its symmetric algebra, which is a graded $\Lc_k$-algebra.  The length of a
weight $\phi= \sum_{i=1}^l m_i \omega_i\in P_{++} $ is $l(\phi)= \sum_{i=1}^l m_i$.  In each
degree $n$, the weight space $S^n(L)_{\phi}$ is 1-dimensional when $l(\phi)=n$.  Therefore there
exists an $\Lc_k$- submodule $J^n\subset S^n(L)$ such that
\begin{displaymath}
S^n(L)/J^n \cong \bigoplus_{l(\phi) =n}  L_{\phi}.
\end{displaymath}
Moreover, the $\Lc_k$-submodule $S^n(L) J^m$ is contained in $ J^{n+m}$, since $S^n(L) J^m $
contains no weight vectors of length $ n+m$. Therefore $J^\bullet= \oplus_{m\geq 1} J^m \subset
S^\bullet(L) $ is an $\Lc_k$-invariant ideal, hence $\Rc_{\Lc_k}^\bullet:= S^\bullet(L)/J^\bullet$
is an $\Lc_k$-algebra, which clearly is graded and noetherian. By construction,
\begin{displaymath}
  \Rc^\bullet_{\Lc_k}= \bigoplus_{n\geq 0} \bigoplus_{l(\phi)=n} L_\omega  = \bigoplus_{\phi \in P_{++}}L_{\omega}
\end{displaymath}
as $\Lc_k$-module, so that it
contains every simple finite-dimensional $\Lc_k$-module with multiplicity $1$.  We note also that
$\Rc_{\Lc_k}^\bullet$ is a $\g_k$-algebra such that $\rf \cdot \Rc_{\Lc_k}^\bullet =0$.  The tensor
product
\begin{displaymath}
  \Rc^\bullet_{\g_k}= k[Q]\otimes_k \Rc_{\Lc_k}^\bullet
\end{displaymath}
of the $\g_k$-algebras $k[Q]$ and $\Rc^\bullet_{\Lc_k}$ is then a semi-simple $\g_k$-algebra, where
each simple $\g_k$-module has multiplicity $1$.  We call $\Rc^\bullet_{\g_k}$ the representation
algebra of $\g_k$.

\begin{remark}\label{rep-remark}
  \begin{enumerate}
  \item Let $\g_k$ is a split semi-simple Lie algebra. The construction of the representation
    algebra is then due to Cartan, and therefore the map $L_\omega\otimes L_{\omega'} \to L_{\omega
      + \omega'} $ that arises from the ring structure is called ``Cartan multiplication'' in
    $\Rc_{\g_k}^\bullet$. Generators and relations for $\Rc_{\g_k}^\bullet$ are described by Kostant
    \cite{lancaster-towber:1}.  A version for Lie groups was introduced in \cite{hadziev} under the
    name ``universal algebra''. Bernstein, Gelfand, and Gelfand
    \cite{gelfand-gelfand-bernstein:models} and Gelfand and Zelevinsky
    \cite{gelfand-zelevinsky:models} have made concrete realisations of representation algebras
    (a.k.a. ``models'') for different Lie groups.
  \item Assume that $\g_k$ is semi-simple and $\Gamma $ be a sub-semigroup of $P_{++}$. Then
    $\Rc^\bullet_{\Gamma} = \oplus_{\gamma \in \Gamma} L_{\gamma}$ is naturally a sub $\g_k$-algebra
    of $\Rc^\bullet_{\g_k}$, and such algebras have natural geometric interpretations.  Let for
    example $\Gamma = \{n \lambda \ \vert \ n= 1,2, \dots \}$, where $\lambda \in P_{++}$.  Let $G$
    be a semi-simple algebraic group over $k$ and $\g_k$ be its Lie algebra.  Then
    \begin{displaymath}
      \Proj  \Rc^\bullet_\Gamma = G/P,
    \end{displaymath}
    where $P$ is a parabolic subgroup of $G$ \cite{ramanan-ramanathan:proj-norm}.  Representation
    algebras are studied geometrically in \cite{ramanathan-kempf}.
  \end{enumerate}

\end{remark}

\subsection{Algebras over solvable Lie algebras}\label{toruscase}
Let $(S^\bullet, \g_k)$ be a noetherian graded $\g_k$-algebra, where
$\g_k$ is a finite-dimensional solvable Lie algebra. Let $\Ch(\g_k)=
(\g_k/[\g_k,\g_k])^*$ be the character group, and if $\chi \in
\Ch(\g_k) $, we put $S^\bullet_{\chi} = \{s\in S^\bullet \ \vert \ (X-
\chi (X))^n s =0, n\gg 1, X\in \g_k\}$.  Similarly we define
$M^\bullet_\chi$ when $M^\bullet\in \Mod(S^\bullet, \g_k)$, and put
$\supp_{\g_k} M^\bullet= \{\chi \in \Ch(\g_k)\ \vert\
M^\bullet_{{\chi}}\neq 0\}$.  It is interesting to find noetherian
subalgebras $S_1 \subset S^\bullet$ and $S_1$-submodules of finite type
$M_1\subset M^\bullet$. The theorem below is one way to achieve this
using subsets of  $C= \supp_{\g_k} S^\bullet$ and $C_M=\supp_{\g_k}
M^\bullet$; this will be applied  in \Theorem{\ref{hilb-liealg}} to get
new rational Hilbert series from a given one.

Note that $C$ is a commutative sub-semigroup of $\Ch(\g_k)$, where the
binary operation is induced by the ring structure of $S^\bullet$, and
we have
  \begin{eqnarray*}
    S^\bullet &=& \bigoplus_{\chi \in C  } S^\bullet_{\chi},\\
    M^\bullet &=& \bigoplus_{\phi \in C_M} M^\bullet_\phi.
  \end{eqnarray*}

\begin{theorem}[\cite{kallstrom:two-extensions}]\label{solv-theo}
  Let $S^\bullet $ be a noetherian graded $\g_k$-algebra, where $\g_k$
  is a solvable Lie algebra, and $M^\bullet \in \Mod(S^\bullet,
  \g_k)$.
  \begin{enumerate}
  \item Let $\Gamma $ be a sub-semigroup of $C$, and put $\Gamma^c =
    C\setminus \Gamma$.  Then
    \begin{displaymath}
      S^\bullet_\Gamma = \bigoplus_{\chi \in \Gamma} S^\bullet_{\chi}
    \end{displaymath}
    is a graded $\g_k$-algebra, and if moreover $\Gamma + \Gamma^c \subset
    \Gamma^c$, then $S^\bullet_\Gamma$ is noetherian.
  \item Let $\Gamma $ be a sub-semigroup of $C$, $\Phi$ be a subset of
    $C_M$, and put $\Phi^c = C_M \setminus \Phi$. Consider the
    conditions:
  \begin{enumerate}
  \item $\Gamma \cdot \Phi \subset \Phi$,
  \item $ \Gamma^c \cdot \Phi \subset \Phi^c$.
  \end{enumerate}
  Then $(a)$ implies that
\begin{displaymath}
  M_\Phi^\bullet = \bigoplus_{\phi \in \Phi } M^\bullet_\phi
\end{displaymath}
is a graded $(S^\bullet_\Gamma, \g_k)$-module. If also (b) is
satisfied, then $M^\bullet_\Phi$ is of finite type over
$S_\Gamma^\bullet$.
\end{enumerate}
\end{theorem}

We give a hint at how to prove \Theorem{\ref{solv-theo}}. Since $S^\bullet$ is graded
noetherian there exists an integer $r$ such that the graded $ S_\Gamma^0$-submodule of $V =
\oplus_{i=1}^r S_\Gamma^i\subset (S_\Gamma^\bullet)_+$ generates $S^\bullet (S_\Gamma)_+^\bullet$ over
$S^\bullet$. Let $B^\bullet$ be the subalgebra of $ S_\Gamma^\bullet$ that is generated by $V$ and $
S_\Gamma^0$, so in particular $B^i = S_\Gamma^i$ when $0\leq i \leq r$; clearly,  $B^\bullet$ is
noetherian. The proof is then   by showing that  $B^\bullet = S^\bullet_\Gamma$.
\begin{remark}
  \begin{enumerate}

  \item Let $(S^\bullet, \rf_k)$ be an $\rf_k$-algebra with $\rf_k$ a solvable Lie
    algebra. % Let $\Gamma =
    % \Gamma^1/\Gamma^0 $ is a subquotient of $C= \supp _{\bar \bfr_k}
    % S^\bullet$, where
    Let $\Gamma^0 \subset \Gamma^1$ be sub-semigroups of $C$. Then $S^\bullet_{\Gamma^0}\subset
    S^\bullet_{\Gamma^1}$ are $ \rf_k$-subalgebras of $(S^\bullet, \rf_k)$, and
    $S_{\Gamma^1/ \Gamma^0}^\bullet := S^\bullet_{\Gamma^1}/S^\bullet_{\Gamma^0} = \oplus_{\chi \in
      \Gamma^1 \setminus \Gamma^0} S^\bullet_\chi$ is again an $\rf_k$-algebra.

    We can generalize \Remark{\ref{rep-remark}, (2)}.  Let, as in \Section{\ref{rep-alg}}, $\rf_k$ be the radical of a Lie algebra $\g_k$. If $\Gamma \subset Q\times P_{++} $ is a
    sub-semigroup we put
    \begin{displaymath}
      \Rc^\bullet_\Gamma = \oplus_{\gamma \in \Gamma} L_\gamma \subset \Rc^\bullet_{\g_k},
    \end{displaymath}
    which a $\g_k$-subalgebra. More generally, if $\Gamma^0 \subset \Gamma^1 \subset Q\times P_{++}$
    are two sub-semigroups, then
    \begin{displaymath}
      \Rc^\bullet_{\Gamma^1/\Gamma^0} = \frac{      (\Rc^\bullet_{\g_k})_{\Gamma^1}} {(\Rc^\bullet_{\g_k}) _{\Gamma^0}}
    \end{displaymath}
    is a $ \g_k$-algebra such that the multiplicty $[\Rc^\bullet_{\Gamma^1/\Gamma^0} : L_\phi] =1$
    when $\phi \in \Gamma^1 \setminus \Gamma^0 $, and otherwise $[\Rc^\bullet_{\Gamma^1/\Gamma^0} :
    L_\phi]=0$.
  \end{enumerate}
\end{remark}
\subsection{Hilbert series over solvable Lie algebras} 
Let $\Zb[C]$ be the group ring of the monoid $C$, and $\Zb[C][t]$ the polynomial ring over $\Zb[C]$.
Let $\Zb[C_M][t]$ be the set of functions $f: C_M \times \Nb\to \Zb$ which take the value $0$ for
almost all points in $C_M\times \Nb$. We write $f= \sum_{\phi, n} k_{\phi, n} \phi t^n$ (finite sum)
for the function that maps $(\phi, n) \in C_M \times \Nb$ to the integer $k_{\phi,n}$.  We regard
$f$ as a polynomial, though we do not have a counterpart of polynomial products since $C_M$ is not a
monoid.  Similarly, let $\Zb[C][[t]]$ be the ring of formal power series with coefficients in the
ring $\Zb[C]$, and $\Zb[C_M][[t]]$ be the set of all functions $C_M \times \Nb \to \Zb$. The action
of $C$ on $C_M$ gives $\Zb[C_M][t]$ ($\Zb[C_M][[t]]$) a structure of $\Zb[C][t]$-module
($\Zb[C][[t]]$-module), so that $m \chi t^{n'} \cdot f = \sum m k_{\phi, n} (\chi \cdot \phi )
t^{n+n'}$.

% As usual,
% say that $H\in \Zb[C_M][t] $ is divisible by $f\in \Zb[C][t]$ if there
% exists an element $G \in \Zb[C_M][t]$ such that $H= f \cdot G$.

Let $\g_k$ be a solvable Lie algebra, $S^\bullet$ a Noetherian graded $\g_k$-algebra, 
$M^\bullet\in \Mod(S^\bullet, \g_k)$, and $C_M$ and $C$ be defined as in (\ref{toruscase}). The
equivariant Hilbert series of $M^\bullet$ is
\begin{displaymath}
  H^{eq}_{M^\bullet}(t)=\sum_{n\geq n_0, \chi   \in C_M} \dim_k
  (M^n_\chi)  \chi t^n\in
  \Zb[C_M][[t]], 
\end{displaymath}
where $n_0$ is an integer such that $M^n=0$ when $n < n_0$ (such an integer exists since $M^\bullet$
is a module  of finite type over a Noetherian graded ring).
\begin{lemma}\label{lemma-hilb}
  Let $M^\bullet\in \Mod(S^\bullet, \g_k)$.  Then 
  \begin{displaymath}
    H^{eq}_{M^\bullet}(t) = \frac{f(t)}{\prod_{\chi \in R} (1-\chi t^{n_\chi})^{d_\chi}},
\end{displaymath}
for some finite subset $R\subset C$ and integers $d_\chi >0$, and $f\in \Zb[t]$. The integers
$n_{\chi}$ are determined by choosing generators of $S^\bullet$ that belong to $S^{n_\chi}_\chi$.
\end{lemma}
\begin{proof} 
  Proceed as in the ordinary proof of the rationality of Hilbert series of finitely generated graded
  modules over commutative noetherian graded algebras, by induction over the number of generators,
  which are required to be homogeneous relative to the grading $C_M\times \Nb$ (see
  \cite{matsumura}*{Th. 13.2} for the ordinary case, and for torus actions, see
  \cite{renner:torus}).
\end{proof}
\begin{remark}
  Replacing the solvable Lie algebra by a torus, Renner
  \cite{renner:torus} interpreted the integers $d_\chi$ in the
  rational presentation of $H^{eq}_{S^\bullet}(t)$ in terms of the
  geometry of $X=\Proj S^\bullet $ and its line bundle $\Oc_X(1)$, and
  also related it to the geometry of $R\subset C$.
\end{remark}
Let $\Gamma \subset C$ be a sub-semigroup and $\Phi \subset C_M$ a
subset such that the natural $\Gamma$-action on $C_M$ preserves
$\Phi$, $\Gamma \cdot \Phi \subset \Phi$, so we have the
$(S^\bullet_\Gamma, \g_k)$-module $M^\bullet_\Phi$ (see
\Theorem{\ref{solv-theo}}).  We define the $\Phi$-Hilbert series
\begin{displaymath}
  H^\Phi_{M^\bullet} (t) = \sum_{n\geq n_0} \dim_k (M_\Phi^n)t^n \in \Zb[[t]].
\end{displaymath}
Recall that $\dim_k (M_\Phi^n) = \ell_{\g_k} (M_\Phi^n) $ since $k$ is
algebraically closed of characteristic $0$, and $\g_k$ is solvable.

\begin{lemma}\label{hilb-torus} Assume that $\Gamma\subset C$ and
  $\Phi\subset C_M$ satisfy the conditions in
  \Theorem{\ref{solv-theo}}, (1) and (2). Then
  \begin{displaymath}
    H^\Phi_{M^\bullet} (t) = \frac {g(t)}{\prod_{i=1} ^d(1-t^{n_i})}.
  \end{displaymath}
  The integers $n_i$ are determined by the degrees of a choice of
  homogeneous generators of $S_\Gamma^\bullet$.
\end{lemma}
\begin{proof}
  By \Theorem{\ref{solv-theo}} $S_\Gamma^\bullet$ is noetherian and
  $M_\Phi^\bullet$ is finitely generated. Then the result follows from
  Hilbert's theorem \cite[Th. 13.2]{matsumura}.
\end{proof}

For any finite subset $\Omega \subset C_M$ we have a summation map $\int : \Zb[\Omega]\to \Zb$,
$\sum a_\phi \phi \mapsto \sum a_\phi$. It induces a map, which we denote the same, $\int : \Zb
[C_M][[t]]\to \Zb[[t]]$.

\begin{remark} Clearly $H^{\Phi}_{M^\bullet} (t) = \int (
  H^{eq}_{M^\bullet_{\Phi}}(t))$, but we can see no apparent other reason
  why $\int$ should map the rational function $
  H^{eq}_{M^\bullet_\Phi}(t)$ to a rational function of the same form.
\end{remark}

\subsection{ Case of general Lie algebras} We keep the notation in \Section{\ref{toruscase}}, so
$(S^\bullet, \g_k)$ is a graded Noetherian $\g_k$-algebra, $\g_k$ is a finite-dimensional Lie
algebra, and $M^\bullet\in \Mod(S^\bullet, \g_k)$.  Let $\bar \bfr_k$ be a maximal solvable Lie
subalgebra of $\g_k$, and for a subset $\Phi \subset \supp_{\bar \bfr_k} M$ we put
\begin{displaymath}
  \ell^{\Phi}_{\g_k}(M^n) = \sum_{\phi \in
    \Phi} [M^n: L_\phi],
\end{displaymath}
where $[M^n: L_\phi]$ is the multiplicity of the simple $\g_k$-module $L_\phi$ whose highest weight
is $\phi$.
\begin{remark}
  We can write $\bar \bfr_k = \rf \oplus \bfr_k$, where $\bfr_k$ is a maximal solvable Lie
  subalgebra of a Levi subalgebra of $\g_k$. This gives an inclusion
  $Q\oplus \hf_k^* = (\rf_k/\sfr_k )^* \oplus \hf_k^* \subset (\rf_k/[\rf_k, \bar \bfr_k])^* \oplus
  \hf_k^*= (\bar \bfr_k/[\bar \bfr_k, \bar \bfr_k])^* = \Ch(\bar \bfr_k)$
  (see \Section{\ref{rep-alg}}), and the weights $\phi\in \Phi$ that give a non-zero contribution to
  the sum above in fact belong to the subset $Q\oplus \hf_k^*$.
\end{remark}

The following theorem is crucial for proving the rationality of Hilbert series
of local systems \Th{\ref{main}}, quoted in the introduction to this paper.

\begin{theorem}\label{hilb-liealg} Let $S^\bullet$ be a Noetherian
  graded $\g_k$-algebra and $M^\bullet$ be graded module over $(S^\bullet, \g_k)$, which is of
  finite type over $S^\bullet$. Let $\Gamma$ be a sub-semigroup of $ C = \supp _{\bar
    \bfr_k}S^\bullet$ and $\Phi \subset C_M=\supp_{\bar \bfr_k} M^\bullet $ a subset satisfying the
  conditions in \Theorem{\ref{solv-theo}}, (1) and (2).

  Then the Hilbert series
  \begin{displaymath}
    H_{M^\bullet}^{eq, \Phi}(t) =  \sum_{n\in \Zb, \phi \in \Phi} [M^n: L_\phi] \phi t^n \in
    \Zb[\Phi] [[t]]
\end{displaymath}
and  
 \begin{displaymath}
       H_{M^\bullet}^\Phi(t) =  \int (    H_M^{eq, \Phi}(t))=
       \sum_{n\in \Zb } \ell^{\Phi}_{\g_k}(M^n) t^n\in
    \Zb [[t]]
\end{displaymath}
  are rational functions of the form
  \begin{eqnarray}\label{eq:1}
    H^{eq, \Phi}_{M^\bullet}( t)&=&\frac {f^{eq}_M(\chi, t)}{\prod_{\chi\in \Xi}(1-\chi t^{n_\chi} )},\\
    \label{eq:2}    H^{\Phi}_{M^\bullet}(t)&=& \frac{f_M(t)}{\prod_{i=1}^r(1-t^{n_i})},
 \end{eqnarray}
 where the poly\-nomials $f^{eq}_M(\chi, t)\in \Zb[\Phi, t]$ and $f_M(t)\in \Zb[t]$, and $\Xi
 \subset \Gamma$ is a finite subset.
\end{theorem}

Letting $\Gamma=C$ and $\Phi =C_M$, we note in particular that the
full Hilbert series $H^{eq}_{M^\bullet}:= H^{eq, C_M}_{M^\bullet}(t)$
and the generating function $H_{M^\bullet}(t) :=
H^{C_M}_{M^\bullet}(t)$ of the lengths $\ell_{\g_k}(M^n)$ are rational
functions.

\begin{remark}
  The assertion (\ref{eq:1}) was noted in \cite{renner:torus} when $\g_k$ is a semi-simple Lie
  algebra and $\Phi = P_{++}$.
\end{remark}
\begin{proof} 
  There exists a Levi subalgebra $\Lc_k \subset \g_k$, and a Borel algebra $\bfr_k \subset \Lc_k$
  such that $\bar \bfr_k = \rf + \bfr_k$.  Put $\nf_k = [\bfr_k,\bfr_k ]$. 

  We first make a remark: Recall that the length $\ell_{\g_k}(M)$ of a finite-dimensional
  $\g_k$-module $M$ equals $\dim_k M^{\nf_k}$. Now the $\nf_k$-invariant space $(M^n)^{\nf_k}$ of
  the $\g_k$-module $M^n$ is a $\bfr_k$-module, but need not be an $\rf$-module when the nilpotent
  radical $\sfr_k\neq 0$ and $M^n$ is not semi-simple. Therefore the $((S^\bullet)^{\nf_k},
  \bfr_k)$-module $(M^\bullet)^{\nf_k}$ is not provided with a structure of $ \bar \bfr_k$-module,
  which would be needed for applying \Lemma{\ref{hilb-torus}}.  We therefore need to eliminate the
  disturbance caused by the presence of a non-trivial nilpotent radical.

  For this purpose we construct a filtration of the $(S^\bullet, \g_k)$-module $M^\bullet$ by
  $(S^\bullet, \g_k)$-submodules, which induces a $\g_k$-module composition series in each
  homogeneous component $M^n$.  Such a filtration can be constructed inductively.  There is an
  integer $n_0$ such that $M^{n}=0$ when $n < n_0$.  Let $\cdots \gamma^{i+1}_{n_0}\subset
  \gamma^{i}_{n_0}\subset \cdots$ be a composition series of the $\g_k$-module $M^{n_0}$ and
  $\Gamma^{i,\bullet}_{n_0} \subset M^\bullet$ be the $(S^\bullet, \g_k)$-submodule that is
  generated by $\gamma^i_{n_0}$, so $\Gamma^{i,n_0}_{n_0}= \gamma^i_{n_0}$.  Beware that we will
  below reuse the index $i$ several times to simplify the notation.  Let $\gamma^i_{n_0+1}$ be a
  refinement of the filtration $\Gamma^{i, n_0+1}_{n_0}$ into a composition series of $M^{n_0+1}$,
  and $\Gamma_{n_0+1}^{i, \bullet}\subset M^\bullet$ be the refinement of the filtration
  $\Gamma_{n_0} ^{i, \bullet}$ that is determined by $\gamma^i_{n_0+1}$.  Then $ \Gamma_{n_0+1}^{i,
    \bullet}\subset M^\bullet$ is a filtration by $(S^\bullet, \g_k)$-submodules which induces a
  composition series of the $\g_k$-modules $M^{n_0}$ and $M^{n_0+1}$. Inductively, knowing the
  filtration $\Gamma_{n_0+r}^{i, \bullet}$ we refine $\Gamma_{n_0+r}^{i, n_0+r+1}$ into a filtration
  $\gamma_{n_0+r+1}^{i}$ of $M^{n_0+r+1}$, and let $\Gamma_{n_0+r+1}^{i, \bullet}$ be the
  $(S^\bullet, \g_k)$-filtration which is the refinement of $\Gamma_{n_0+r}^{i, \bullet}$ that is
  determined by $\gamma^i_{n_0+r+1}$.  The associated graded $(S^\bullet, \g_k)$-module
\begin{displaymath}
  \bigoplus_{n\geq n_0} \bigoplus_{i} \frac {\Gamma_{n_0 +  r}^{i,n } }{\Gamma_{n_0 +
      r}^{i+1, n}}
\end{displaymath}
is semi-simple over $\g_k$ in degrees $n \leq n_0 +r$.  Moreover, the
sequence of filtrations $r \mapsto \Gamma_{n_0 + r}^{i, \bullet}$ of
$M^\bullet$  has the property that for a  fixed $n$   and for all   $r$ such that $n\leq n_0+r$ we
have $\Gamma_{n_0 + r}^{i, n } = \Gamma_{n}^{i, n } $, so we can put
  \begin{displaymath}
    G^n(M^\bullet)  =    \bigoplus_i  \frac {\Gamma_{n}^{i,n } }{\Gamma_{n}^{i+1, n}},
  \end{displaymath}
  where thus on the right hand we have the common module $\oplus_i \frac {\Gamma_{n_0 + r}^{i,n }
  }{\Gamma_{n_0 + r}^{i+1, n}}$ for all $r$ such that $n \leq n_0 +r$.  Put
  \begin{displaymath}
      G^\bullet =   \bigoplus_{n\geq n_0} \bigoplus_{i} \frac {\Gamma_{n}^{i,n } }{\Gamma_{n}^{i+1, n}} = \bigoplus_{n \geq n_0}  G^n(M^\bullet).
  \end{displaymath}
% \begin{displaymath}
%   G^\bullet =   \lim_{r\to \infty}  \bigoplus_{n\geq n_0} \bigoplus_{i} \frac {\Gamma_{n_0 +  r}^{i,n } }{\Gamma_{n_0 +
%       r}^{i+1, n}} = \bigoplus_{n \geq n_0}  G^n(M^\bullet).
% \end{displaymath}
  Then $G^\bullet$ is a $(S^\bullet, \g_k)$-module which is semi-simple over $\g_k$, so that the
  nilpotent radical $\sfr_k$ (the intersection of the kernels of all the simple
  $\g_k$-modules) acts trivially on $G^\bullet$, i.e.  $\sfr_k \cdot G^\bullet =0$.  Let $L_\phi$ be the
  simple module corresponding to the heighest weight $\phi \in Q \times P_{++}$. Then we have by a
  heighest weight argument for short exact sequences, noting also that the multiplicity function
  $[\cdot : L_\phi ]$ is additive in short exact sequences, that
\begin{displaymath}
  [M^n: L_\phi] = [G^n: L_\phi] = \dim_k (G^n)^{\nf_k}_\phi;
\end{displaymath}
we also take notice that since $\sfr_k \cdot G^n=0$ it follows that $(G^n)^{\nf_k}$ is a
module over $\bar \bfr_k $, and we can therefore define the weight space $(G^n)^{\nf_k}_\phi$.  Let
$\bar S^\bullet$ be the image of $S^\bullet$ in $\End_k G^\bullet$. Then if $s\in \sfr_k $, $u\in
\bar S^\bullet$, and $g \in G^\bullet$, we have $(s\cdot u) g = s \cdot (ug) - u (s\cdot g) = s (ug)
= 0$, $ug \in G^n$ and $\sfr_k \cdot G^n =0 $; hence $\sfr_k \cdot \bar S^\bullet=0$, and therefore
$(\bar S^\bullet)^{\nf_k}$ is a $\bar \bfr_k$-algebra. We conclude that the Hilbert series
$H^{eq}_{M^\bullet}(t)$ coincides with the Hilbert series of the $((\bar S^\bullet)^{\nf_k}, \bar
\bfr_k )$-module $(G^\bullet)^{\nf_k}$.

Let $\Rc= \Rc_{\Lc_k}$ and $U(\Lc_k)$ be the representation algebra and enveloping algebra of
$\Lc_k$, respectively. Note that if $M$ is a finite-dimensional $\Lc_k$-module, then for any
$\Lc_k$-linear map $f: U(\Lc_k)\otimes_{U(\nf_k)} k \to M $, there exist maps $g:
U(\Lc_k)\otimes_{U(\nf_k)} k \to \Rc$ and $h: \Rc \to M $ such that $f= h\circ g$, where $g$ is
uniquely determined up to multiplication by a constant; hence
\begin{displaymath}
Hom_{\Lc_k}(U(\Lc_k)\otimes_{U(\nf_k)}k , M)= Hom_{\Lc_k}(\Rc, M).
\end{displaymath}
The graded dual $\g_k$-module $\Rc^*= \oplus_{\omega \in P_{++}} L_\omega^*$ is isomorphic to the
$\g_k$-module $\Rc$, where the latter moreover is an algebra. We have now (as detailed below)
  \begin{align*}
    ( \bar S ^\bullet)^{\nf_k}& =  Hom_{\nf_k}(k, \bar S^\bullet ) =
    Hom_{\Lc_k}(U(\Lc_k)\otimes_{U(\nf_k)}k , \bar S^\bullet ) \\
&=    Hom_{\Lc_k}(\Rc ,    \bar S^\bullet) =    (\Rc^*\otimes_k \bar S^\bullet)^{\Lc_k} =
    (\Rc \otimes_k \bar S^\bullet)^{\Lc_k}.
  \end{align*}
  The second equality is an adjunction isomorphism and the third was explained above.  Since each
  homogeneous component $\bar S^n$ is of finite dimension it contains only finitely many
  non-isomorphic $\Lc_k$-modules, so that only finitely may terms in $\Rc$ contributes to
  $Hom_{\Lc_k}(\Rc, \bar S^n)$. This implies the fourth equality.

 Similarly,
\begin{displaymath}
  (  G^\bullet) ^{\nf_k} = (\Rc \otimes_k G^\bullet)^{\Lc_k}.
\end{displaymath}
One can note that the grading of $\Rc$ induces a second grading of the invariant ring and module,
but we do not use this.  By Hilbert's theorem about invariant rings and modules\footnote{The original
argument is for groups, see \cite{hilbert-invarianttheory} and a modern treatment
\cite[Th. 3.3]{dolgachev-invariants}; the proof for semisimple Lie algebras is similar, see
\cite{kallstrom:two-extensions} which also extends this result to Lie algebroids.} it follows that
$(\bar S^\bullet)^{\nf_k}$ is a noetherian graded $ \bar \bfr_k$-algebra and $(G^\bullet)^{\nf_k} $
is a finitely generated graded $((\bar S^\bullet)^{\nf_k}, \bar \bfr_k)$-module.  Since $\Gamma$ and
$\Phi$ satisfy the conditions in \Theorem{\ref{solv-theo}} it follows also that $((\bar
S^\bullet)_\Gamma^{\nf_k}, \bar \bfr_k)$ is noetherian and $(G^\bullet)^{\nf_k}_\Phi $ is a $((\bar
S^\bullet)^{\nf_k}_\Gamma, \bar \bfr_k)$-module of finite type.  \Lemma{\ref{lemma-hilb}} now
implies that the Hilbert series $ H^{eq, \Phi}_{M^\bullet} (t) = H^{eq, \Phi}_{G^\bullet}(t) $ is of
the form that is described in equation (\ref{eq:1}).  Similarly, \Lemma{~\ref{hilb-torus}} implies
equation ~\eqref{eq:2}.
\end{proof}
\begin{remark}
  The key idea of expressing $\nf_k$-invariants as $\Lc_k$-invariants was used by Roberts
  \cite{roberts:invariants}, then again by Hadziev \cite{hadziev}, to extend Hilbert's finiteness
  theorem.
\end{remark}
\begin{example}
  Let $\Rc^\bullet$ be the representation algebra of a free sub-semigroup $ \Gamma\subset C \subset
  Q \times P_{++}$, so $\Gamma \cong \Nb^l$ for some integer $l$. If $\omega_1, \dots , \omega_l$
  are free generators of $\Gamma$, then
  \begin{eqnarray*}
    H^{eq,\Gamma}_{\Rc^\bullet}( t)& =& \sum_{n\geq 0} (\sum_{\sum_{i=1}^l
      n_i = n} \sum_{i=1}^l n_i \omega_i) t^n = \prod_{i=1}^l \frac
    1{1-\omega_i t},\\
   H^\Gamma_{\Rc^\bullet}(t)& =&  \sum_{n \geq 0} \binom{n+l-1}{l-1}t^n =  \frac 1{(1-t)^{l}}.
  \end{eqnarray*}
\end{example}
\begin{example}\label{sl2-example}  Let $\g_k = \Sl_2(k)$, so $P_{++} = \{1, 2, \dots \}$
  and the representation algebra $\Rc^\bullet_{\g_k}$ can be identified
  with the symmetric algebra $ S^\bullet =S^\bullet(V)$, where $V$ is
  the simple 2-dimensional $\Sl_2(k)$-module.  The Hilbert series of
  the $(S^\bullet, \g_k )$-module $M^\bullet_d = S^\bullet (S^d(V)) $ is
  the ordinary Hilbert series of the graded ring
  \begin{displaymath}
    (M_d^\bullet)^{\nf_k} = (M_d^\bullet  \otimes_k \Rc^\bullet)^{\g_k}
    =S^\bullet(S^d(V)\otimes_k V)^{\g_k},  
\end{displaymath}
where $\nf_k$ is a maximal nilpotent subalgebra of $\g_k$.  Hence
$(M_d^\bullet)^{\nf_k}$ coincides with the covariant algebra $C^\bullet_d
$ of binary forms of degree $d$; see \cite[\S
3.3]{springer:invarianttheory}. The coefficients of the Hilbert series
$H_{C^\bullet_d}(t)$ can be determined using the Cayley-Sylvester
decomposition of $S^n(S^d(V))$,
  \begin{displaymath}
\dim_k C^n_d   = \sum_{e\geq 0}  [S^n(S^d(V)): S^e(V)] = \sum_{e\geq 0} (p(n,d;  \frac {nd-e}{2})- p(n,d; \frac {nd-e}2 -1)),
  \end{displaymath}
  where $p(d,n; m)$ is the number of partitions of size $m$ inside the
  rectangle $d\times n$.  The rational representation of Hilbert
  series of invariant algebras is a classical problem; a formula for
  the rational function $H_{S^\bullet (W)}(t)$ for simple
  $\Sl_2(k)$-modules $W$ is presented in \cite{springer-sl2}, and
  computations for some non-simple ones can be found in
  \cite{bedratyuk:binary}.  In \cite[\S 3.4]{springer:invarianttheory}
  one can find descriptions of the algebra $C_d^\bullet$ for low
  $d$. For example:

  $C_2^\bullet = k[x_1, x_2]$ where $x_1$ and $x_2$ are algebraically
  independent elements of degree $1$ and $2$, respectively. We get
  \begin{displaymath}
    H_{M_2^\bullet} (t) = \sum_{n\geq 0} ([\frac n2] +1)  t^n=
    \sum_{n\geq 0}( \frac n2 -\frac
    12  |\cos (\frac {(n+1)\pi}2)|+1)t^n= \frac1{(1-t) (1-t^2)},
  \end{displaymath}
  where $[\cdot ]$ denotes the integer part.  %$n^2/4$.

  $C_3^\bullet = k[x_1, x_2, x_3, x_4]/(x_1^2+ x_2^3 + x_3^2 x_4)$,
  where $\deg x_1 =3$, $\deg x_2 = 2$, $\deg x_3= 1$, $\deg x_4 = 4$.
  The Hilbert series of the $(S^\bullet (V), \g_k)$-module
  $M_3^\bullet$ can now be computed:
\begin{displaymath}
  H_{M_3^\bullet}(t) =
  \sum_{n\geq 0} (\frac{n^2}8+\frac n2+ \frac 3{16} (-1)^n+\frac 14\cos(\frac{ n\pi}2)+\frac 9{16})t^n =  \frac{t^2-t+1}{(1-t)^2(1-t^4)}. 
\end{displaymath}
We thank L. Bedratyuk for providing us with the Hilbert
quasi-polynomial for $C^\bullet_3$.
\end{example}

  \subsection{ Local systems}\label{loc-syst-sec}  This section contains the main general  theorems about Hilbert series
  of local systems. Let $A$ be an allowed local $k$-algebra and $\g_A$ a Lie algebroid over
  $A$. By a $\g_A$-algebra we mean a pair $(S^\bullet, \g_A)$ where $S^\bullet$
  is a commutative graded $A$-algebra with an action of $\rho: \g_A\to
  T_{S^\bullet/k}$ such that $\g_A \cdot S^i \subset S^i$, and
  $\rho(\delta)(as)= \alpha(\delta)(a)s + a \rho(\delta)(s)$, $a\in A, s\in
  S^\bullet, \delta \in \g_A$.

 \begin{theorem}\label{loc-system-hilbert}
   Let $(S^\bullet, \g_R)$ be a graded $\g_R$-algebra where $R$ is
   simple as $\g_R$-module, and let $M^\bullet$ be a $(S^\bullet,
   \g_R)$-module of finite type over $S^\bullet$, such that $M^n=0$ when $n < n_0$. Assume that each
   $M^n \in \Loc (\g_R)$. Then the Hilbert series
  \begin{displaymath}
    H_{M^\bullet} (t) = \sum_{n \geq n_0} \ell_{\g_R}(M^n) t^n  \in \Zb[[t]]
  \end{displaymath}
is a rational function of the form
\begin{displaymath}
  H_{M^\bullet} (t) = \frac{f_M(t)}{\prod_{i=1}^r(1-t^{n_i})}
\end{displaymath}
where $f_M(t)\in \Zb[t]$.
\end{theorem}
We abstain from writing down the straightforward $\Gamma$-equivariant
generalisations that can be deduced from \Theorem{\ref{hilb-liealg}}.

\begin{proof}  Since $M^n\in \Loc(\g_R)$ we have
  \begin{displaymath}
    H_{M^\bullet} (t) = \sum_{n\geq n_0} \ell_{\g_k}(k\otimes_R M^n),
  \end{displaymath}
so  the assertion follows from \Theorem{\ref{hilb-liealg}}.
\end{proof}

Note that if $J_m$ is a maximal defining ideal in an allowed $k$-algebra $A$, and $N$ is a
$\g_A$-module of finite length, then $\ell_{\g_A}(N)= \ell_{\g_R}(G^\bullet_{J_m}(N))$, where
$G^\bullet_{J_m}(N) =\oplus_{i\geq 0} J_m^iN/J_m^{i+1}N $ and $\g_R= \g_A/J_m\g_A$. Recall also that $R$ is a
simple $\g_R$-module \Prop{\ref{reg-def-ideal}}.
\begin{theorem}\label{main}  Let $A$ be an allowed $k$-algebra and $\g_A$ be a Lie algebroid over $A$.
  Let $M$ be a $\g_A$-module of finite type as $A$-module, $J$ be a defining ideal of $A$ (and hence
  of $M$ \Prop{\ref{defining-module}}), and put
  $M^\bullet = \oplus_{n\geq 0} J^nM/J^{n+1}M$. Let $J_m=\sqrt{J}$ be the maximal defining ideal of
  $A$ and put $R= A/J_m$ and $\g_R = \g_A/J_m\g_A$, which is a Lie algebroid over $R$.  Assume that
  $G^i_{J_m}(M^n) \in \Loc (\g_R)$ for each integer $n,i =0,1,2,\dots $.
  \begin{enumerate}
  \item The Hilbert series
    \begin{displaymath}
      H_M(t) = \sum_{n\geq 0} \ell_{\g_A}(M^n)t^n \in \Zb[[t]]
    \end{displaymath}
    is a rational function of the form
    \begin{displaymath}
      H_M(t)=   \frac{f_M(t)}{\prod_{i=1}^r(1-t^{n_i})},
    \end{displaymath}
    where $f_M(t)\in \Zb[t]$.
\item  The  length function
  \begin{displaymath}
    n \mapsto \chi_M^J (n)= \ell_{\g_A}(M/J^{n+1}M)
  \end{displaymath}
  is determined by a quasi-polynomial $\phi_M^J(t)$ with integer
  coefficients for high $n$ (so $\chi_M^J (n) =\phi_M^J(n)$ for high $n$). The degree and leading coefficient of this
  quasi-polynomial are well-defined numbers, so
  \begin{displaymath}
    \phi^J_M(t) = \frac e{d!} t^d +  g(t),
  \end{displaymath}
  where $g(t)$ is a quasi-polynomial of degree at most $d-1$.
\end{enumerate}
\end{theorem}

\begin{remark}
  \begin{enumerate}
  \item We do not know if it suffices that
    \begin{displaymath}
      \frac{J_m^n (J^iM/J^{i+1}M)}{J_m^{n+1}(J^iM/J^{i+1}M) } ) =\frac {J_m^nJ^iM}{J_m^{n+1} J^iM +
      (J_m^nJ^iM)\cap (J^{i+1}M)}
  \end{displaymath}
are local systems for finitely many $n, i$ to conclude
    that they are local systems for all $n,i$.
  \item The condition that $ J_m^nJ^iM/(J_m^{n+1} J^iM + (J_m^nJ^iM)\cap (J^{i+1}M))$ are
    local systems for all $n$ and $i$ is satisfied when $A$ is an
    analytic algebra \Prop{\ref{complete-prop}}.
  \item If $M$ is not a local system, then the Hilbert series still is rational, but the proof
    requires a different approach, using a Tannakian theorem for $\g_A$-modules in connection with a
    field extension to make our $\g_A$-modules into local systems.  This will be treated in a
    separate work, where \Theorem{\ref{maintheorem}} will be extended to any local ring $A$ and
    $\g_A$-module $M$ of finite type over $A$. Still, the actual computation of the Hilbert
    quasi-polynomial is much more tractable for local systems since it is then reduced to that of a
    module of invariants over a Lie subalgebra of the fibre Lie algebra, while the general case
    involves the hard problem of computing differential Galois groups for partial differential
    equations.
  \end{enumerate}
\end{remark}

 \begin{proof}
We have
\begin{align*}
  H_M(t)& = \sum_{n\geq 0}\ell_{\g_A}(\frac {J^nM}{J^{n+1}M})t^n =  \sum_{n\geq 0}\ell_{\g_R}(
  G^\bullet_{J_m}(\frac {J^nM}{J^{n+1}M}))t^n\\ &=    \sum_{n\geq 0}\ell_{\g_R}(
  G^\bullet_{J_m}(\frac {J^nM}{J^{n+1}M}))t^n.
\end{align*}
so the assertion follows from \Theorem{\ref{loc-system-hilbert}}, for the pair $(S^\bullet= 
\oplus_{n\geq 0} J_m^n/J_m^{n+1}, \g_R)$, applied to the graded module $N^\bullet$, where $N^n =
G^\bullet_{J_m}(\frac {J^nM}{J^{n+1}M})$ (for fixed $n$ the module $G^i_{J_m}(\frac {J^nM}{J^{n+1}M}) $ is
non-zero only for finitely many $i$). It is well-known that (1) implies (2). For a discussion why
the quasi-polynomial has a well-defined degree $d$ and leading coefficient $e$, see
\cite{dichi-sangare:quasi-poly}.
\end{proof}

\begin{definition}\label{def-dimension-mult} Let $(A, \g_A)$, $M$, and $J$ be as in
  \Theorem{\ref{main}}.  The dimension $d_{\g_A}(M, J)$ and multiplicity
  $e_{\g_A}(M,J)$ of $M$ is the integer $d$ and rational number $e$ in
  \Theorem{\ref{main}}, (2).
  \end{definition}

  The number $d_{\g_A}(M, J)$ is independent of the choice of defining
  ideal $J$. To see this, let $J'$ be any other defining ideal. Then
  $J_1=J+ J'$ is a defining ideal and $J \subset J_1$.  Since
  $l_{\g_A}((J_1^{n+1}+J)/J) < l_{\g_A}((J_1^n+J)/J)$ as long as
  $(J_1^n+J)/J \neq 0$, it follows that $J_1^n \subset J \subset J_1$
  when $n\gg 1$.  Therefore there exist positive integers $a$ and $b$
  such that $J^a \subset J'$ and $J'^b \subset J$. From this the
  assertion follows. Thus we can write $d_{\g_A}(M)$, without
  reference to a choice of defining ideal.

  \begin{remark}
    It is straightforward to see that the dimension $d(M)$ and
    multiplicity $e(J,M)$ have the same properties as expounded in
    \cite[\S 13]{matsumura}.
  \end{remark}

  \subsection{Structure of fibre Lie algebras}
Recall that for a $\g_k$-module $M$ of finite dimension we have
\begin{displaymath}
  \ell_{\g_k}(M) = \dim_k M^{\nf_k} ,
\end{displaymath}
where $\nf_k$ is a maximal nilpotent subalgebra of a Levi subalgebra
of $\g_k$, and thus to compute lengths of $\g_k$-modules a first step
is to find the structure of the Levi factors of $\g_k$.

Let $\DC_s(\g_k)$ denote the derived central series of a Lie algebra $\g_k$, so $\DC_0(\g_k)=\g_k$,
$\DC_s(\g_k)= [\DC_{s-1}(\g_k), \DC_{s-1}(\g_k)]$, $s=1, \dots$, and $\g_k$ is solvable if
$\DC_s(\g_k)=0$ when $s \gg 1$. Since $k$ is algebraically closed of characteristic $0$, we have
$\ell_{\g_k}(M)= \dim_k M$ when $\g_k$ is solvable.  The following proposition shows more precisely
the relevance of knowing if the fibre Lie algebra of a Lie algebroid is solvable.  Recall that if
$\mf_A$ is stable under $\g_A$ then the maximal defining ideal is $\mf_A$, and $k=
R$ in the notation of  \Section{\ref{fibreslocal}}.

We recall that if $\g_A$ is a Lie algebroid and $I$ and $\mf_A$ both are defining ideals of the
$\g_A$-module $A$, then $\sqrt{I}= \mf_A$ \Prop{\ref{reg-def-ideal}, (2)}, and $R = A/\mf_A =k$, so
the fibre Lie algebra is $\g_k = \g_A/\mf_A \g_A$.
\begin{proposition}\label{solv-fibre} Let $\g_A$ be a Lie algebroid
  over a local ring $A$ that preserves the maximal ideal $\mf_A$. Let $I$ be a defining ideal of the
  $\g_A$-module $A$ and $M$ be a $\g_A$-module which is of finite type over $A$. If the fibre Lie
  algebra $\g_{k}= \g_A/\mf_A \g_A$ is solvable, then
  \begin{displaymath}
    \sum_{i \geq 0} \ell_{\g_A} (\frac {I^i M}{I^{i+1}M})t^i  =
    \sum_{i\geq 0} \dim_{k} (\frac {I^i M}{I^{i+1}M})t^i = \frac {f(t)}{(1-t)^{d-1}}, 
  \end{displaymath}
  where $f\in \Zb[t]$.
\end{proposition}

\begin{proof} If $N$ is an $A$-module we put $G_{\mf_A}^\bullet (N)= \oplus_{i\geq 0}
  \mf_A^iN/\mf_A^{i+1}N$.  The first equality follows from the equalities
  \begin{displaymath}
\ell_{\g_A}(\frac {I^i M}{I^{i+1}M}) = \ell_{\g_A} (G^\bullet_{\mf_A}(\frac
  {I^i M}{I^{i+1}M})) = \ell_{\g_k} (G^\bullet_{\mf_A}(\frac {I^i M}{I^{i+1}M})) = \dim_k
  (G^\bullet_{\mf_A}(\frac {I^i M}{I^{i+1}M})) .
\end{displaymath}
The last equality follows from Hilbert's theorem \cite[Th. 13.2]{matsumura}.
\end{proof}

Given a maximal defining ideal $J= J_m $ for a Lie algebroid $\g_A$ one gets the vector space
$V_J= J/\mf_AJ$, which is a module over the fibre Lie algebra $\g_k = k\otimes_R \g$, where $R= A/J$
and $\g= \Ker (\g_A/J\g_A \to T_R)$ (see \Section{\ref{fibreslocal}}). It turns out that the
$\g_k$-module $V_J$ often gives the structure of Levi subalgebras of $\g_k$.
\begin{proposition}\label{solv-kernel}
  Let $\g_A$ be a Lie algebroid over a regular allowed $k$-algebra $A$ and let $J$ be a maximal
  defining ideal of $A$.  Put $\bfr_A = \Ker (\alpha: \g_A \to T_A)$, and assume that $\bfr_k=
  \bfr_A/\mf_A \bfr_A$ is a solvable Lie algebra. Then the Levi factors of the fibre Lie algebra
  $\g_k$ of $\g_A/J \g_A$ act faithfully on $V_J= J/\mf_A J$. Hence the Levi sub-algebras of $\g_k$
  are determined by the Levi sub-algebras of the image of $\g_k$ in $\gl (V_J)$.
 \end{proposition}
 \begin{proof}
   It suffices to see that $\af_k := \Ker(\g_k \to \gl(V_J))$ is
   solvable.  Let $\af\subset \g \subset \g_A/J\g_A$ be the inverse image of $\af_k$
   for the projection $\g \to \g_k= \g/\mf_A \g$, and $\af_A$ the
   inverse image of $\af$ for the projection $\g_A \to
   \g_A/J\g_A$. Then $\af_A$ is a Lie sub-algebroid of $\g_A$ such that
   \begin{displaymath}
     \bfr_A \subset  \af_A  , \quad  \alpha(\af_A) \subset JT_A, \quad \alpha (\af_A ) (J) \subset \mf_A
     J.
   \end{displaymath}  
   Therefore,
\begin{displaymath}
  \alpha (\DC_1(\af_A))=  [\alpha(\af_A), \alpha(\af_A)] \subset \mf_AJ T_A \subset \mf_A^2T_A.
\end{displaymath}
Since $[\mf^k_A T_A, \mf^k_A T_A]\subset \mf_A^{2k-1}T_A$, and
$\cap_{i \geq 1} \mf_A^i T_A =0$ by Krull's intersection theorem, we
have
\begin{displaymath}
  \cap_{s\geq 1} \DC_s(\af_A) \subset  \bfr_A. 
\end{displaymath}
Since $\bfr_k$ is solvable there exists a positive integer $r$ such that
$\DC_r(\bfr_A)\subset \mf_A \bfr_A$.  Therefore 
\begin{displaymath}
  \cap_{s\geq 1} \DC_s(\af_A) \subset  \mf_A \bfr_A \subset \mf_A \af_A.
\end{displaymath}
Since $\dim \af_k < \infty$ it follows that there exists a positive
integer $r_1$ such that $\DC_{r_1}(\af_A) \subset \mf_A \af_A$, so
$\DC_{r_1}(\af_k)=0$.
 \end{proof}
 By \Proposition{\ref{solv-kernel}}, if $\bfr_k$ is solvable we can recognize the structure of the
 Levi factors of $\g_k$ from its image in $\gl_k(V_J)$. A number of results dealing with the problem
 of identifying Lie sub-algebras of $\gl_k(V)$, where $V$ is any finite-dimensional $k$-space, can
 be found in \cite[Ch. 1]{katz:exponential}.  The following consequence of one of these results is
 essential for determining the structure of fibre Lie algebras in the next section. 
 \begin{theorem}\label{recognition}
   Let $\g\subset \gl_k(V) $ be a Lie subalgebra with radical $\rf \subset \g$, and $\hf$ be a
   Cartan subalgebra of $\Sl_k(V)\subset \gl_k(V)$, and assume that $\hf \subset \g $.  Let
   $V_{r+1}=0 \subset V_r \subset V_{r-1}\subset \cdots \subset V_1 \subset V_0 = V$ be a
   composition series of $V$ as $\g$-module, and put $n_i= \dim_k V_i/V_{i+1}$, $i=0,1, \dots , r $
   be the dimensions of the simple subquotients. Let $F$ be the subset of indices $f$ such that $n_f
   \geq 2$. Then
   \begin{displaymath}
     \g \cong \rf\rtimes \bigoplus_{f\in F} \Sl_{n_f}.
   \end{displaymath}
\end{theorem}
\begin{proof} Put $W_i= V_{i}/V_{i+1}$, which is a simple $\g$-module, hence the nilpotent radical
  $\sfr$ acts trivially on $W_i$, and by \cite{bourbaki-lie-Ch1}*{\S 5, Th 1}, $[\g, \rf] = [\g,
  \g]\cap \rf =[\g, \rf]= \sfr$, hence the canonical homomorphism $\g\to \gl_k(W_i)$ maps $\rf$ to $
  k\id \subset \gl_k(W_i)$. Therefore there exists an injective homomorphism
  \begin{displaymath}
\g/\rf \to \bigoplus_{f\in F} \Sl_{k}(W_f).
\end{displaymath}
Let $\g_f$ be the image of $\g/\rf$ in $ \Sl_{k}(W_f)$.  Since $\hf\subset \g$, the image of $\hf$
in $\g_f$ contains a Cartan subalgebra of $\Sl_k(W_f)$. In particular, in a corresponding  weight basis of $W_f$,
the diagonal matrix $\Diag(n_f-1, -1, \dots , -1)$ is contained in $ \g_f $, and as $\g/\rf$ is
semi-simple, hence $\g_f$ is semi-simple, it follows by Kostant's theorem (see \cite[Theorem
1.1]{katz:exponential}) that $\g_f = \Sl_k(W_f)$.
\end{proof}
\section{Toral Lie algebroids}\label{toral-section} 
In (\ref{local-toric}) we first consider general ideals $I$ in an allowed regular local $k$-algebra
$A$ and Lie subalgebroids $\g_A$ of the tangential derivations $T_A(I)$, showing how some simple
assumption on the existence of certain good elements in $\g_A$ and a good basis of $I$ gives us
local systems \Prop{\ref{loc-ideal}}. Then we make the assumption that $\g_A$ contains the
derivations $x_i \partial_{x_i}$ in some regular system of parameters of $A$, again ensuring that we
get local systems and therefore rational Hilbert series and polynomial generalized Hilbert-Samuel
functions \Th{\ref{mon-the}}.  In (\ref{sec-monomial}) we switch from local to graded rings, working
instead with monomial ideals in a polynomial ring, and describe the structure of the fibre Lie
algebras in detail together with the Hilbert series of a monomial ring.  Finally in (\ref{stanley})
we compute the Hilbert series of the Stanley-Reisner ring of a simplicial complex $\Delta$ by
relating it to a certain stratification of the vertex set of $\Delta$, and show how it recovers the
face vector of a certain smaller simplicial complex $\tilde \Delta$ that can be associated with
$\Delta$.

\subsection{Local toric ideals}\label{local-toric}
Let $A$ be regular and allowed, and $I= (f_1, \dots ,f_s) \subset A$ be an ideal.  Let $\g_A\subset
T_A(I)$ and $J_m$ be a maximal defining ideal of the $\g_A$-module $A$. Then $R= A/J_m$ is a simple
module over the Lie algebroid $\g_R = \g_A/J_m\g_A$ and by \Proposition{\ref{defining-module}} it
follows that $\ell_{\g_A} (I/J_m I) < \infty $.  We ask the question when each homogeneous component
of the $\g_R$-module
\begin{displaymath}
  G^\bullet_{J_m}(I) = \oplus_{n\geq 0} \frac{J_m^nI}{J_m^{n+1}I}
\end{displaymath}
belongs to
$\Loc(\g_R)$.  If moreover $\alpha(\g_R)= T_R$ and $R$ is as in \Proposition{\ref{complete-prop}}
then $G^\bullet_{J_m}(I)$ is always a local system, but in general it is a subtle question. The
following proposition, based on \Proposition{\ref{local-cond}}, shows that if one can select the
generators in such a way that the ``differential part'' of $\g_A$ does not spread $f_i$ out
too much, then we do get local systems.

Let $\g^A= \{\delta \in \g_A \ | \ \alpha(\delta) \in J_mT_A\}$. It is the inverse image in
$\g_A$ of $\g= \Ker (\g_R \to T_R)$. Let also $\Dc(\g)^A$ be the  subalgebra of $\Dc(\g_A)$ that is
generated by $\g^A$. 

\begin{prop}\label{loc-ideal} Let $\g_A$ be a Lie algebroid and $I= (f_1, \dots , f_s)$ an  ideal in
  a regular allowed ring $A$ such that $\g_A\cdot I \subset I$. Let $\{x_1, \dots , x_n\}$ be a
  regular system of parameters for $A$ such that $J_m=(x_1, \dots , x_r)$ (see
  \Proposition{\ref{reg-def-ideal}}, (1)), and put $R= A/J_m$ and $\g_R= \g_A/J_m\g_A$.  Assume that there exist $\delta_1,
  \dots , \delta_l\in \g_A$ that map to generators of the Lie algebroid $\bar \g_R = \Imo(\g_R \to
  T_R) \subset T_R$ such that
  \begin{displaymath}\tag{L-I}
\delta_j \cdot
  f_i\in   \Dc(\g)^A   f_i \mod (x_{r+1}, \dots , x_n)+ I \cdot (x_1,
  \dots , x_r), \quad j=1, \dots , l.
\end{displaymath}
Then $G^p_{J_m}(I)\in \Loc (\g_R)$, $p= 0,1,2, \dots $. 
\end{prop}

Thus if $\g_A$ and $I$ satisfy \thetag{L-I} it follows from \Theorem{\ref{main}} that the Hilbert
series $H_{I}(t)$ of the $\g_A$-module $I$ is rational.  In particular, if $ I= (f)$ is a principal
ideal and $\g_A =T_A((f))$, then \thetag{L-I} is trivially satisfied, so the Hilbert series
$H_{(f)}(t)$ is always rational. An interesting situation is when we have ideals $I_1$ and $I$ of
$A$ such that $\g_A = T_A(I_1)\subset T_A(I)$, where e.g. $I=A$ (so that \thetag{L-I} is trivially
true) or that $I$ is some ideal that is naturally associated with $I_1$, e.g.  $I$ is the integral
closure or the Jacobian ideal of $I_1$. One may ask when the condition \thetag{L-I} is satisfied in
these two cases.

\begin{proof}
  The $R$-module $G^p_{J_m}(I)= J_m^pI/J_m^{p+1}I$ is generated by the vector space $V\subset
  G^p_{J_m}(I) $ that is generated by the elements $m_{\alpha,i} = X^\alpha f_i \ \mod J_m^{p+1}I$,
  $i=1, \dots , s$, $|\alpha| = p$, where $X^\alpha = x_1^{\alpha_1}x_2^{\alpha_2}\cdots
  x_r^{\alpha_r}$.  Since
  \begin{displaymath}
    \alpha(\delta_i) \in T_A(J_m)  = J_mT_A   + \sum_{i=r+1}^n A\partial_{x_i}
  \end{displaymath}
it follows that $\delta_j \cdot X^\alpha f_i \equiv X^\alpha \delta_j (f_i) \mod J_m I $. Hence by   \thetag{L-I},
\begin{displaymath}
 \delta_j \cdot  X^\alpha f_i \in  X^\alpha \Dc(\g)^A f_i  
\mod (J_m I + (x_{r+1}, \dots , x_n))
\end{displaymath}
Since $\mf_R = (x_{r+1}, \dots , x_n)\mod J_M$ this implies that $\bar \delta_j\cdot m_{\alpha, i}
\subset \Dc(\g) m_{\alpha, i} + \mf_R G^p_{J_m}(I)$, where $\bar \delta_i = \delta_ i \mod J_m
\g_A\in \g_R$. Hence by \Proposition{\ref{local-cond}} it follows that $G^p_{J_m}(I)\in \Loc(\g_R)$.
\end{proof}

Let $\qb_A$ be a Lie sub-algebroid of $T_A(\mf_A)$ and $\qb_k= \qb_A/\mf_A \qb_A$ its fibre Lie
algebra. Say that $\qb_A$ is a {\it weakly toral} Lie algebroid if the $\qb_k$-module $
\mf_A/\mf_A^2$ satisfies:
\begin{enumerate}\renewcommand{\labelenumi}{(\roman{enumi})}
\item $\ell_{\qb_k}(\mf_A/\mf_A^2) = \dim _k \mf_A/\mf_A^2$ $\Leftrightarrow \qb_k$ is solvable.
\item $\mf_A/\mf^2_A$ is multiplicity free, i.e. every irreducible $\qb_k$ -submodule (defined by a
  character of $\qb_k$) occurs with at most multiplicity $1$.
\item  $(\mf_A/\mf_A^2)^{\qb_k}=0$.
\end{enumerate}
It follows that if $\qb_A$ is weakly toral, then there exists a
regular system of coordinates $\{x_1, \dots , x_n\}$ such that
$\mf_A/\mf_A^2 = \oplus_{i=1}^n k \bar x_i$, where $ x_i \equiv x_i
\mod \mf_A^2$, and $k \bar x_i \ncong k \bar x_j$ as $\qb_k$-modules
when $i\neq j$. Say that $\qb_A$ is a {\it toral } Lie algebroid if
the derivations $\nabla_{x_i}= x_i \partial_{x_i}$, $i=1, \dots , n$,
are contained in $\qb_A$ for some choice of regular system of
parameters; clearly, toral Lie algebroids are weakly toral. The
advantage of the notion of weakly toral as opposed to toral Lie
algebroids, besides being more general, is of course that the former
does not refer to a choice of parameters.

Say that an ideal $I\subset A$ is {\it monomial} if there exists a
toral algebra $\qb_A$ such that $\qb_A \subset T_A(I)$. We have the
following characterisation:

\begin{proposition}\label{mon-ideal} The following are  equivalent  for an ideal $I\subset
  A$ and a toral Lie algebroid $\qb_A$:
  \begin{enumerate}
  \item $I$ is a monomial ideal with respect to the toral Lie algebroid $\qb_A$.
  \item There exists a regular system of parameters $\{x_1, \dots ,
    x_n\}$ such that $\{\nabla_{x_1}, \dots ,$ $ \nabla_{x_n}\}\subset
    \qb_A$ and $I$ is generated by monomials of the form $X^\alpha =
    x_1^{\alpha_1} \cdots x_n^{\alpha_n}$.
  \end{enumerate}
  \end{proposition}
  \begin{proof}
    $(1)\Rightarrow (2)$: Select a regular system of parameters such
    that $\{\nabla_{x_1}, \dots , \nabla_{x_n}\}$ $\subset \qb_A$. For
    any integer $r \geq 0$ we can write $f = \sum_{|\alpha| \leq r}
    f_\alpha X^\alpha + f_r$, where $f_\alpha \in k$ and $f_r \in
    \mf_A^{r+1}$. Let $cont_r (f)$ be the set of monomials $X^\alpha$
    in this expression such that $f_\alpha \neq 0$, put
    \begin{displaymath}
      I_r = (cont_r (f) \ \vert \ f\in I),   
    \end{displaymath}
    and $\hat I = \cup_{r\geq 0} I_r$. Since $A$ is noetherian, $\hat I
    = I_r $ for sufficiently high $r$, hence $\hat I$ is generated by
    certain monomials $X^\alpha$.  We assert that $I= \hat I$.

    If $f\in I$ and $\bar f$ is its projection in $B=A/\hat I$, we
    have by the definition of $\hat I$ that $\bar f\in \cap_{i \geq 1}
    \mf_B^i =\{0\}$, by Krull's intersection theorem, since $B$ is a
    local noetherian ring. Therefore $f\in \hat I$, showing that $I
    \subset \hat I$. To see that $\hat I \subset I$, let $f\in I$ and
    $X^\alpha \in cont_r(f) $ for some integer $r$.  For any integer
    $r'\geq r$ there exists a differential operator $P(\nabla)$ in the
    $\nabla_{x_i}$ such that $P(\nabla) (f) = X^\alpha + g $, where
    $g\in \mf_A^{r'+1}$.  Since $P(\nabla)(f)\in I$, it follows that
    $X^\alpha \in \cap_{r >0} ( I+ \mf_A^r)= I$, by Krull's
    intersection theorem.

    $(2)\Rightarrow (1)$: This is evident.
  \end{proof}

  Notice that although the maximal defining ideal of a weakly toral Lie algebroid $\qb_A$ of $T_A$
  is $\mf_A$, and $\qb_A \subset \alpha(\g_A)$, the maximal ideal $\mf_A$ need not be a defining
  ideal of the Lie algebroid $\g_A$.
  \begin{proposition}\label{lem-torus}
    Let $\qb_A$ be a weakly toral Lie algebroid in $T_A$ and $\g_A$ be
    a Lie algebroid such that $\qb_A \subset \alpha(\g_A)$, and assume
    that $J_{m}$ is the maximal $\g_A$-defining ideal of $A$. 

    There exists a regular system of parameters $\{x_1, \dots , x_n\}$
    such that $J_{m}= (x_1, \dots, x_r)$ and $\mf_A/\mf_A^2 =
    \oplus_{i=1}^n k\bar x_i$ as $\qb_k$-module, and the Lie algebroid
    $\g_R= \g_A/(J_{m}\g_A)$ over $R= A/J_m$ is transitive. The
    following are equivalent:
    \begin{enumerate}
    \item $x_i \in J_m$.
\item $\partial_{x_i}\not\in \alpha (\g_A) $.
    \end{enumerate}
 \end{proposition}
 \begin{proof}   
   That $J_m$ is generated by a subset of a regular system of
   parameters for $A$, as stated, follows since $A$ and $R$ are
   regular rings. As $\g_A$ preserves $(x_1, \dots , x_r)$, we get
   that $\qb_k$ acts stably on $\sum_{i=1}^r k \bar x_i\subset
   \mf_A/\mf_A^2 $.  Therefore, by a $k$-linear change of the
   parameters $\{x_{r+1}, \dots , x_n\}$ and $\{x_1, \dots, x_r \}$,
   separately, we get a new regular system of parameters which induces
   a decomposition of the $\qb_k$-module $\mf_A/\mf_A^2$.

   We prove that the image $\bar \g_R$ of the map $\g_R \to T_R$
   equals $T_R$ by considering the exact sequence $0\to \bar \g_R \to
   T_R \to T_R/\bar \g_R\to 0$.  Since $R$ is simple over $\bar \g_R$
   and $\bar \g_R$ acts with the adjoint action on the terms in the
   sequence, it follows that all the terms are free over $R$
   \Prop{\ref{freelemma}}, so $\bar \g_R = T_R$ will follow if their
   ranks agree, i.e. it suffices to see that $\rank \bar \g_R = n-r$.
   Consider now the $\qb_k$-module $\mf_A/\mf_A^2 = \oplus_{i=1}^n k
   \bar x_i$.  There exist elements $\bar \delta_1, \dots, \bar
   \delta_n\in \qb_k$ such that $\bar \delta_i \cdot \bar x_j =
   \delta_{ij}$.  Let $\delta_i \in \g_A$ be elements that map to the
   $\bar \delta_i $, and $\hat \delta_i $ be their image in
   $K(R)\otimes_R \bar \g_R$, where $K(R)$ is the fraction field of
   $R$. It suffices now to see that $B= \{\hat \delta_{r+1},\dots \hat
   \delta_{n} \}$ forms a basis of the vector space $K(R)\otimes_R
   \bar \g_R$, which will follow if it is linearly independent.  Since
   $\delta_i$ maps to $\bar \delta_i $ it follows that $\delta_i = u_i
   x_i \partial_{x_i} + \eta_i$, where $\eta_i \in \mf^2_A T_A$, and
   therefore $\hat \delta_i = \hat u_i x_i \partial_{x_i} + \hat
   \eta_i $, $i=r+1, \dots , n$, where $\hat u_i$ is a unit in $R$ and
   $\hat \eta_i \in \mf_R^2 T_R$.  Put $S= (\hat
   \delta_i(x_j))_{r+1\leq i,j\leq n}$. Then
   \begin{displaymath}
     \det S = \hat u_{r+1} \cdots \hat u_n x_{r+1}\cdots x_n + \phi,
   \end{displaymath}
   where $\phi \in \mf_R^{n-r+1}$, so in particular $\det S \neq 0$.  This
   implies that $B$ is a basis of $K(R)\otimes_R \bar \g_R$.  

   That $(1)\Rightarrow (2)$ is evident since $\g_A \cdot J_m \subset
   J_m$. The  converse follows since $\g_R \to T_R$ is surjective. 
 \end{proof}

\begin{theorem}\label{mon-the}
   Let $\g_A$ be a Lie algebroid over a regular allowed algebra
   $A$ and $I\subset A$ an ideal such that $\g_A \cdot I \subset I$
   (e.g. $I=A$), and put $\bfr= \Ker (\g_A\to T_A)$. Make the
   following assumptions:
   \begin{enumerate}\renewcommand{\labelenumi}{(\alph{enumi})}
   \item $\g_A$ contains a Lie sub-algebroid $\qb'_A$ such that
     $\alpha(\qb'_A)$ is a weakly toral Lie sub-algebroid of $T_A$.
   \item $\bfr_k=\bfr/\mf_A\bfr$ is solvable.
   \end{enumerate}

   Let $J_m$ be a maximal defining ideal of $A$, $S^\bullet=
   \oplus_{i\geq 0} J_m^i/J_m^{i+1}$, $R= A/J_m$, $\g_R= \g_A/J_m
   \g_A$, and $V= J_m/\mf_A J_m$. Let $\g_k = k\otimes_R \Ker(\g_R \to
   T_R)$ be the fibre Lie algebra of $\g_R$.
   \begin{enumerate} 
   \item The fibre Lie algebra is the semi-direct product
     \begin{displaymath}
       \g_k = \rf \rtimes \bigoplus_{f\in F} \Sl_{n_f}
     \end{displaymath}
     where $\rf$ is the radical of $\g_k$, and the integers $n_f$ are
     the dimensions of the simple subquotients of the $\g_k$-module
     $V$ that are of dimension $> 1$.
   \item Let $\nf_k$ be a maximal nilpotent subalgebra of a Levi
     factor of $\g_k$. Then
 \begin{displaymath}
   (k\otimes_RS^\bullet )^{\nf_k} = \So^\bullet(V^{\nf_k}),
\end{displaymath}
where $S^\bullet (V^{\nf_k})$ is the symmetric algebra of the
$\nf_k$-invariant subspace of $V$.
\item Assume that the Lie algebroid $\alpha(\qb_A')$ in (a) is toral, and let $J$ be any defining
  ideal of the $\g_A$-module $A$. The generalized Hilbert-Samuel function
   \begin{displaymath}
     n \mapsto \chi_I^J (n)= \ell_{\g_A}(I/J^{n+1}I)
   \end{displaymath}
   is a polynomial for high $n$.  The $\g_A$-dimension and
   $\g_A$-multiplicity with respect to $J_m$ (see
   \Definition{\ref{def-dimension-mult}}) of the $\g_A$-module $A$ are
   \begin{displaymath}
     d_{\g_A}(A) =   \ell_{\g_k}(V) \quad \text{and} \quad  e_{\g_A}(A, J_m ) =1.
 \end{displaymath}
  \end{enumerate}
\end{theorem}

\begin{pfof}{\Theorem{\ref{mon-the}}} (1): Put $\g= \Ker(\g_R \to T_R )$, so $\g_k= k\otimes_R\g$,
  and let $\phi: \g_k \to \gl_k(V)$ define the representation of $\g_k$ in $V$. By (a) there exists
  a regular system of parameters as in \Proposition{\ref{lem-torus}} such that $V= \oplus_{i=1}^r k
  \bar x_i$.  Letting $\qb_R'$ be the image of $\qb'_A$ in $\g_R$, the image of $\qb'_R\cap \g$
  under the composed map $\qb'_R\cap\g \to \g \to \g_k \to \gl_k(V)$ will then contain the
  commutative Lie algebra $\sum_{i=1}^rk \nabla_{x_j}$. Hence a Cartan subalgebra of $\Sl_k(V)$ is
  contained in $\phi(\g_k)$. By \Theorem{\ref{recognition}} it follows that $\phi (\g_k) = \rf_1
  \rtimes \oplus_{f\in F} \Sl_{n_f}$, where $\rf_1$ is the radical of $\phi (\g_k)$; therefore,
  since (b) holds, \Proposition{\ref{solv-kernel}} implies (1).
% The maximal defining ideal $J_m$ is
%   the radical of $J$. By (a) it follows that $ J \subset J_m$ is an
%   inclusion of monomial ideals, where moreover $J_m= (x_j)_{j\in
%     \mathcal I}$ as in \Proposition{\ref{lem-torus}}, so $\bar x_j=
%   x_j \mod J_m$ forms a basis of $V$. 

  (2): We use the notation in the proof of
  \Theorem{\ref{recognition}}. Put $\Lc_k = \oplus_{f\in F}
  \Sl_{n_f}$, so $V= \oplus_{i=1}^r W_i= V^{\Lc_k}\oplus \oplus_{f\in
    F}W_f$, where the $W_i$ are simple $\Lc_k$-modules. Each symmetric
  product $S^l(W_f)$ is simple over $\Sl_k(W_f)$, and therefore the
  $\Lc_k$-module $k\otimes_R J_m^n/J_m^{n+1}= S^n(V)$ has the
  semi-simple decomposition
   \begin{displaymath}
     S^n(V) =\bigoplus_{k_i \geq 0, \sum k_i = n} S^{k_1} (W_1)\otimes_k S^{k_2}(W_2) \otimes_k \cdots
     \otimes_k S^{k_r}(W_r).
   \end{displaymath}
Therefore, if
   $\nf_k$ is a maximal nilpotent subalgebra of $\Lc_k$, the
   corresponding highest weight space is
\begin{displaymath}
  S^n(V)^{\nf_k} = \bigoplus_{k_i \geq 0, \sum k_i = n}  k X_1^{k_1}
  \otimes 
  X_2^{k_2}\cdots \otimes  X_r^{k_r}
\end{displaymath}
where $X_i$ is a basis of $W_i^{\nf_k}$.  Since $k[X_1, \dots , X_r]
=k[X_1]\otimes_k \cdots \otimes_k k[X_r]$, it follows that
$(k\otimes_RS^\bullet)^{\nf_k}= S^\bullet (V^{\nf_k})$.

(3): Since $\bfr$ acts trivially on $A$ we can assume that $\g_A
\subset T_A(I)$.  Put $N^n = k\otimes_RG^\bullet
_{J_m}(J^nI/J^{n+1}I)$ (note that $G^i_{J_m}(J^nI/J^{n+1}I)= J_m^i
J^nI/ (J_m^{i+1}J^n I + J^ {n+1}I)$ is non-zero only for finitely many
indices $i$, since $\ell_{\g_A}(J^nI/J^{n+1}I)< \infty$), so
$N^\bullet = \oplus_{n\geq 0} N^n$ is a $(S^\bullet(V),\g_k) $-module,
which is finitely generated over $ S^\bullet(V)$. We first have
\begin{displaymath}
  \ell_{\g_A} (J^nI/J^{n+1}I) = \ell_{\g_R}(G^\bullet_{J_m}(J^nI/J^{n+1}I)),
\end{displaymath}
hence the Hilbert series of the $(S^\bullet, \g_A )$-module
$\oplus_{n\geq 0} J^nI/J^{n+1}I$ is the same as the Hilbert series of
the $(S^\bullet, \g_R)$-module $\oplus_{n\geq 0}
G^\bullet_{J_m}(J^nI/J^{n+1}I)$.  Since $I\subset J_m$ and $J\subset
J_m$ are inclusions of monomial ideals one can select a regular system
of parameters $\{x_1, \dots , x_n\}$ of $A$ as in
\Proposition{\ref{lem-torus}}, so that $J_m= (x_1, \dots, x_r)$, and
$I$ and $J$ are generated by monomials in the parameters $x_1, \dots ,
x_r$ \Prop{\ref{mon-ideal}}; note that the same regular system of
parameters works for all monomial ideals simultaneously.  Clearly, the
$R$-module $G^i_{J_m}(J^nI/J^{n+1}I) $ is generated by elements
$\{m_s\}_{s=1}^j$ where $m_s$ is a product of monomial generators of
$I,J$ and $J_m$ in the parameters $x_1, \dots , x_r$.  Since $\alpha :
\g_R\to T_R$ is surjective \Prop{\ref{lem-torus}} there exist
$\delta_1, \dots , \delta_{n-r}\in \g_A\subset T_A(I) $ that map to
elements in $\g_R$ that lift the partial derivatives
$\partial_{x_{r+1}}, \dots, \partial_{x_n} \in T_R$. Therefore
$\delta_i = \partial_{x_{r+i}} + J_m \eta_i$, where $\eta_i \in
T_A$. Now since $\qb_A = \oplus A x_i \partial_{x_i}$ acts on $ \g_A$,
a weight argument gives that $\partial_{x_{r+i}} \in \g_A$.  We can
therefore assume that $\delta_i = \partial_{x_{r+i}}$, and thus
$\delta_i \cdot m_s = 0$. Hence \Proposition{\ref{local-cond}} implies
that $G^i_{J_m}(J^nI/J^{n+1}I) \in \Loc(\g_R)$.  Therefore the Hilbert
series of the $( S^\bullet,\g_R)$-module $ \oplus_{n\geq 0}
G^\bullet_{J_m}(J^nI/J^{n+1}I)$ is the same as the Hilbert series of
the $( S^\bullet(V),\g_k) $-module $N^\bullet$,
   \begin{displaymath}
     H_{N^\bullet}(t) = \sum_{n\geq 0} \ell_{\g_k}(N^n)t^n.
   \end{displaymath}
   By \Theorem{\ref{hilb-liealg}} this is a rational function, and
 by (2) the invariant ring
   $S^\bullet(V)^{\nf_k}=S^\bullet(V^{\nf_k})$ is generated in degree
   $1$, so the exponents $n_i =1$ in the rational expression for
   $H_{N^\bullet}(t) $, it follows in a standard way (see
   \cite{matsumura}) that the function $n\mapsto \chi_I^ J(n)$ is
   given by a polynomial for high $n$. If $I=A$ and $J= J_m$, then
   $N^\bullet = S^\bullet (V)$, and $\ell_{\g_k}(S^n (V)) = \dim S^n
   (V)^{\nf_k} = \dim S^n(V^{\nf_k})$, which implies that
   $d_{\g_A}(A)=\ell_{\g_k}(V)$ and $e_{\g_A}(A, J_m)=1$.
\end{pfof}

\subsection{Monomial ideals in a polynomial ring}\label{sec-monomial}
Put $\Vc= \{1, \dots , n\}$ and define the polynomial ring $B = k[x_i | i \in \Vc]$. To a subset
$\Lambda $ of $\Nb^n$ one associates the monomial ideal $I=I_\Lambda= (X^\alpha \ | \ \alpha \in
\Lambda ) \subset B$, and one gets the graded ring $S^\bullet = B/I$. Let $B_{x_j} = k[x_1, \dots ,
\hat x_j, \dots , x_n]\subset B$ be the polynomial subring with the variable $x_j$ excluded.  The
tangential Lie algebroid $T_B(I_\Lambda )$ can be decomposed as vector space over $k$
\begin{equation}\label{mon-dec}
T_B (I) =  (B\otimes_k \tf )\oplus  \bigoplus_{j=1}^n I_{x_j}\partial_{x_j},  
\end{equation}
where $\tf = \oplus k\nabla_{x_i}$ and the monomial ideal in $B_{x_j}$
\begin{displaymath}
  I_{x_j} = (X^\beta \ \vert \
  \beta_j =0, x^ {\beta - e_j + \alpha_i} \in I_\Lambda, \text{  for all } \alpha_i \in \Lambda) \subset  B_{x_j};
\end{displaymath}
see (\citelist{\cite{brumatti-simis}*{Th. 2.2.1} \cite{tadesse:monomial}*{Th. 2.2}}). Letting $N =
\{i \ \vert \ I_{x_j}\neq B_{x_j} \} \subset \Vc$, the maximal defining ideal of the $T_B(I)$-module
$B$ is $J_m = (x_i \ \vert i\in N)\subset B$ \Prop{\ref{lem-torus}}, so that $R= B/J_m =
k[x_i]_{i\not\in N}$.  A Levi factor of the fibre Lie algebra $\g_k (\Lambda) = k\otimes \Ker (
R\otimes_B T_B(I) \to T_R)$ can be determined by decomposing the $\g_k(\Lambda)$-module
\begin{displaymath}
  V=  \frac {J_m}{(x_1 \dots, x_n)J_m} = \bigoplus_{i\in  N} k\bar x_i,
\end{displaymath}
where $\bar x_i = x_i \mod J_m (x_1, \dots , x_n) $\Props{\ref{solv-kernel}}{\ref{recognition}}.  We
denote this Levi factor by $\Lc(\Lambda)$, so that upon restricting $V$ to a $\Lc(\Lambda)$-module,
one can work out the semisimple decomposition
\begin{displaymath}
  V   =  V_0 \oplus \oplus_f W_f =   \bigoplus_{f\in F}   \bigoplus_{i \in N_f } k \bar x_{j_i},
\end{displaymath}
where $V_0 $ is the trivial component, $F$ is the set of isomorphism classes of the simple
$\Lc(\Lambda)$-modules in $V$, and $N_f\subset \Vc $ indexes a basis of a module in $f$.  We have a
disjoint union
\begin{displaymath}
  N  = N_0 \cup  \bigcup_{f\in F}  N_f;
\end{displaymath}
where the set $ \{ \bar x_s \}_{s\in N_f } $ forms a basis of $W_f$ and $N_0\subset \Vc$ indices a
basis of $V_0$. Again by \Theorem{\ref{recognition}} (and its proof), we have
\begin{displaymath}
  \Lc(\Lambda) = \bigoplus_{f\in F} \Sl(W_f). 
\end{displaymath}
A $k$-basis of $ \Sl(W_f)\subset \Lc(\Lambda)$ is formed by the fibres of the elements
$x_i\partial_{x_j}, x_i \partial_{x_i}- x_j \partial_{x_j}\in T_B(I)$, where $i, j\in N_f$, $i
\neq j$.
\begin{example} Let $I= (x_1 X^\gamma, x_2X^\gamma, x_3X^\gamma, X^\beta)$, where $X^\beta, X^\gamma
  \not \in (x_1, x_2, x_3) $. Then the derivations $x_i \partial_{x_j}$, $i,j=1, \dots , 3$, generate
  a Lie subalgebra of $T_A(I)$ that is isomorphic to $k\oplus \Sl_3$.
\end{example}

\begin{example}
  \begin{enumerate}
  \item If $\Lambda = \emptyset$, then $\g_k(\Lambda) = k$. If $\Lambda = \{e_i, i=1, \dots , n\}$,
    then $\g_k(\Lambda) = \gl_n$. If $\Lambda = \{e_i + e_j, i,j=1, \dots, n\}$, then $\g_k(\Lambda)
    = \gl_n$.
  \end{enumerate}

\end{example}

We want to compare the classical Hilbert series of $S^\bullet= B/I$ to our Lie algebroid Hilbert series:
\begin{align*}\tag{C}
  H_{cl}(S^\bullet , t) &= \sum_{i\geq 0} \dim_k (S^i) t^i = \sum_{i\geq 0}\ell_{\tf} (S^i ) t^i \\
  \gg H_{Lie}(S^ \bullet, t) &= \sum_{i\geq 0} \ell_{\g_R}(G_{J_m}(S)^i) t^i,
\end{align*}
where the inequality signifies that the lengths $\ell_{\tf} (S^i ) \geq \ell_{\g_R}(G_{J_m}(S)^i) $
for all $i$. Here the ring  $G_{J_m}(S)^\bullet $ can be identified with $S$, but it should be decomposed 
\begin{displaymath}
G_{J_m}(S)^\bullet  = S_1^\bullet \otimes_k R, 
\end{displaymath}
were $S_1^\bullet$ is a quotient of the subring $k[x_i \ |\ i\in N]\subset B$, and
$G_{J_m}(S)^\bullet$ is graded only with respect to the total degree of the variables $\{x_i\}_{i
  \in N}$.  Moreover, the canonical map $\g_R= R\otimes T_A(I) \to T_R$ is surjective, and
in fact $\g_R = \g \rtimes T_R$.  Since the condition in \Proposition{\ref{loc-ideal}} is satisfied
for the basis $(\partial_{x_i})_{i\in \Vc \setminus N}$ of $T_R$, so that $G_{J_m}(S)^\bullet $ is a
local system and $\ell_{\g_R}(G_{J_m}(S)^i) = \ell_{\g_k}((k\otimes_{\mf_B} S)^i)$, we also get
from \Theorem{\ref{main}}:
\begin{proposition}\label{mon-hilb}
  The Hilbert series $H_{Lie}(S^ \bullet, t)$ is rational.
\end{proposition}
Notice again that $\ell_{\g_k}(k\otimes_R S^i) = \dim (k\otimes_BS^i)^{\nf_k}$ is the invariant space in
$S_i$ for a maximal nilpotent subalgebra $\nf_k$ of Levi factor of $\g_k(\Lambda)$.  If the fibre
Lie algebra is solvable \Prop{\ref{solv-fibre}} then in \thetag{C} we have $H_{Lie}(S^\bullet, t)
=H_{cl}(S^\bullet, t) (1-t)^s$, where $s= |\Vc \setminus N|$, so that if moreover $R= k$, then we
have equality in \thetag{C} .

\begin{remark} The Euler derivations $E_f = \sum_{i\in \Omega_f} x_i\partial_{x_i}$, $f\in F$ form a
  basis of the commutative Lie algebra $\hf =\rf/ [\rf, \rf]$, where $\rf $ is the radical of the
  fibre Lie algebra $\g_k$, so that the characters $\phi : \rf\to k$ are determined by the values of
  the $E_f$.  Notice also that the Euler operators commute with $\Lc(\Lambda)$.  Since all the
  characters of $S^\bullet$ are integral, we can write the equivariant Hilbert series of $S^\bullet$
  with respect to the solvable Lie algebra $\rf$ as follows
  \begin{displaymath}
    H^{eq}_{S^\bullet}(t) = \sum \sum  \dim_k (S^i_\phi)^\nf \phi t^i = \sum
    \sum \dim_k (S^i_\phi) ^{\nf}
    X^\phi t^n,
  \end{displaymath}
  where $ (S^i_\phi)^\nf$ is the space of $\nf$-invariants of weight $\phi$ in
  $S^i$, and $X^\phi = X_1^{\phi_1}\cdots X_l^{\phi_l}$ (abusing the notation),
  $\phi(E_i) = \phi_i$, $l = |F|$.  The $\Nb^n$-grading of $S^\bullet = \sum
  S^\alpha$ results in a $\Nb^n$-graded (so-called fine) Hilbert series of
  $S^\bullet$. The action of the Euler operators $E_f$ is what remains of the
  $\Nb^n$-grading when we instead of $\dim_k S^\alpha$ study the numbers
  $\ell_{\g_k}(S^i)$.
\end{remark}

\subsection{Stanley-Reisner rings}\label{stanley} In this section our set $\Vc$ is regarded as the vertex set of an
abstract simplicial complex $\Delta \subset 2^\Vc$, so that if $\sigma \in \Delta$ and $\mu \subset
\sigma$ it follows that $\mu\in \Delta$. Elements in $\Delta$ are called faces and the maximal faces
with respect to inclusion are {\it facets}. Let $\Delta_{fc} \subset \Delta$ be the subset of
facets.  The dimension of a face $\sigma$ is $\dim \sigma = |\sigma |-1$.

For a function $\alpha: \Vc \to \Nb$, $i \mapsto \alpha_i$ we put $\supp \alpha
= \{i \in \Vc \ | \ \alpha_i\neq 0\} \subset \Vc$ and $x^\alpha
=x_1^{\alpha_1}\cdots x_r^{\alpha_r} $; we identify a subset $\sigma \subset
\Vc$ with its indicator function $\sigma_i =1$, $i\in \sigma$, $\sigma_i =0$,
$i\not \in \sigma$.  Letting $\Lambda = 2^\Vc \setminus \Delta$ be the set of
non-faces, the Stanley-Reisner ideal and ring of $\Delta$, respectively, are
\begin{displaymath}
  I_\Delta = (x^\sigma \ \vert \ \sigma \in \Lambda), \quad \text{and }   k[\Delta] = \frac {k[x]}{I_\Delta}.
\end{displaymath}
See \cite{miller-sturmfels:book} for an introduction to the basics of the Hilbert series of
$k[\Delta]$.

We now provide the vertex set $\Vc$ with a topology $\tau$ and also a stratification of this
topological space that will reflect the structure of the fibre Lie algebra of the Lie algebroid of
derivations of $k[\Delta]$.  Finite topological spaces are discussed for example in
\cite{stong:finite, mccord}. The closed sets of the topology $\tau$ are generated by the facets of
$\Delta$, so that the open sets are generated by the sets $ F^c \in \tau$, $F\in \Delta_{fc}$; in
this topology the closed points in $\Vc$ are vertices that can be cut out by an intersection of
facets. In general the closure of a point $i\in \Vc$ is $\{i\}^- = \cap_{i\in F \in \Delta_{fc}}F$,
and the closure of an arbitrary subset $U \subset \Vc$ is
\begin{displaymath}
U^-=  \bigcup_{i\in U} \bigcap_{i  \in  F\in \Delta_{fc}} F
\end{displaymath}
(recall that $\Vc$ is a finite set ). A closed subset of $\Vc$ is irreducible if and only if it is
an intersection of facets. The topological space $(\Vc, \tau)$ is clearly noetherian so that any
nonempty closed set $M$ of $\Vc$ can be decomposed uniquely $M= \cup M_t$, where $M_t$ is closed and
irreducible, and $M_t \not \subset M_s$ when $s\neq t$.

For each subset $\lambda \subset \Delta_{fc}$, the closed set $ \cap_{F\in \lambda} F $ contains a
closed subset formed by the union of the closed sets $ G \cap \cap_{F\in \lambda} F $, where
$G\in \Delta_{fc}$ is a facet such that $ \cap_{F\in \lambda} F \not \subset G$. Then
\begin{displaymath}
  \Vc_\lambda =  \bigcap_{F\in \lambda} F  \ \setminus \bigcup_{\substack {G\in \Delta_{fc}, \\ \cap_{F\in \lambda} F \not \subset G}}   (G \bigcap \bigcap_{F\in \lambda} F )
\end{displaymath}
is a locally closed subset of $\Vc$, which is closed if each facet $G\in \Delta_{fc}$ either
contains $\Vc_\lambda$ or is disjoint with it.  Notice that $ \bigcap_{F\in \lambda} F $ may very
well be a union of proper subsets that are formed by intersections of facets, and then $\Vc_\lambda
= \emptyset$.  If $\lambda'$ is another subset of $\Delta_{fc}$, then $\Vc_\lambda \cap
\Vc_{\lambda'} = \Vc_{\lambda\cup \lambda'}$.  Since $\Vc_\lambda$ is the set of vertices that
belong to all the facets $F \in \lambda$ and all facets that contain the intersection $\cap_{F\in
  \lambda} F$, but no other facets, it follows that if $\Vc_\lambda \cap \Vc_{\lambda'} \neq
\emptyset $, then $\Vc_\lambda = \Vc_{\lambda'}$; we then say that $\lambda$ and $ \lambda' $ are
equivalent, defining an equivalence relation on the power set of $\Delta_{fc}$.  Let $\Omega$ be the
set of equivalence classes and put (abusing the notation slightly) $\Vc_\omega = \Vc_\lambda$ when
$\lambda \in \omega\in \Omega$.  Say that $\omega$ dominates $\omega'$, and write $\omega >
\omega'$, if $ V_{\omega'} \subset \bar V_\omega$; we also say that $\Vc_\omega$ dominates
$\Vc_{\omega'}$ . The relation $>$ makes $\Omega$ into a partially ordered set $(>, \Omega)$.  Then
we get a stratification
\begin{displaymath}
  \Vc = \bigcup_{\omega \in \Omega} \Vc_\omega
\end{displaymath}
by mutually disjoint locally closed subsets;  the closure of one stratum $\Vc_\omega$ is a union of
the strata that are dominated by it. Put $d_\omega = |\Vc_\omega|$.  A face $\sigma \in \Delta$ can
be decomposed $\sigma = \cup_{\omega \in \Omega} \sigma_\omega$, where $\sigma_\omega \in
V_\omega$. Let $\tilde \sigma $ be the subset of elements $\omega$ in $\Omega$ such that
$\sigma_\omega \neq \emptyset$.  The following subset of the power set of $\Omega$
\begin{displaymath}
  \tilde \Delta =
  \{\tilde \sigma \ |  \ \sigma \in \Delta \}
  \subset 2^\Omega 
\end{displaymath}
is then a simplicial complex on the vertex set $\Omega$.  We call $\tilde
\Delta$ the {\it essential simplicial complex of } $\Delta$. Since the facets of
$\tilde \Delta$ are induced by facets of $\Delta$, it follows that upon repeating the
procedure the partitioning of $\Omega$ are all singleton subsets and $\tilde
{\tilde \Delta} $ can be identified with $ \tilde \Delta$.

We change notation slightly compared to \Section{\ref{sec-monomial}} and instead let
$\g_k(\Delta)=k\otimes_R \g$ denote the fibre Lie algebra, where $\g = \Ker (R\otimes_BT_B(I_\Delta)
\to T_R)$.  The theorem below shows how the stratification of $\Vc$ reflects the structure of
$\g_k(\Delta)$.

Notice that $x^\alpha \partial_{x_i} \in I_\Delta T_B $ if some subset of $\supp \alpha $ belongs to
$ \Lambda$, so to determine $\g_k(\Delta)$ we need only consider the case $\supp \alpha \in \Delta$,
and if moreover $x^\alpha$ is not square-free, then $x^\alpha\partial_{x_i} \in (x_1, \dots ,
x_n)T_B(I)$ and therefore maps to $0$ in $\g_k(\Delta)$. Therefore it suffices to determine which
$x^\sigma \partial_{x_i}$ preserve $I_\Delta$ when $\sigma \in \Delta$.

Let $V_\omega$ be the $k$-vector space spanned by $\{x_i\}_{i\in \Vc_\omega}$, $E_\omega =
\sum_{x_i\in \Vc_\omega} x_i\partial_{x_i}$, and if $d_\omega \geq 2$, then $\Sl(V_\omega)$ is the
special linear algebra on $V_\omega$.

\begin{theorem}\label{simplex-lie}
  \begin{displaymath}
    \g_k(\Delta) = \rf \rtimes \bigoplus_{\omega \in \Omega, d_\omega \geq 2} \Sl(V_{\omega}),
  \end{displaymath}
  where $\rf$ is the radical of $\g_k$.  Let
  \begin{displaymath}
    \Delta_{\omega} = \{\sigma \in
    \Delta \ | \ \exists \ \omega_1 \in\Omega, \omega_1 > \omega, \omega \neq
    \omega_1 \text{ and } \ \ \sigma \cap \Vc_{\omega_1}\neq \emptyset \}.
\end{displaymath}
Then 
  \begin{displaymath}
    \rf  =  \bigoplus_{\omega \in \Omega} kE_\omega \bigoplus_{\omega \in \Omega} \bigoplus_{i\in \Vc_\omega}
    \bigoplus_{\sigma \in \Delta_\omega}  kx^\sigma\partial_{x_i}.
  \end{displaymath}
\end{theorem}
\begin{example} Let $\Delta_n$ be the power set of $\Vc$. Then $\g_k(\Delta_n) =\gl_n$. Let
  $\Delta_n^k$ be the $k$-skeleton of $\Delta_n$, $0\leq k < n$. Then $\g_k(\Delta_n^k) = \tf$; see
  equation (\ref{mon-dec}).
\end{example}
\begin{remark}
  The structure of the $k[\Delta]$-module $T_{k[\Delta]}=T_{A}(I_\Delta)/I_\Delta T_A$ is studied in
  \cite{brumatti-simis}.
\end{remark}

\begin{lemma}\label{lie-lemma-simplex}  Let $i\in \Vc$ and $\sigma \in \Delta$.
  The following are equivalent:
  \begin{enumerate}
  \item$ x^\sigma \partial_{x_i}\in T_A(I_\Delta)$.
  \item $i \in (\supp \sigma)^-$.
% For every facet $F\in \Delta_f$, $ \supp \alpha \cap F \neq \emptyset
%     \Rightarrow i \in F $. $\supp \alpha \subset \Delta_{rf}^i \alpha $
\end{enumerate}
  \end{lemma}
  \begin{proof}
    We decompose the closed set into irreducible components $\sigma^- = \cup_{t} (\sigma^{(t)})^- $, where the
    closure of each set $ \sigma^{(t)}$ is an intersection of facets.

    $(1)\Rightarrow (2)$: Assume on the contrary that for each $t$ there exists a facet $F_t$ such
    that $i \not\in F_t$ and $ \sigma^{(t)} \subset F_t$.  Then since $F_t$ is a facet there
    exists $H_t\in \Lambda$ such that $i \in H_t$ and $H_{t1}= H_t\setminus \{i\} \subset F_t$, so
    that $ \sigma^{(t)} \cup H_{t1} \in \Delta$.  It follows $H = \cup_t H_t \in \Lambda$
    (since otherwise each $H_t \in \Delta$), and if $H_1 = H \setminus \{i\}$, then $\sigma
    \cup H_1\not\in \Lambda$.  Therefore $x^{\sigma} \partial_{x_i}(x^H) = x^{\sigma} x^{H_1} \not\in
    I_\Delta $.

    $(2)\Rightarrow (1)$: By (2) we can assume that for some $t$ we have $\sigma^{(t)} \subset F_t\in \Delta_{fc}$ and $i\in F_t$, so that if
    $i\in H $ and $\supp \sigma^{(t)} \cup (H\setminus \{i\}) \in \Delta$, then
    $H\in \Delta$. Therefore if $x^{\sigma^{(t)}} \partial_{x_i} (x^H)\in x^\beta$
    and $H\in \Lambda$, it follows that $\supp \beta \in \Lambda$. This implies
    that $x^\sigma \partial_{x_i}(x^H)\in I_\Delta$.

%  % $x_j\in F$ and $x_i\not\in F$.  Then, by the definition of a facet, there
%  %    exists a vertex $x_k \in F$ such that $H=\{x_i, x_k\} \in \Lambda$. Hence
%  %    $H_1 = (H \setminus \{x_i\}) \cup \{x_j\} = \{x_j, x_k\} \subset F$, so that
%  %    $H_1 \not\in \Lambda$. Therefore $x_j \partial_{x_i}(x^H) = x^{H_1}\not \in
%  %    I_\Delta$.

%   $(2)\Rightarrow (1)$: 
% It suffices to prove
%   \begin{displaymath}
%     H\in \Lambda \Rightarrow  H_1 =(H \setminus \{x_i\})
%     \cup \supp \alpha \in \Lambda,
% \end{displaymath}
% since then $x^\alpha \partial_{x_i}(x^H) = x^{H_1}\in I_\Delta$.  If on the
% contrary $H_1 \not\in \Lambda$, then there exists $F\in \Delta_f$ such that $H_1
% \subset F$, so that $F\cap \supp \alpha \neq \emptyset$, hence by (2) $i \in F$,
% and therefore $ H = H_1 \cup \{x_i\}\subset F$, and so $H \in \Delta$.

% It suffices to prove
%   \begin{displaymath}
%     H\in \Lambda \Rightarrow  H_1 =(H \setminus \{x_i\})
%     \cup \{x_j\} \in \Lambda,
% \end{displaymath}
% since then $x_j \partial_{x_i}(x^H) = x^{H_1}\in I_\Delta$.  If on the contrary $H_1
% \not\in \Lambda$, the there exists $F\in \Delta_f$ such that $H_1 \subset F$. Therefore $x_j \in F$,
% and by (2), $x_i \in F$; hence $H \subset F$, and so $H \not \in \Lambda$.
\end{proof}

\begin{pfof}{\Theorem{\ref{simplex-lie}}}
  Let $\Delta_i$ be the subset of $\sigma\in \Delta$ that satisfy the condition (2) in
  \Lemma{\ref{lie-lemma-simplex}}, so that by the decomposition (\ref{mon-dec}),
  \begin{displaymath}\tag{*}
  \g_k = \bigoplus_{i=1}^n \bigoplus_{ \sigma \in \Delta_i } k x^\sigma \partial_{x_i}.
\end{displaymath}
It follows from \Lemma{\ref{lie-lemma-simplex}} that if $j\in \Vc_{\omega}$ and $i \in
\Vc_{\omega'}$, then $x_j\partial_{x_i}\in \g_k$ if and only if $\omega > \omega'$. Therefore
$x_i\partial_{x_j}\in \g_k$ when $i,j \in \Vc_\omega$, implying that $\gl(V_\omega)\subset \g_k$.
\Lemma{\ref{lie-lemma-simplex}} also implies that if $\sigma \in \Delta_i$ and $\Vc_\omega \cap
\supp \alpha \neq \emptyset$, where $\omega > \omega'$, while $\omega \neq \omega'$, then
$x^\sigma \partial_{x_i}\in \rf$.  Since the Euler operators $E_\omega = \sum_{i\in \Vc_\omega}
x_i \partial_{x_i} \in \gl(V_\omega) $ belong to $\rf$, we have now accounted for all the terms in
the decomposition \thetag{*}. This implies the asserted description of $\rf$ and the Levi algebra of
$\g_k$.
\end{pfof}
% Say that a face $\sigma \in \Delta$ is {\it even/odd} if the number of vertices $|\tilde \sigma|$ in
% $\tilde \sigma$ is an even/odd number, respectively.
% \begin{definition}
%   Let $\Delta$ be a simplicial complex of the vertex set $\Vc$, $\Vc = \cup_{\omega \in \Omega}
%   \Vc_\omega$ its canonical stratification, so that any face $\sigma = \cup \sigma_\omega $, where
%   $\sigma \subset \Vc_\omega$, and put $d_\omega = |\Vc_\omega|$.  The {\it weight} of $\sigma$ is
%   $\wt(\sigma) = \sum_{ \omega \in \tilde \sigma } d_\omega $, and the {\it mass} of $\sigma$ is
%   $m(\sigma) = \prod_{\omega \in \tilde \sigma} d_\omega$.  Let $g^+_{w,m}$ ( $g^-_{w,m}$) be the
%   number of even faces (odd) $\sigma $ of given weight and mass, $\wt (\sigma) = w$ and $m(\sigma) =
%   m$ .
% \end{definition}
Let $B = \{d_\omega\}_{\omega\in \Omega}\subset \Nb_{\geq 1}$ be the set of sizes of the cells
$\Vc_\omega$, and let ${\Nb} ^B$ be the set of all "multisets'' on the set $B$, with the relaxed
condition that we allow also the value $0$. Each face $\tilde \sigma\in \tilde \Delta$ defines a
multiplicity function $m_\sigma : B \to \Nb$, $m_\sigma(s) = |\{\omega \in \Omega \ | \
\sigma_\omega \neq \emptyset, d_\omega = s \}|$, so indeed $m_\sigma \in {\Nb} ^B$; notice that this
function is determined by $\tilde \sigma\in \tilde \Delta$ although we write $m_\sigma$.  If $m_d$
is the number of cells of size $d$, then the multiplicities satisfy $0\leq m_\sigma (s) \leq
m_d$. We define the {\it mass} of $m\in {\Nb} ^B $ by $\mu (m) = \sum_{s\in \supp m} sm(s)$, so that
in particular $\mu (m_\sigma) = |\sigma |$. The support of $m$ is $\supp m = \{s \ | \ m(s)\neq
0\}$.
\begin{definition}
  For each function $m\in {\Nb} ^B$ we let $g(m)$ be the number of faces $\sigma \in \Delta$ such
  that $m_\sigma = m $. We call $g(m)$ the {\it face multiplicities} of $\Delta$.
\end{definition}
Let $f_i$ be the face numbers of the complex $\Delta$, the number of faces $\sigma\in \Delta$ such
that $\dim \sigma = i$.  Let $d = \dim \Delta +1= \max f_i +1 $, and $e= |\{ i \in \Vc \ | \
\{i\}\not\in \Delta\}| $. In practise, most often $\cup_{\sigma \in \Delta} \sigma = \Vc$, and then
$e=0$.
\begin{theorem}\label{hilb-eq} The Hilbert series of $k[\Delta] $ is 
  \begin{align*}
    H_{cl} (k[\Delta], t ) &= \frac {1}{(1-t)^d} \sum_{i=0}^d f_{i-1} (1-t)^{d-i } t^i\\ &= \frac
    {1}{(1-t)^{d+e}} \sum_{m \in {\Nb}^B } g(m)(1-t)^{d-\mu(m)} \prod_{ s\in \supp m }
    (1-(1-t)^{s})^{m(s)}.
  \end{align*}
\end{theorem}
Thus the face multiplicities $g_m$ determine the face numbers $f_i$, so this piece of information in
$\Delta$ is determined by its essential complex $\tilde \Delta$. Notice that if $\Delta = \tilde
\Delta$, so that $\Omega = \Vc$ and $d_\omega =1$ for all $\omega$, then $g_m = f_{i-1}$, where $i =
|\supp m|$, making the equality in \Theorem{\ref{hilb-eq}} a tautology.

\begin{lemma}\label{incl-lemma} Let $\sigma \in \Delta$ and decompose it as $\sigma = \cup_{\omega} \sigma_\omega $,
  where $\sigma_\omega = \sigma \cap \Vc_\omega$.  If $\sigma'_\omega \subset \Vc_\omega$ and
  $|\sigma'_\omega| = |\sigma_\omega|$, then
  \begin{displaymath}
    \sigma' = \bigcup \sigma_\omega' \in \Delta.
  \end{displaymath}
\end{lemma}

\begin{proof}
  Simple proof: Each facet $F_\omega$ that contains $\sigma_\omega$ also contains $\Vc_\omega$ and
  $\sigma'_\omega\subset \Vc_\omega \subset F_\omega$, so that $\sigma' \subset \cap_{\sigma_ \omega
    \subset F_\omega } F_\omega \in \Delta$.

Funny proof: The element $x^\sigma = \prod_{\omega\in \Omega} x^{\sigma_\omega}$ is a
non-zero face monomial in $k[\Delta]$.  Since each homogeneous component of the polynomial ring
$k[x_i | i \in V_\omega ]$ is a simple $\Sl_{n_\omega}$-module, by the density theorem there exist
elements $P_\omega\in \Dc(\Sl_{n_\omega})$ in the enveloping algebra of $\Sl_{n_\omega}$ such that
$P_\omega x^{\sigma_\omega} = x^{\sigma'_\omega}$. Since moreover there exists $Q_\omega \in
\Dc(\Sl_{n_\omega}) $ such that $Q_\omega P_\omega x^{\sigma_\omega} = x^{\sigma_\omega}$ it follows
that $x^{\sigma'}= \prod_\omega x^{\sigma'_\omega} $ is again a non-zero face monomial in
$k[\Delta]$, and therefore $\sigma' \in \Delta$.
\end{proof}

\begin{lemma}\label{s-r-iso} As graded vector spaces we have 
  \begin{displaymath}
    k\otimes_R    k[\Delta] = \sum_{\tilde  \sigma \in \tilde \Delta }  \ \bigoplus_{\omega \in
      \tilde \sigma, n_\omega
      \geq 1} \bigotimes S_{n_\omega}(V_\omega), 
  \end{displaymath}
  where the elements in the spaces $V_\omega$ all have degree $1$. Here the term with $\tilde
  \sigma = \emptyset$ is defined to be the 1-dimensional space $k$.
\end{lemma}
\begin{proof}
  We know that $x^\alpha$, $\sigma = \supp \alpha \in \Delta $ forms a $k$-basis of $ k\otimes_R
  k[\Delta] $.  From the partitioning $\sigma = \cup_{\omega \in \Omega} \sigma_\omega$ we can write
  $\alpha = (\alpha_\omega)$, where $\supp \alpha_\omega = \sigma_\omega $. By
  \Lemma{\ref{incl-lemma}} it follows that $x^{\alpha'}$ belongs to the basis whenever $|\supp
  \alpha_\omega' | = |\sigma_\omega| $.  This implies the assertion.\end{proof}

\begin{pfof}{\Theorem{\ref{hilb-eq}}}
  As in the usual proof (see \cite[\S 1.2]{miller-sturmfels:book}) of the structure of the Hilbert
  series of a Stanley-Reisner ring, which also gives the first equality, one first works with the
  fine Hilbert series using all the variables $(x_{i_j})_{i_j\in \Vc_j }$, then we put
  $y_{\omega_j}= x_{i_j}$, when $\omega_j \in \Vc_j$ and $i_j \in \Vc_j$, resulting in the $\Nb^r$-
  grading that is induced by the Euler operators $E_\omega = \sum_{i \in \Vc_i}
  x_i\partial_{x_i}$. Put $y^\alpha = y_{\omega_1}^{\alpha_{\omega_1}}\cdots
  y_{\omega_i}^{\alpha_{\omega_i}}$. Let $(k\otimes_B k[\Delta])_\alpha$ be the common eigenspace of
  the operators $E_\omega$.  Since
  \begin{displaymath}
    \sum_{n\geq 1}  \bigotimes_{\substack {\omega \in \tilde \sigma,  n_\omega \geq 1 \\ \sum_{\omega \in \tilde \sigma} n_\omega =n} } S_n (V_\omega)\cong \prod_{\omega \in \tilde \sigma } \sum_{n_\omega \geq
      1} S_{n_\omega}(V_\omega),
\end{displaymath}
as graded vector spaces, \Lemma{\ref{s-r-iso}} implies
  \begin{align*}
    \sum_{\alpha_i \geq 0} \dim_k ((k\otimes_B k[\Delta])_\alpha) y^\alpha &= \sum _{\tilde \sigma \in
      \tilde \Delta}\prod_{\omega \in \tilde \sigma} (\sum_{n_\omega\geq 1} \dim_k
    (S_{n_\omega}(V_\omega)) y_\omega^{n_\omega}) \\ &= \sum _{\tilde \sigma \in \tilde \Delta}
    \prod_{\omega \in \tilde \sigma} (\frac 1{(1-y_\omega )^{d_\omega}}-1).
  \end{align*}
  Setting $y_\omega = t$ we get, noticing that $\sum_{\omega \in \tilde \sigma } d_\omega = \mu
  (m_\sigma)$,
\begin{align*}
  H_{cl}(k[\Delta], t) & = H(R\otimes_k k\otimes_B k[\Delta], t) = H(R,t) H(k\otimes_B k[\Delta], t) \\ &
  = \frac 1{(1-t)^e} \sum_{i \geq 0} \dim k[\Delta]_n t^n = \frac 1{(1-t)^e} \sum _{\tilde \sigma
    \in \tilde \Delta} \prod_{\omega \in \tilde \sigma}
  \frac{1-(1-t)^{d_\omega}}{(1-t)^{d_\omega}}\\ & =\frac 1{(1-t)^e} \sum _{\tilde \sigma \in \tilde
    \Delta} \frac 1{(1-t)^{\mu(m_\sigma )} } \prod_{\omega \in \tilde \sigma}
  (1-(1-t)^{d_\omega})\\
  & = \frac 1{(1-t)^e} \sum_{m \in {\Nb}^B } \frac {g(m)} {(1-t)^{\mu(m)}} \prod_{m(s)\neq 0}
  (1-(1-t)^{s})^{m(s)}.
\end{align*}
\end{pfof}
Let us now also describe the Lie algebroid Hilbert series of $k[\Delta]$.  
\begin{theorem}\label{lie-alg-hilb} Let $r= |\Omega|$ and $\tilde f_l$ be the face numbers of $\tilde
  \Delta$. Then
  \begin{align*}
    H_{Lie}(t) &=  \sum_{n\geq 0}\ell_{T_A(I_\Delta)}(G_{J_m}(k[\Delta])^n)t^n \\ 
&=\sum_{n\geq 0} \ell_{\g_k} (k\otimes_R k[\Delta]_n)t^n 
    = \frac 1{(1-t)^r} \sum_{i=0}^r \tilde f_{i-1}t^i(1-t)^{r-i}.
\end{align*}
We have $\tilde f_i \leq f_i$.
\end{theorem}
We notice in passing that the dimension and multiplicity, described in
\Definition{\ref{def-dimension-mult}}, equals $r$ and 1, respectively.
\begin{proof} The first equality in the second line follows since $G_{J_m}(k[\Delta])^i$ is a local
  system for each $i$; see the discussion before \Proposition{\ref{mon-hilb}}.  Recall also that the
  degree $0$ term in $k\otimes_R k[\Delta]$, corresponding to the face $\emptyset$, is the simple
  $\g_k$-module $k$. For a face $\tilde \sigma \in \tilde \Delta$ we let $V_{\tilde \sigma} $ be a
  vector space of dimension $|\tilde \sigma|$.  We have (as detailed below)
  \begin{align*}
    & \sum_{n\geq 0} \ell_{\g_k} (k\otimes_R k[\Delta]_n)t^n = 1+ \sum_{n\geq 1}
    \left(\sum_{\tilde \sigma \in \tilde \Delta } \ \bigoplus_{\substack{\omega
          \in \tilde \sigma, n_\omega \geq 1,\\ \sum n_\omega = n}} \ell_{\g_k}
      (\bigotimes
      S_{n_\omega}(V_\omega))\right) t^n\\
    &= 1+ \sum_{n\geq 1} \left(\sum_{\tilde \sigma \in \tilde \Delta }
     \dim_k S_n(V_{\tilde \sigma }) \right) t^n = \sum_{n\geq 0} \left(\sum_{\tilde
        \sigma \in \tilde \Delta } \dim_k S_n(V_{\tilde \sigma }) \right) t^n = \frac
    1{(1-t)^r} \sum_{i=0}^r \tilde f_{i-1}t^i(1-t)^{r-i}.
  \end{align*}
  \Lemma{\ref{s-r-iso}} implies the first equality.  Since a Levi sub-algebra of $\g_k$ is of the
  form $\oplus_{\omega \in \Omega, d_\omega \geq 2 } \Sl (V_\omega)$, it follows that $ \otimes
  S_{n_\omega}(V_\omega)) $ is a simple $\g_k$-module (see \Theorem{\ref{mon-the}} and the proof of
  (2) in that theorem).  This explains the second equality, and the last one is as in the classical
  case.  To see that $\tilde f_i \leq f_i$ it suffices to think that each face $\sigma $ in $\tilde
  \Delta$ come by a face $ \sigma $ in $\Delta$.
\end{proof}
\begin{example} Let $\Delta$ be the 3-dimensional non-pure simplicial complex of the set $\Vc =
  \{a,b,c,d,e,f,g\}$ as depicted in the picture the left below, so its facets are $\{\{a,b,c,g\},
  \{d,e,f,g\}, \{c,d,g\}\} $. The partitioning is $\Vc = \{\{a,b\}, \{f,e\}, \{c\}, \{d\}, \{g\}\}$,
  and the essential simplicial complex $\tilde \Delta$ is the picture to the right, which is pure of
  dimension $2$.
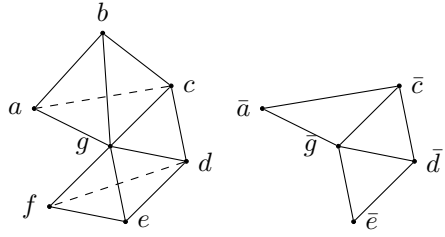
\begin{center}
\begin{tikzpicture}
[scale=1, vertices/.style={draw, fill=black, circle, inner sep=0.5pt}]
\node[vertices, label=left:{$g\ $}] (a) at (0,0) {};
\node[vertices, label=right:{$c$}] (b) at (0.8,0.8) {};
\node[vertices, label=left:{$a$}] (c) at (-1,0.5) {};
\node[vertices, label=above:{$b$}] (d) at (-0.1,1.5) {};
\node[vertices, label=right:{$d$}] (e) at (1.0,-0.2) {};
\node[vertices, label=left:{$f$}] (f) at (-0.8,-0.8) {};
\node[vertices, label=right:{$e$}] (g) at (0.2,-1.0) {};
\foreach \to/\from in {a/b,a/c,a/d,b/d,c/d,a/e,b/e,f/g,a/f,a/g,g/e}
\draw [-] (\to)--(\from);
\draw [dashed] (b)--(c);
\draw [dashed] (e)--(f);
\node[vertices, label=left:{$\bar g\ $}] (a1) at (3,0) {};
\node[vertices, label=right:{$\bar c$}] (b1) at (3.8,0.8) {};
\node[vertices, label=left:{$\bar a$}] (c1) at (2,0.5) {};
\node[vertices, label=right:{$\bar d$}] (e1) at (4.0,-0.2) {};
\node[vertices, label=right:{$\bar e$}] (g1) at (3.2,-1.0) {};
\foreach \to/\from in {a1/b1,a1/c1,a1/e1,b1/e1,a1/g1,g1/e1,b1/c1}
\draw [-] (\to)--(\from);
\end{tikzpicture}
\end{center}
We have $\Omega = \{\bar a, \bar c, \bar d, \bar e, \bar g\}$ and $d_{\bar a} = d_{\bar e}= 2,
d_{\bar c}= d_{\bar d} = d_{\bar g}=1$, so that $B= \{1,2\}$.  We can collect the basic data in two
tables:
\begin{center}
  \begin{tabular}{cc}
    Faces              & Multiplicity fcn \\
    \hline
    $\emptyset$ & 0 \\
    $\bar c, \bar d, \bar g$           &  $\delta_{s1} $\\
    $\bar a$, $\bar e$          &   $\delta_{s2}$ \\
    $\bar c \bar d, \bar d \bar g , \bar g \bar c  $ & $2\delta_{s1}$ \\
    $\bar a \bar c, \bar d \bar e, \bar e \bar g, \bar g \bar a  $    &   $\delta_{s1}+
    \delta_{s2}$ \\
    $\bar a \bar c \bar g, \bar d  \bar e \bar g $ &  $2\delta_{s1} + \delta_{s2} $\\
    $  \bar c \bar d \bar g$ &  $3 \delta_{s1}$ \\
  \end{tabular}
\end{center}
\begin{center}
  \begin{tabular}{cccc}
    Multiplicity fcn $m$ &  face multiplicity $g$ & mass $\mu$ & $\prod_{m(s)\neq 0 }(1-(1-t)^{s})^{m(s)}$\\
    \hline
    $0$ & 1  & 0&1 \\
    $\delta_{s1} $ & 3 &1& $t$ \\
    $\delta_{s2}$ & 2 & 2&$ (1-(1-t)^2) $\\
    $2\delta_{s1} $ & 3 &2&$t^2$ \\
    $\delta_{s1}+
    \delta_{s2}$ &4 & 3& $t(1-(1-t)^2)$ \\
    $2\delta_{s1} + \delta_{s2} $ &2 & 4& $t^2(1-(1-t)^2)$ \\
    $3 \delta_{s1}$ & 1 &3&$ t^3$ \\
  \end{tabular}
\end{center}
The face numbers are $f_{-1}=1, f_0 =7, f_1=13,f_2=9, f_3=2 $, so \Theorem{\ref{hilb-eq}} implies
the not immediately evident identity
\begin{align*}
H(t)&=  \frac 1{(1-t)^7} [(1-t)^7+ 7 (1-t)^6 t +13 (1-t)^5 t^2 + 9 (1-t)^4t^3 + 2(1-t)^3 t^4 ] =
\\
& = \frac 1{(1-t)^7}[ (1-t)^7+ 3  (1-t)^6t+3  (1-t)^5 t^2+ 2 (1-t)^5 (1-(1-t)^2) \\ & +   (1-t)^4
t^3+4  (1-t)^4 (1-(1-t)^2) t+2  (1-t)^3 (1-(1-t)^2) t^2
].
\end{align*}

We have $\tilde f_{-1}=1, \tilde f_0=5, \tilde f_1= 7, \tilde f_2= 3, \tilde f_3=0$, the
fibre Lie algebra $\g_k= \rf \rtimes \Sl_2 \oplus \Sl_2 $, and one can  immediately read off the Lie
algebroid Hilbert series from \Theorem{\ref{lie-alg-hilb}} .
\end{example}
\section{Complex analytic singularities}\label{hypersurfaces}
Let $\Oc_n$ be the ring of complex convergent power series, and denote
its maximal ideal by $\mf$. Let $ I\subset \mf $ be an ideal,
$B=\Oc_n/I$, and $ T_{\Oc_n}(I)$ the Lie algebroid of tangential
derivations, so $T_B= T_{\Oc_n}(I)/IT_{\Oc_n}$.  Let $J\subset \Oc_n$
be the contraction of the Jacobian ideal $\bar J \subset B$, as
defined in \Section{\ref{allowed}}, so for example when $I=(f)$, then
$J = I + T_{\Oc_n}\cdot f $. Then $I \subset J $ is an inclusion of
$T_{\Oc_n}(I)$-preserved ideals (see \cite{kallstrom:preserve,
  scheja:fortzetsungderivationen}), so that putting $A= \Oc_n/J$ we
have a homomorphism of $T_{\Oc_n}(I)$-modules $B\to A$, and 
surjective homomorphisms of $\Oc_n$-modules and Lie algebras
$T_{\Oc_n}(I) \to T_{B} $ and $T_{\Oc_n}(J)\to T_{A}$.  Note that a
radical ideal $I$ is an ideal of definition of the
$T_{\Oc_n}(I)$-module $\Oc_n$ if and only if $B$ is a regular local
ring \Prop{\ref{reg-def-ideal}}.

It follows from Rossi's theorem \cite{rossi:vectorfields} that we may
assume (see e.g. a discussion in \cite{granger-schulze:initial-lieb}), that
\begin{displaymath}
  T_{\Oc_n}(I) \subset T_{\Oc_n}(\mf)= \mf T_{\Oc_n},
\end{displaymath}
and we will do so for the remainder of this section.  This implies that $\mf$ is
the maximal defining ideal for any $T_{\Oc_n}(I)$-module of finite type over
$\Oc_n$, and the fibre Lie algebra
\begin{displaymath}
  \g_{\Cb}= \frac {T_{\Oc_n}(I)}{\mf T_{\Oc_n}(I)}.
\end{displaymath}
The finite dimensional Lie algebra $\g_{\Cb}$ reflects certain symmetries of the singularity $B$
(see \Example{\ref{fibre-example}}).  We have the following two rather extreme cases: if $I=
(x^{\alpha})\subset \Cb\{x_1, x_2, \dots , x_n\}$ (a normal crossing divisor), one can assume
without loss of generality that $\alpha = (\alpha_1, \dots , \alpha_n)$ satisfies $\alpha_i>0$ for
all $i$, then $\g_{\Cb} = \Cb^n$, the $n$-dimensional commutative Lie algebra; if $I= \mf$, then
$\g_\Cb = \gl_n$, the general linear Lie algebra.  More examples are given below.

\begin{lemma} Let $I \subset \mf^2$ be an ideal, and $J\subset \mf$
  the Jacobian ideal of $I$.  Consider $\Oc_n$ as a
  $T_{\Oc_n}(I)$-module.  The following  are equivalent:
\begin{enumerate}
\item $J$ is an ideal of definition.
\item  $\dim_\Cb \Oc_n/J <  \infty$.
\item $\Oc_n/I$ has an isolated singularity.
\end{enumerate}
\end{lemma}
Of course, the assertion $(2) \Llra (3) $ is very well known, and it is easy to see that $(1)\Llra
(2)$, so the proof is omitted.

% \begin{proof} Put $\g= T_{\Oc_n}(I)$.  $(1)\Rightarrow (2)$: By assumption, $\mf$ is the maximal
%   defining ideal of the $\g$-module, hence $\sqrt{J}= \mf$ \Prop{\ref{reg-def-ideal}}, so there
%   exists an integer $r$ such that $\mf^r \subset J$.  This implies (2).  $(2)\Rightarrow (1)$: This
%   is obvious.  $(2)\Rightarrow(3)$: The ideal $I$ defines a germ of an analytic variety $V\subset
%   (\Cb^n,0)$.  (2) implies that $\sqrt{J}= \mf$, hence the sheaf of ideals on $V$ that is induced by
%   $J$, and again denoted by $J$, has the property $J_{x'}= \Oc_{\Cb^n,x'}$ when $x'\neq 0$ is close
%   to $0$.  By the Jacobian criterion of regularity, it follows that $\Oc_{\Cb^n,x'}/I_{x'}$ is a
%   regular local ring.

%   $(3)\Rightarrow (2)$: The locus $V(J)$ defines the singular locus of
%   the variety germ $V(I)\subset (\Cb^n, 0)$.  By assumption $V(J)=0$, hence $\sqrt{J}=
%   \mf$ by Hilbert's nullstellensatz for complex analytic spaces. This
%   implies (2).
% \end{proof}
It is a basic problem to understand the structure of $T_B$-modules $M$ such as $B$ or $\bar J= B J$,
and for this purpose the Hilbert series $H_{M}(t)$ captures essential information.  In the light of
\Proposition{\ref{solv-fibre}} it is important to know when the fibre Lie algebra $\g_{\Cb}$ is
solvable.
\begin{proposition}[\cite{granger-schulze:formalstructure}*{Prop. 6.2}]\label{granger-schulze} Put $I=(f)$, $f\in \Oc_n$, and
  $\g = T_{\Oc_n}(I)$. If $\g$ is free over $\Oc_n$ and $n \leq 3$,
  then $\g_{\Cb}$ is solvable.
\end{proposition}
\begin{theorem}\label{isol-solv} Let $I\subset \mf^2$ be an ideal generated by a
  regular sequence, and assume that $\Oc_n/I$ has an isolated
  singularity.  If $I$ is a principal ideal, assume moreover that
  $I\subset \mf^3$.  If $\g = T_{\Oc_n}(I)$, then $\g_{\Cb} $ is
  solvable.
\end{theorem}
\begin{remark}
  \Theorem{\ref{isol-solv}} was proven in
  \cite[Proposition 1.4]{granger-schulze:initial-lieb} when $I$ is a principal ideal,
using a different method.
\end{remark}

\begin{proof} We have an injective homomorphism $
  T_{\Oc_n}(I)/IT_{\Oc_n} \to T_{\Oc_n}/IT_{\Oc_n} $, so we can
  identify $T_{\Oc_n}(I)/IT_{\Oc_n} $ with a submodule of $
  T_{\Oc_n}/IT_{\Oc_n}$. By the Artin-Rees lemma there exists a
  positive integer $c$ such that for every integer $r> c+1$, we have
  \begin{displaymath}
    (\mf^r \frac {T_{\Oc_n}}{IT_{\Oc_n}}) \cap \frac{T_{\Oc_n}(I)}{IT_{\Oc_n} }=
    \mf^{r-c}((\mf^c \frac{ T_{\Oc_n}}{IT_{\Oc_n}}) \cap
\frac{    T_{\Oc_n}(I)}{IT_{\Oc_n} })\subset \mf^2 \frac{    T_{\Oc_n}(I)}{IT_{\Oc_n} },
\end{displaymath}
and therefore $ (I+ \mf^r )T_{\Oc_n} \cap T_{\Oc_n}(I)\subset \mf^2
T_{\Oc_n}(I) + IT_{\Oc_n}$. Select $r$ so that this inclusion holds.
By the initial assumption that $T_{\Oc_n}(I)\subset T_{\Oc_n}(\mf)$ we have $T_{\Oc_n}(I)\subset
T_{\Oc_n}(I+ \mf^r)$, and therefore a canonical
homomorphism of Lie algebroids
\begin{displaymath}
  T_{\Oc_n}(I)\to  T_{\Oc_n/(I+\mf^{r})}=\frac {T_{\Oc_n}(I+\mf^r)}{(I+\mf^r) T_{\Oc_n}}.
\end{displaymath}
The assumptions on $I$ imply by
\cite[Korollar 3.2]{scheja-wiebe:derivationen-isolierten} that the
finite-dimensional Lie algebra $T_{\Oc_n/(I+\mf^{r})}$ is solvable,
hence there exists a positive integer $l$ such that
  \begin{displaymath}
    \DC_l(T_{\Oc_n}(I))\subset (I + \mf^{r})T_{\Oc_n}\cap
    T_{\Oc_n}(I) \subset \mf^2 T_{\Oc_n}(I) + IT_{\Oc_n}.
\end{displaymath}
Since $[\mf^2 T_{\Oc_n}(I) + IT_{\Oc_n}, \mf^2 T_{\Oc_n}(I) +
IT_{\Oc_n}]\subset \mf T_{\Oc_n}(I)$ it follows that
$\DC_{l+1}(T_{\Oc_n}(I)) $ $ \subset \mf T_{\Oc_n}(I)$, and therefore
$\DC_{l+1}(\g_{\Cb})=0$.
\end{proof}
\begin{theorem}\label{yau-solv} Let $f\in \mf^3$, put $J= T_{\Oc_n}\cdot f + \Oc_n f$,
  and $\g = T_{\Oc_n}(J)$. If $\Oc_n/(f)$ has an isolated singularity,
  then $\g_{\Cb}$ is solvable.
\end{theorem}

\begin{proof} 
  Let $N_s$ be the $\Oc_n$-submodule of $T_{\Oc_n}$ that
  $\DC_s(T_{\Oc_n}(J))$ generates, $s=0,1,\dots$.  By the Artin-Rees
  lemma there exists a positive integer $c=c(s)$ such that
  \begin{displaymath}
    N_s \cap \mf^lT_{\Oc_n} = \mf^{l-c} (\mf^c T_{\Oc_n} \cap
    N_s),
  \end{displaymath}
  when $l > c$, implying $N_s\cap \mf^l T_{\Oc_n}\subset \mf
  T_{\Oc_n}(J)$ for such $l$.  By \cite{schulze:solvable},
  $T_{\Oc_n}(J)/JT_{\Oc_n}$ is solvable, hence $\DC_s(T_{\Oc_n}(J))
  \subset JT_{\Oc_n} \subset \mf^2 T_{\Oc_n}$, for sufficiently high
  $s$ and since $f\in \mf^3$.  Therefore there exists an integer
  $s'>s$ such that $\DC_{s'}(T_{\Oc_n}(J))\subset \mf^lT_{\Oc_n}$,
  implying
  \begin{displaymath}
    \DC_{s'}(T_{\Oc_n}(J)) \subset \DC_{s}(T_{\Oc_n}(J))\cap
    \mf^lT_{\Oc_n} \subset     N_s \cap \mf^lT_{\Oc_n} \subset \mf T_{\Oc_n}(J).
  \end{displaymath}
Hence $\DC_{s'}(\g_{\Cb}) =0$.
\end{proof}
\begin{remark}
  The argument in \cite{yau:solvability} for the solvability of
  $T_{\Oc_n}(J)/JT_{\Oc_n}$ is incomplete, as noted in
  \cite{schulze:solvable}. The proof in \cite{schulze:solvable} on the
  other hand does depend on ideas from \cite{yau:solvability}.
\end{remark}
\begin{theorem}\label{yau-cor} For $f\in \mf$ we put $I_f= (f)$, $J_f=
  (f) +
  T_{\Oc_n}\cdot f$, and we let $\g$ be one of the following
Lie algebroids:
  \begin{enumerate}
  \item $\g = T_{\Oc_n}(I_f)$,
  \item $\g = T_{\Oc_n}(J_f)$.
  \end{enumerate}
Then 
  \begin{displaymath}
    H_{G_{J_f}(\Oc_n)}(t)=    \sum_{i\geq 0}  \ell_{\g}    (\frac{J_f^i}{J_f^{i+1}}) t^i
  \end{displaymath}
  is a rational function.  Let $ g\in \mf$ be a function whose modular
  algebra is isomorphic to that of $f$, $\Oc_n/J_f \cong \Oc_n/J_g$.
  Then
  \begin{displaymath}
    H_{G_{J_f}(\Oc_n)}(t)=  H_{G_{J_g}(\Oc_n)}(t).
  \end{displaymath}

\end{theorem}
Clearly, in the situations covered by \Proposition{\ref{granger-schulze}} and
\Theorems{{\ref{isol-solv}}}{\ref{yau-solv}}, we have $\ell_{\g}(\frac{J_f^i}{J_f^{i+1}})= \dim_\Cb
\frac{J_f^i}{\mf J_f^{i+1}}$.
\begin{proof} In the case (1) we have invoked Rossi's theorem from the
  onset to reduce to the case when $\mf$ is the maximal defining
  ideal, so $\frac{J_f^i}{J_f^{i+1}}$ is a local system $\g$-module
  for each $i$. In (2) $T_{\Oc_n}(J_f)$ need not preserve $\mf$, but
  again by Rossi's theorem we can reduce to this situation, implying
  that $\frac{J_f^i}{J_f^{i+1}}$ is a local system for each $i$. Then
  \Theorem{\ref{main}} implies that $H_{G_{J_f}(\Oc_n)}(t)$ is a
  rational function.  According to Mather and Yau \cite{mather-yau}
  (the case of isolated singularities) and the general case by Greuel,
  Lossen and Shusten \cite{greuel-lossen-shusten}*{Th. 2.26}, if
  $\Oc_n/J_f \cong \Oc_n/J_g$, then there exists an isomorphism $\Oc_n
  \to \Oc_n$ mapping $I_f$ onto $I_g$. Therefore $J_f$ is mapped onto
  $J_g$, which implies the equality.
\end{proof}
\begin{question}
  What is the dimension of the stratum of constant Hilbert series
  $H^\g_{G_{J_f}(\Oc_n)} (t)$ in a semiversal deformation of a
  hypersurface with an isolated singularity?
\end{question}
% \begin{theorem} Let $\Lc$ be a finite dimensional Lie algebra over the
%   field $k$.  There exists a tangential Lie algebroid $\g_{\Oc_n}=T_{\Oc_n}(I)$,
%   where $I\subset \Oc_n$ is a principal ideal, such that $\g_k = \Lc$. 
% \end{theorem}
% \begin{proof}
%   Unfinished*** Apply Saito's criterion to $\Oc_n$-submodules
%   $\g_{\Oc_n}\subset T_{\Oc_n}$ that are generated by $k$-Lie
%   subalgebra of $T_{\Oc_n}$ isomorphic to $\Lc$. We need that $\dim
%   \Lc = n$, and that the natural map $T_{\Oc_n}\to \g_{\Oc_n}\subset
%   T_{\Oc_n}$ has a reduced determinant.  Construction: Take a
%   sufficiently big faithful $\Lc$-module $W$ and an integral variety
%   $X$ for the Lie algebroid $\g_{\Pb(W)}$ generated by $\Lc \to
%   T_{\Pb(W)}$.  The restriction of$\g_{\Pb(W)}$ to $X$ should be
%   generically free with basis $\Lc $. Then apply Saito's criterion. If the
%   determinant is not reduced, make a blowup. ***
% \end{proof}

\begin{example}\label{gl-example}
  Assume that $I=(x_1, \dots , x_r)\subset \Oc_n $, where $x_1, \dots,
  x_n$ is a regular system of parameters. Then $T_{\Oc_n}(I) =
  \sum_{i=r+1}^n \Oc_n \partial_{x_i} + \sum_{i,j=1}^r \Oc_n
  x_i\partial_{x_j}$ and the fibre Lie algebra
  \begin{displaymath}
    \g_\Cb = \sum_{i=r+1}^n \Cb\partial_{x_i} +    \sum_{i,j=1}^r
   \Cb x_i\partial_{x_j} =\sum_{i=r+1}^n \Cb \partial_{x_i} + \gl_{r}
  \end{displaymath}

\end{example}
\begin{example}
  Put $I = (\sum_{i=1}^n x_i^2)\subset \Oc_n$, $\partial_{ij} =
  x_i\partial_{x_j} - x_j \partial_{x_i}$, and $E= \sum_{i=1}^n
  x_i \partial_{x_i}$.  Then
  \begin{displaymath}
    T_{\Oc_n}(I) = \sum_{i< j} \Oc_n\partial_{ij}  + \Oc_n E
\end{displaymath}
and the fibre Lie algebra $\g_\Cb = \Cb \oplus \of_{n}$, where
$\of_{n}$ is the orthogonal Lie algebra.  
\end{example}
\begin{example} The Whitney umbrella is defined by the principal ideal
  $I = (z^2 - x^2y)\subset \Oc_3$.  We have $T_{\Oc_3}(I)=
  \sum_{i=1}^4\Oc_3\delta_i$, where
 \begin{align*}
 \delta_1 &= \nabla_x - 2 \nabla_y ,  &\delta_2 &= \nabla_x+ \nabla_z, \\
 \delta_3 &= 2z\partial_y+x^2\partial_z ,  &\delta_4& = z\partial_x+xy\partial_z. 
\end{align*}
The singularity is not isolated and $T_{\Oc_3}(I)$ is not free.  The
fibre Lie algebra $\g_\Cb = \sum \Cb \bar \delta_i$, where $[\bar
\delta_1, \bar \delta_2] = [\bar \delta_2,\bar \delta_4] = [\bar
\delta_3, \bar \delta_4] = 0$ (note that $[\delta_3, \delta_4]= x
\delta_1$), $[\bar \delta_1,\bar \delta_3] = 2[\bar \delta_2,\bar
\delta_3] = 2\bar \delta_3$, and $[\bar \delta_1,\bar \delta_4] =
-\bar \delta_4$.  Therefore $\g_\Cb$ is solvable.
\end{example}
% \begin{proof}
%   Let $\Lc \subset \gl(V)$ 
% \end{proof}
% Let $\Lc$ be a semi-simple Lie algebra. Can one find a tangential Lie
% algebroid $\g_A= T_A(I)$ such that $\g_k $ has a Levi factor
% isomorphic to $\Lc$?  If $\g_A$ is free over $A$ and $I$ is a
% principal ideal, the condition is that $\Lc$ should be isomorphic to a
% subalgebra of $T_{A}(I)$ that can be complemented to a basis of
% $T_A(I)$ such that the determinant of the map that takes a basis of
% $T_A$.  Let $\Bc$ be a basis of $T_A$ and $\Bc(I)$ a basis of $T_A(I)$

%  Let $\g_k$ be a Lie algebra over $k$ and $V$ be
% a faithful $\g$-module, i.e. the map $\g \to \gl(V)$ is
% injective. Put $R= S(V)$, and assume that $\g \to T_R$ is generically
% transitive.
\begin{example} This example is borrowed from
  \cite{granger-schulze:formalstructure}, where they start with a
  $4$-dimensional representation of the Lie algebra $\gl_2(\Cb)$ and
  then apply Saito's criterion for free divisors to get the ideal $I =
  (y^2x^2 - 4xz^3-4y^3w+18xyzw-27w^2z^2) \subset \Oc_4$.  Then
  $T_{\Oc_4}(I)= \oplus_{i=1}^4 \Oc_4\delta_i $, where
 \begin{align}
& \delta_1 = \nabla_x + \nabla_y+ \nabla_z+ \nabla_w \nonumber \\
& \delta_2 = -3\nabla_x-\nabla_y+ \nabla_z+3\nabla_w \nonumber\\
& \delta_3 = y\partial_x+2z\partial_y+3w\partial_z \nonumber\\
& \delta_4 = 3x\partial_y+2y\partial_z+z\partial_w. \nonumber
\end{align}
The ideal $I$ is also the ideal of the discriminant of a cubic
polynomial.  Here $n> 3$ and the singularity is not isolated, so we
are not in the cases covered by \Proposition{\ref{granger-schulze}}
and \Theorem{\ref{isol-solv}}.  The fibre Lie algebra is $
T_{\Oc_4}(I)/\mf T_{\Oc_4}(I) = \gl_2 (\Cb)$, so a Levi factor is $\Lc
= \Sl_2(\Cb)$.  The Hilbert series of the $T_{\Oc_4}(I)$-module
$\Oc_4$ is therefore the generating function for the numbers
$\ell_{\Sl_2(\Cb)}(S^n(\mf/\mf^2))$. Since $ \frak m/\frak m^2$ is
$\Sl_2(\Cb)$-simple of dimension $4$, it can be identified with the
symmetric product $V= S^3(\Cb^2)$.  The Hilbert series $H_V(t)$ of the
$( S^\bullet (V),\Sl_2(\Cb))$-module $S^\bullet (V)$ is determined in
\Example{\ref{sl2-example}}. We get in particular the dimension
$d_{T_{\Oc_4}(I)}(\Oc_4) =2$ and multiplicity $e_{T_{\Oc_4}(I)}(\Oc_4,
\mf) = 1/4$.
\end{example}
\begin{question}
  Which finite-dimensional complex Lie algebras can be constructed as
  a fibre Lie algebra of a hypersurface singularity?
\end{question}

%--------------------------------------------------------

%%% bib section

\end{document}